\documentclass[11pt,oneside,dvipsnames]{amsart}
\usepackage{stackrel}
\usepackage[french, english]{babel}
\usepackage{a4wide}
\usepackage[utf8,latin1]{inputenc}
\usepackage{times}
\usepackage{adjustbox}
\usepackage{float}
\usepackage[cyr]{aeguill}
\usepackage{amsmath,amscd}
\usepackage{amssymb}
\usepackage{mathrsfs}
\usepackage{amsthm}
\usepackage{geometry}
\usepackage{graphicx}
\usepackage{epstopdf}
\usepackage{amsfonts}
\usepackage{amsmath}
\usepackage{amsmath,mathtools}
\usepackage{amsmath,graphicx}
\usepackage{amscd}
\usepackage{dsfont}
\usepackage{latexsym}
\usepackage{verbatim}
\usepackage{comment}
\usepackage{tikz-cd}
\usetikzlibrary{matrix,arrows}
\usepackage{tikz}
\usetikzlibrary{tikzmark}
\usepackage{tensor}
\usepackage{leftidx}
\usepackage[retainorgcmds]{IEEEtrantools}
\usepackage[title]{appendix}
\usepackage[mathscr]{euscript}
\usepackage[colorinlistoftodos]{todonotes}
\usepackage{bbm}
\usepackage{url}
\usepackage{upgreek}
\usepackage{mathtools, nccmath}
\usepackage[utf8]{inputenc}
\usepackage{multirow}
\usepackage{rotating}
\usepackage{marginnote}
\usepackage{hyperref} 
\usepackage{nameref}
\usepackage{ytableau}
\usepackage{youngtab}

\usepackage[all]{xy}
\usepackage[T1]{fontenc}

\usepackage{hyperref}
\usepackage{array}
\usepackage{tabularx}
\usepackage{yfonts}
\usepackage{setspace}  
 \usepackage[foot]{amsaddr}
\makeatletter
\renewcommand\subsubsection{\@startsection{subsubsection}{3}{\z@}%
                                     {-2ex\@plus -1ex \@minus -.2ex}%
                                     {-0.5em}
                                     {\normalfont\normalsize\bfseries}}
                                     
                                     \renewcommand\subsection{\@startsection{subsection}{3}{\z@}%
                                     {-3.25ex\@plus -1ex \@minus -.2ex}%
                                     {-0.5em}
                                     {\normalfont\normalsize\bfseries}}

\makeatother

\theoremstyle{plain}
\newtheorem{Thm}{Theorem}
\newtheorem{ThmA}{Theorem}

\newtheorem{Prop}[Thm]{Proposition}
\newtheorem{Lem}[Thm]{Lemma}
\newtheorem{Cor}[Thm]{Corollary}

\theoremstyle{definition}

\newtheorem{Rem}[Thm]{Remark}

\newtheoremstyle{named}{}{}{}{}{\bfseries}{.}{.5em}{\thmnote{#3}#1}
\theoremstyle{named}

\DeclareMathOperator{\Aut}{Aut}
\DeclareMathOperator{\Bun}{Bun}

\DeclareMathOperator{\GL}{GL}

\DeclareMathOperator{\Hom}{Hom}
\DeclareMathOperator{\SHom}{\mathscr{H}\textit{om}}

\DeclareMathOperator{\Id}{Id}
\DeclareMathOperator{\Ima}{Im}

\DeclareMathOperator{\isom}{\!\!\smash{\begin{array}{c}\sim\\[-1em]
\rightarrow\end{array}}\!\!}
\DeclareMathOperator{\lisom}{\!\!\smash{\begin{array}{c}\sim\\[-1em]
\longrightarrow\end{array}}\!\!}

\DeclareMathOperator{\Ind}{Ind}
\DeclareMathOperator{\Ker}{Ker}

\DeclareMathOperator{\Nm}{Nm}

\DeclareMathOperator{\Pic}{Pic}
\DeclareMathOperator{\PGL}{PGL}
\DeclareMathOperator{\pr}{pr}

\DeclareMathOperator{\red}{red}
\DeclareMathOperator{\reg}{reg}

\DeclareMathOperator{\Res}{Res}

\DeclareMathOperator{\Sch}{Sch}

\DeclareMathOperator{\SL}{SL}

\DeclareMathOperator{\Spec}{Spec}

\DeclareMathOperator{\Stab}{Stab}
\DeclareMathOperator{\Supp}{Supp}

\DeclareMathOperator{\Tr}{Tr}

\newcommand*\circled[1]{
  \tikz[baseline=(char.base)]\node[shape=circle,draw,inner sep=0.2pt,font=\tiny,minimum size=8pt] (char) {#1};}

\makeatletter
\newcommand*{\rom}[1]{\expandafter\@slowromancap\romannumeral #1@}
\makeatother

\renewcommand{\descriptionlabel}[1]{\hspace\labelsep\upshape\bfseries #1.}
\makeatletter
\let\orgdescriptionlabel\descriptionlabel
\renewcommand*{\descriptionlabel}[1]{%
  \let\orglabel\label
  \let\label\@gobble
  \phantomsection
  \edef\@currentlabel{#1}%
  \let\label\orglabel
  \orgdescriptionlabel{#1}%
}
\makeatother

\usepackage{subfig}	 

\title{Mirror of Orbifold Singularities in the Hitchin Fibration: the case $(\SL_n,\PGL_n)$}
\author{Yongbin Ruan}
\email[Y. Ruan]{ruanyb@zju.edu.cn}
\author{Cheng Shu}
\email[C. Shu]{shuc@zju.edu.cn}
\address[Y. Ruan]{Institute for Advanced Study in Mathematics, Zhejiang University, Hangzhou, China}
\address[C. Shu]{Institute for Theoretical Sciences, Westlake University, Hangzhou, China}

\usepackage{pxfonts}
\usepackage{indentfirst}
\usepackage{setspace}  
\usepackage{hyperref}
\colorlet{ivory}{Apricot!30!}
\colorlet{space}{black!85!}
\definecolor{bgc}{RGB}{29, 44, 46}
\definecolor{txt}{RGB}{223, 222, 189}
\definecolor{cmd}{RGB}{206, 151, 88}
\begin{document}
\let\bs\boldsymbol
\begin{abstract}
We study the geometry of singular $\SL_n$-Hitchin fibres over the elliptic locus. We show that orbifold singularities appear in the $\PGL_n$-moduli space $\mathcal{M}^{ell}_C(\PGL_n)$ exactly when the $\SL_n$ side $\mathcal{M}^{ell}_C(\SL_n)$ has a reducible Hitchin fibre. Our main theorem shows that the Fourier-Mukai transform of a skyscraper sheaf supported at an orbifold singularity in $\mathcal{M}^{ell}_C(\PGL_n)$ satisfies a version of the fractional Hecke eigenproperty, as conjectured by Frenkel and Witten.
\end{abstract}
\maketitle

\tableofcontents
\addtocontents{toc}{\protect\setcounter{tocdepth}{-1}}
\setcounter{tocdepth}{1}
\numberwithin{Thm}{subsection}
\numberwithin{equation}{section}
\addtocontents{toc}{\protect\setcounter{tocdepth}{1}}
\section{Introduction}\label{INT}
\subsection*{Duality of Hitchin Fibrations}\hfill

Let $G$ be a reductive group over $\mathbb{C}$
and let $C$ be a smooth projective curve of genus $g>1$. We denote by $\check{M}$ the moduli space of semi-stable $G$-Higgs bundles on $C$. The Hitchin map is a proper morphism $\check{h}:\check{M}\rightarrow\check{A}$ with the target being a certain affine space called the Hitchin base. The Hitchin base parametrises cameral covers over $C$. If $\hat{G}$ is the Langlands dual group of $G$, then we have a similarly defined morphism $\hat{h}:\hat{M}\rightarrow\hat{A}$. According to \cite[Theorem A]{DP}, the two Hitchin bases $\check{A}$ and $\hat{A}$ are isomorphic, and so we will write $\hat{A}$ as $\check{A}$ if no confusion arises. The classical limit geometric langlands conjecture of Donagi and Pantev predicts the existence of an equivalence of categories (\cite[Conjecture 2.5]{DP})
\begin{equation}\label{eq-intro-DP2.5}
\mathfrak{cl}:\mathcal{D}(\hat{M})\lisom\mathcal{D}(\check{M})
\end{equation}
where $\mathcal{D}(-)$ denotes the bounded derived category of coherent sheaves on the indicated space. This equivalence should be compatible with two families of operators acting on each side. Such an equivalence has been constructed by Donagi and Pantev over the open subset $\check{A}^{\diamondsuit}\subset\check{A}$ where the corresponding cameral covers are smooth. That is, a derived equivalence between $\check{M}^{\diamondsuit}$ and $\hat{M}^{\diamondsuit}$, the inverse images of $\check{A}^{\diamondsuit}$ under the respective Hitchin maps.
Indeed, they have shown that for any $a\in\check{A}^{\diamondsuit}$, the neutral connected components of the fibres of $\check{M}^{\diamondsuit}$ and $\hat{M}^{\diamondsuit}$ respectively, are dual abelian varieties; that is
\begin{equation}\label{eq-intro-dual-Ab}
\Pic^0(\check{h}^{-1}(a)^{\circ})\cong \hat{h}^{-1}(a)^{\circ},
\end{equation} 
where we have used the superscript $\circ$ to indicate the neutral component. Then the equivalence (\ref{eq-intro-DP2.5}) over $\check{A}^{\diamondsuit}$ is the Fourier-Mukai transform defined by a Poincar\'e line bundle on $\check{M}^{\diamondsuit}\times_{\check{A}^{\diamondsuit}}\hat{M}^{\diamondsuit}$. Fibrewise, this equivalence is reduced to the derived equivalence of dual abelian varieties due to Mukai \cite{Muk}. 

There are only partial results outside $\check{A}^{\diamondsuit}$. For $G=\GL_n$ and over the open subset of integral spectral covers $\check{A}^{ell}\subset\check{A}$, Arinkin \cite{Ari2} has proved a fibrewise equivalence. This has been generalised to the larger open subset $\check{A}^{\heartsuit}\subset\check{A}$ of reduced spectral covers by Melo-Rapagnetta-Viviani \cite{MRV2}, with an additional assumption on the polarisation on the spectral curve. For the dual groups $(\SL_n,\PGL_n)$, a fibrewise equivalence over $\check{A}^{ell}$ has been proved by Franco-Hanson-Ruano \cite{FHR}. A variant taking into account of certain gerbes on the moduli spaces is also proved by Groechenig-Shen \cite{GS}. For general reductive groups, an integral functor is constructed in \cite[Theorem 1.2.2]{Li} which is only known to be fully faithful.

The goal of this article however, is to study the geometry of singular Hitchin fibres $\check{h}^{-1}(a)$ for $a\in\check{A}^{ell}$ and $G=\SL_n$, and understand how the duality (\ref{eq-intro-dual-Ab}) behaves in the singular case. The duality (\ref{eq-intro-dual-Ab}) is a key requirement in the SYZ definition of mirror partners. See for example \cite{HT}. Therefore, our work is an attempt to understand the mirror symmetry for singular fibres. The mirror of a smooth point is known to be a line bundle on the dual Hitchin fibre, and by mirror we mean the Fourier-Mukai transform defined by a certain Poincar\'e sheaf on $\check{M}^{ell}\times_{\check{A}^{ell}}\hat{M}^{ell}$. It turns out that in $\hat{M}^{ell}$, there are at worst orbifold singularities, and we can indeed give a nice description of the mirror.  Some results in this direction include \cite{FGOP} for $\GL_n$ and \cite{Ho} for $\SL_2$.

\subsection*{Singular Hitchin Fibres}\hfill

For $G=\GL_n$ and $a\in\check{A}^{ell}$, the Hitchin fibre $\check{h}^{-1}(a)=\hat{h}^{-1}(a)$ is isomorphic to the compactified Jacobian of the corresponding spectral curve. The generalisation of (\ref{eq-intro-dual-Ab}) in this case is \cite[Theorem B]{Ari2}, which says that $\check{h}^{-1}(a)$ parametrises torsion-free sheaves of rank one on the dual Hitchin fibre, and an open dense subset $\check{h}^{-1}(a)^{\reg}$ parametrises line bundles. It is conceivable that a similar statement should hold for $G=\SL_n$, if the fibre $\check{h}^{-1}(a)$ is irreducible. However, the new phenomenon for $\SL_n$ is that $\check{h}^{-1}(a)$ might not be irreducible (but always connected). This is closely related to the endoscopy theory for $\SL_n$. See \cite{HP} and \cite{dC}.

In the rest of this article, we will only consider the case $G=\SL_n$ and $\hat{G}=\PGL_n$. We will write $\Gamma=\Pic^0(C)[n]$, i.e. the $n$-torsion subgroup of the Jacobian variety of $C$. We also fix a line bundle $D$ on $C$ which is either isomorphic to the canonical bundle $\Omega_C$ or has degree $\deg D>2g-2$. Of course, only the case $D=\Omega_C$ is relevant to mirror symmetry, but our results hold in this greater generality. Now $\check{M}$ denotes the moduli space of semi-stable $D$-twisted Higgs bundles with trivial determinant and traceless Higgs field, and 
$$
\check{A}=\bigoplus_{i=2}^{n}H^0(C,D^{\otimes i})
$$
is the Hitchin base parametrising spectral covers of $C$ of degree $n$. The open subset of $\check{A}$ corresponding to integral spectral covers is denoted by $\check{A}^{ell}$. Then $\check{M}^{ell}=\check{h}^{-1}(\check{A}^{ell})$ is an open subset of the locus $\check{M}^s\subset\check{M}$ of stable Higgs bundles. In particular, it is smooth. The neutral connected component of the moduli space of semi-stable $D$-twisted $\PGL_n$-Higgs bundles is the quotient space $\hat{M}:=\check{M}/\Gamma$. Its Hitchin base $\hat{A}$ is identified with $\check{A}$. Then $\hat{M}^{ell}=\check{M}^{ell}/\Gamma$ only has orbifold singularities. \footnote{We could equally consider the moduli space of Higgs bundles with determinant isomorphic to a nontrivial line bundle. But there will be no essential difference, since we are only interested in a single Hitchin fibre.}

The presence of orbifold singularities in $\hat{M}^{ell}$ is exactly the $\PGL_n$-counterpart of the fact that $\SL_n$-Hitchin fibres might be reducible. According to \cite[Theorem 1.1]{HP} and \cite[Fact 2.5.1]{dC}, for any $a\in\check{A}^{ell}$, the irreducible components of $\check{h}^{-1}(a)$ are parametrised by $K_a^{\vee}$, the group of multiplicative characters of a certain subgroup $K_a\subset\Gamma$. In \cite{FW}, Frenkel and Witten conjecture that if $x$ is an element of the dual fibre $\hat{h}^{-1}(a)$, then the automorphism group $\Aut(x)$ is a subgroup of $K_a$, and that there exists some $x\in\hat{h}^{-1}(a)$ such that $\Aut(x)=K_a$. Our first theorem confirms this conjecture. 

Now the category of coherent sheaves supported at an orbifold singularity is generated by the irreducible representations $H_x^{\vee}$ of the isotropy group $H_x=\Aut(x)$ of $x$. If $\xi\in H_x^{\vee}$, we denote by $\mathbf{A}_x(\xi)$ the mirror of $\xi$, which is a torsion-free sheaf on $\check{h}^{-1}(a)$. If $R$ is the regular representation of $H_x$, then we simply write $\mathbf{A}_x=\mathbf{A}_x(R)$, and we have $\mathbf{A}_x=\bigoplus_{\xi\in H_x^{\vee}}\mathbf{A}_x(\xi)$. It is easy to show that $\mathbf{A}_x$ is the eigensheaf for a family of translation operators acting on the Hitchin fibre. Frenkel and Witten predict that each simple component $\mathbf{A}_x(\xi)$ is what they call a fractional Hecke eigensheaf. The main theorem of this article shows that this prediction holds in a suitable sense. See Theorem \ref{Thm-D} below. Note that Frenkel and Witten work with $\mathscr{D}$-modules or A-branes instead of coherent sheaves, therefore our results are in fact a classical limit of their predictions. 

\subsection*{Hecke Eigenproperty}\hfill

The (fractional) Hecke eigenproperty that we will prove is not what is literally stated in \cite{FW}. Let us recall what the Hecke eigenproperty means, and explain how it is related to the eigenproperty that we will prove. We fix a maximal torus $\hat{T}\subset\PGL_n$. For any dominant character $\lambda\in X(\hat{T})_+$, we denote by $V^{\lambda}$ the irreducible representation of $\PGL_n$ corresponding to $\lambda$. The universal $\PGL_n$-Higgs bundle on $\hat{M}\times C$ is denoted by $(\mathcal{E},\vartheta)$.

There are two families of operators acting on each side of the equivalence (\ref{eq-intro-DP2.5}). On the $\PGL_n$-side we have the \textit{classical limit tensorisation functors} (\cite[\S 2]{DP})
\begingroup
\allowdisplaybreaks
\begin{align*}
\mathbf{W}_{\lambda,p}:\mathcal{D}(\hat{M})&\longrightarrow\mathcal{D}(\hat{M})\\
K&\longmapsto K\otimes (V^{\lambda}_{\mathcal{E}})|_{\hat{M}\times\{p\}}
\end{align*}
\endgroup
defined for any $p\in C$ and $\lambda\in X(\hat{T})_+$, where $V^{\lambda}_{\mathcal{E}}$ is the vector bundle obtained from $\mathcal{E}$ by extending the structure group using $V^{\lambda}$. A complex $K\in\mathcal{D}(\hat{M})$ is called an eigensheaf of $\mathbf{W}_{\lambda,p}$ if there is a quasi-isomorphism
$$
\mathbf{W}_{\lambda,p}(K)\cong K\otimes E
$$
for some $\mathbb{C}$-vector space $E$, called the multiplier. For any smooth point $x\in\hat{M}$, the skyscraper sheaf $\underline{\mathbb{C}}_x$ supported at $x$ is an eigensheaf for $\mathbf{W}_{\lambda,p}$. Indeed, we have
\begin{equation}\label{eq-intro-W-eigen}
\mathbf{W}_{\lambda,p}(\underline{\mathbb{C}}_x)\cong \underline{\mathbb{C}}_x\otimes V_{\mathcal{E}}^{\lambda}|_{\{x\}\times\{p\}}.
\end{equation}

On the $\SL_n$-side, we have the \textit{classical limit Hecke functors} (\cite[\S 2]{DP})
$$
\mathbf{H}_{\lambda,p}:\mathcal{D}(\check{M})\longrightarrow\mathcal{D}(\check{M}),
$$
which are the integral functors defined by certain correspondences called Hecke stacks. There are two approaches to defining these functors. One is proposed by Donagi-Pantev in \cite[\S 2]{DP}. By identifying $\mathcal{D}(\check{M})$ as a subcategory of the derived category of Higgs sheaves on $\Bun_{\SL_n}$, we can use the usual Hecke stacks (see \cite[Eq. (2)]{DP}) to define the Hecke functors, though the classical limits of Hecke kernels are more subtle \cite[Definition 2.4]{DP}. Another approach is directly considering the stack of pairs $((E,\theta),(E',\theta'))$ of $\SL_n$-Higgs bundles, where $(E',\theta')$ is a \textit{Hecke transform} of $(E,\theta)$ at a given point $p\in C$ of the curve with respect to a given $\theta$-invariant subspace of the fibre $E_p$ (see for example \cite[Definition 4.10]{HH}, which should agree with the \textit{$\varphi$-invariant Hecke modifications} in \cite[\S 5.3]{FW}). This gives a correspondence from $\check{M}$ to itself and so defines a Hecke functor. The equivalence (\ref{eq-intro-DP2.5}) is expected to be compatible with $\mathbf{W}_{\lambda,p}$ and $\mathbf{H}_{\lambda,p}$ for any $p$ and $\lambda$, and so the mirror $\mathbf{A}_x=\mathfrak{cl}(\underline{\mathbb{C}}_x)$ must be an eigensheaf for all $\mathbf{H}_{\lambda,p}$; that is,
$$
\mathbf{H}_{\lambda,p}(\mathbf{A}_x)\cong \mathbf{A}_x\otimes E
$$
with $E\cong V_{\mathcal{E}}^{\lambda}|_{\{x\}\times\{p\}}$. Moreover, since the operators $\mathbf{W}_{\lambda,p}$ restricts to well-defined operators on each Hitchin fibre, so do the operators $\mathbf{H}_{\lambda,p}$. Note that the mirror $\mathbf{A}_x$ is supported on a single Hitchin fibre.

\subsection*{Fractional Hecke Eigenproperty}\hfill

Suppose that $x$ is an orbifold singularity of $\hat{M}$ with isotropy group $H_x$. Then a skyscraper sheaf at $x$ can be equipped with an action of $H_x$. We will denote by $\underline{\mathbb{C}}_x(\xi)$ the rank one skyscraper sheaf supported at $x$ with $H_x$ acting on its stalk via the character $\xi\in H_x^{\vee}$. Now the multiplier $E=V_{\mathcal{E}}^{\lambda}|_{\{x\}\times\{p\}}$ is an $H_x$-representation, and let $E=\bigoplus_i E_i$ be its decomposition into irreducible representations $E_i$ of $H_x$. Each $E_i$ is isomorphic to a representation $\xi_i$ for some $\xi_i\in H_x^{\vee}$. There might be isomorphic representations among the $\xi_i$'s. Then (\ref{eq-intro-W-eigen}) becomes an isomorphism of $H_x$-representations
$$
\mathbf{W}_{\lambda,p}(\underline{\mathbb{C}}_x(\xi))\cong \bigoplus_i\underline{\mathbb{C}}_x(\xi)\otimes E_i\cong\bigoplus_i\underline{\mathbb{C}}_x(\xi\otimes\xi_i).
$$
Since the right hand side may not be the direct sum of copies of $\underline{\mathbb{C}}_x(\xi)$, this equation is not literally an eigenproperty. The \textit{fractional Hecke eigenproperty} of the mirror $\mathbf{A}_x(\xi)$ proposed by Frenkel and Witten is that there exists an isomorphism
\begin{equation}
\mathbf{H}_{\lambda,p}(\mathbf{A}_x(\xi))\cong\bigoplus_i\mathbf{A}_x(\xi\otimes \xi_i).
\end{equation}
In the situation of \cite[\S 5.3]{FW}, one expects a family of translation maps $\{\Phi_i\}_i$ on the Hitchin fibre $\check{h}^{-1}(a)$ such that for any $K\in\mathcal{D}(\check{h}^{-1}(a))$, we have
\begin{equation}\label{eq-intro-decomp-Hecke}
\mathbf{H}_{\lambda,p}(K)=\bigoplus_i\Phi_i^{\ast}K,
\end{equation}
and that for each $i$, we have $\Phi_i^{\ast}\mathbf{A}_x(\xi)\cong\mathbf{A}_x(\xi\otimes\xi_i)$. 

In this article, we will construct a family of translation operators $T_{p_1,p_2}$ acting on the Hitchin fibre $\check{h}^{-1}(a)$, and associate to each of them a character $\chi_{p_1,p_2}\in K_a^{\vee}$, where $K_a^{\vee}$ is the finite group parametrising the irreducible components of $\check{h}^{-1}(a)$. We will show that 
$$
T_{p_1,p_2}^{\ast}\mathbf{A}_x(\xi)\cong\mathbf{A}_x(\xi\otimes\Res^{K_a}_{H_x}\chi_{p_1,p_2})
$$
for any $\xi\in H_x^{\vee}$, where $\Res^{K_a}_{H_x}\chi_{p_1,p_2}$ denotes the restriction of the character $\chi_{p_1,p_2}$ to $H_x$. This is our fibrewise version of the fractional Hecke eigenproperty. See Theorem \ref{Thm-D} below. Of course, to verify the full fractional Hecke eigenproperty, one needs to show that the decomposition (\ref{eq-intro-decomp-Hecke}) exists and that it is compatible with the decomposition $E=\bigoplus_i E_i$ of the multiplier. We hope to address this problem in the future.

\subsection*{Main Results}\hfill

Our first theorem proves a conjecture of Frenkel and Witten \cite[\S 1.5]{FW} concerning the automorphism group of a $\PGL_n$-Higgs bundle lying over the elliptic locus of the Hitchin base. This is also an improvement upon the result of Ng\^o \cite[Corollaire 4.11.3]{Ngo} for $\PGL_n$. 
\begin{ThmA}[Theorem \ref{M-T}]\label{Thm-FW}
The following assertions hold:
\begin{itemize}
\item[(i)] For any $a\in \check{A}_n^{ell}(C)$, and any $x\in 
\hat{h}_C^{-1}(a)$, we have $\Aut(x)\subset K_a$.
\item[(ii)] For any $a\in \check{A}^{ell}_n(C)$, there exists $x\in 
\hat{h}_C^{-1}(a)$ such that $\Aut(x)=K_a$.
\end{itemize}
\end{ThmA}

Since we are only interested in fibrewise phenomenon, we will always work with a single $\SL_n$-Hitchin fibre and identify it with a compactified Prym variety. The rest of our results are proved for compactified Prym varieties.

Our second theorem is an analogue of Arinkin's theorem \cite[Theorem 1.3]{Ari1} for Prym varieties. It is well known that the Jacobian variety of a smooth projective curve is a self-dual abelian variety, i.e., there is an isomorphism $\Pic^0(J_C)\cong J_C$. For an integral projective curve $X$ with planar singularities, Arinkin has shown that there is an isomorphism $\Pic^0(\bar{J}_X)\cong J_X$, where $\bar{J}_X$ is the compactified Jacobian of $X$. In both cases, such an isomorphism is induced by the Abel-Jacobi map. Moreover, the torsion components $\Pic^{\tau}(\bar{J}_X)$ (i.e., the subvariety consisting of line bundles $L$ on $\bar{J}_X$ satisfying $L^{\otimes n}\in\Pic^0(\bar{J}_X)$ for some $n>0$.) of the Picard variety of $\bar{J}_X$ is equal to $\Pic^0(\bar{J}_X)$ (see \cite[Proposition 6.2]{Ari1}). For the Prym variety $P$ associated to a smooth spectral curve $X\rightarrow C$, we know that $P$ and $P/\Gamma$ are dual abelian varieties. It is natural to ask what happens for the compactified Prym varieties $\bar{P}$, the inverse images of $0$ under a suitably defined norm map $\Nm:\bar{J}_X\rightarrow J_C$ when $X$ is not smooth. Now the action of $\Gamma$ on $\bar{P}$ is not necessarily free, therefore the dual side is the quotient stack $[\bar{P}/\Gamma]$. 
\begin{ThmA}[Theorem \ref{Thm-Pic'-repre-P}]\label{Thm-A}
The \'etale sheaf $\Pic^{\tau}_{[\bar{P}/\Gamma]}$ on $\Sch_{\mathbb{C}}$ is representable by $P$.
\end{ThmA}
In particular, if $P$ is not connected, then $\Pic^0_{[\bar{P}/\Gamma]}\neq\Pic^{\tau}_{[\bar{P}/\Gamma]}$. We do not know a priori whether $\Pic^{\tau}_{[\bar{P}/\Gamma]}$ is representable, so we directly work with sheaves. The above theorem is needed in the process of proving Theorem \ref{Thm-B} below.

If $x\in\hat{M}^{ell}$ is a smooth point, then the mirror $\mathbf{A}_x$ is a line bundle on $\bar{P}$. However, if $x$ is an orbifold singularity, $\mathbf{A}_x$ is no longer indecomposable. Our third theorem determines the support of each simple summand $\mathbf{A}_x(\xi)$. Here we see a finer interplay between the orbifold singularities on the $\PGL_n$ side and the reducibility of $\SL_n$-Hitchin fibres.
\begin{ThmA}[Theorem \ref{Thm-B-proof}]\label{Thm-B}
Let $a\in\check{A}^{ell}$ and let $\bar{P}$ be the compactified Prym variety associated to the corresponding spectral cover $X_a\rightarrow C$. Let $x\in[\bar{P}/\Gamma]$ be such that its isotropy group is equal to $H\subset K_a$. Then for any $\xi\in H^{\vee}$, we have
$$
\Supp\mathbf{A}_x(\xi)=\bigcup_{\chi\in\Ind^{K_a}_H\xi}\bar{P}_{\chi}^{irr},
$$
where $\Ind^{K_a}_H\xi$ denotes the character of $K_a$ induced from $\xi$, and $\bar{P}_{\chi}^{irr}$ is the irreducible component of $\bar{P}$ corresponding to $\chi$ (the notation $\chi\in\Ind^{K_a}_H\xi$ means that $\chi$ is an irreducible constituent of the induced character).
\end{ThmA} 

Our fourth theorem gives a new parametrisation of the connected components of Prym varieties. Let $G\subset\Gamma$ be a subgroup, and let $\psi_G:\tilde{C}\rightarrow C$ be the \'etale Galois covering determined by $G$. Its Galois group is isomorphic to $G^{\vee}$. Let $J^{-1}_{\tilde{C}}$ be the Jacobian variety of degree -1 line bundles on $\tilde{C}$. For any $\chi\in G^{\vee}$, regarded as a covering transformation of $\psi_G$, define the morphism
\begingroup
\allowdisplaybreaks
\begin{align*}
f_{\chi}:J^{-1}_{\tilde{C}}&\longrightarrow J_{\tilde{C}}\\
L&\longmapsto L^{-1}\otimes\chi^{\ast}L.
\end{align*}
\endgroup
\begin{ThmA}[Theorem \ref{Thm-Conn-Prym-New}]\label{Thm-C}
Let $P_{\psi_G}$ be the Prym variety associated to $\psi_G$. For any $\chi\in G^{\vee}$, the image of $f_{\chi}$ is contained in a unique connected component of $P_{\psi_G}$, and this defines an isomorphism of groups $G^{\vee}\cong\pi_0(P_{\psi_G})$.
\end{ThmA}
In fact, if we denote by $P_{\psi_G,\chi}$ the connected component of $P_{\psi_G}$ corresponding to $\chi$ under the identification in \cite[Lemma 2.1]{HP}, then $P_{\psi_G,\chi}$ contains the image of $f_{\chi}$.

For any smooth points $p_1$ and $p_2$ of $X_a$ lying in a fibre of the spectral cover $\pi:X_a\rightarrow C$, tensoring by $\mathcal{O}_{X_a}(p_1-p_2)$ defines a translation operator $T_{p_1,p_2}$ on the Prym variety. This is our fibrewise version of Hecke operator. Suppose that $\mathcal{O}_{X_a}(p_1-p_2)$ lies in the connected component $P_{\chi^{-1}}$. Then $T_{p_1,p_2}$ induces the multiplication by $\chi^{-1}$ on $\pi_0(P)\cong K_a^{\vee}$. Theorem \ref{Thm-C} is used to determine $\chi$ from $p_1$ and $p_2$.

The following theorem is the main result of this article. It shows that the simple torsion-free sheaves $\mathbf{A}_x(\xi)$ satisfy a fibrewise version of the fractional Hecke eigenproperty proposed by Frenkel and Witten in \cite{FW}.
\begin{ThmA}[Theorem \ref{Thm-FHEP}]\label{Thm-D}
Let $a\in\check{A}^{ell}$, $p_1$, $p_2$ and $\chi$ be as above. Then for any subgroup $H\subset K_a$ and any $x\in[\bar{P}/\Gamma]$ such that $\Aut(x)=H$, we have the following isomorphisms of torsion-free sheaves on $\bar{P}$.
\begin{itemize}
\item[(i)] Hecke eigenproperty.
$$
T_{p_1,p_2}^{\ast}\mathbf{A}_x\cong\mathbf{A}_x.
$$
\item[(ii)] Fractional Hecke Eigenproperty.
$$
T^{\ast}_{p_1,p_2}\mathbf{A}_x(\xi)\cong\mathbf{A}_x(\xi\otimes\Res^{K_a}_H\chi)
$$
for any $\xi\in H^{\vee}$.
\end{itemize}
\end{ThmA}
The isomorphism (i) is easy to obtain. The key to (ii) is Theorem \ref{Thm-B}.

\subsection*{Organisation of the article and the strategy of the proofs}\hfill

Section \ref{Sec-P} collects some definitions and properties concerning (compactified) Jacobian varieties, (compactified) Prym varieties and the norm maps. Some lemmas are proved and will be used in later sections.

Section \ref{Sec-HF} is essentially self-contained. The goal of this section is to prove Theorem \ref{M-T}. The proof relies on \S \ref{SS-MSHB}, \S \ref{SS-SSHB} and \S \ref{subsec-FPLMH}, and the key arguments are given in \S \ref{subsec-OSRHF}. Besides, \S \ref{SS-SSHB} contains the generalisations of some results of Hausel-Pauly \cite{HP} that will be frequently used in later sections.

The rest of the article concerns the Fourier-Mukai transform of a skyscraper sheaf supported at an orbifold singularity of the $\PGL_n$-moduli space $\hat{M}^{ell}$. The problem will be rephrased in terms of the compactified Prym varieties $\bar{P}\cong\check{h}^{-1}_C(a)$, with $a\in\check{A}^{ell}_n(C)$. As of Section \ref{CPV}, we will work with a fixed $a\in\check{A}^{ell}_n(C)$, and the group $K_a$ controlling the irreducible components of $\bar{P}$ will be denoted by $G$. 

Suppose $x\in\bar{P}$ and $\Stab_{\Gamma}(x)=H$; that is, $x$ defines a point of $\hat{M}^{ell}$ with isotropy group $H$. Theorem \ref{M-T} guarantees that $H\subset G$. This allows us to consider
$$
C''\longrightarrow C'\longrightarrow C,
$$
where $C''\rightarrow C$ is a $G^{\vee}$-Galois covering and $C'\rightarrow C$ is an $H^{\vee}$-Galois covering. Associated to $C'$ is an intermediate compactified Prym variety $\bar{P}'$, while $\bar{P}$ is thought of as associated to $C$ (of course, it also depends on the spectral data). The precise definitions will be given in \S \ref{subsec-prym}, where we also introduce many notations that will be used throughout the article. The variety $\bar{P}$ is already studied in the literature (see \cite{dC} and \cite{HP}). We study $\bar{P}'$, as well as its regular part $P'$, in Section \ref{CPV}. In particular, we will determine their connected components. As opposed to $\bar{P}$, the variety $\bar{P}'$ is not connected as long as $H$ is nontrivial. Although $\bar{P}'$ is our focus of interest, the consideration of $C''$ is inevitable, since it is directly related to $G$. 

Section \ref{S-CDC} is a digression. By some computations of determinants of cohomology, we obtain Proposition \ref{Prop-phiWV-phiV}, which is only used in \S \ref{subsec-RPF}.

Section \ref{Sec-PS} begins with the definitions of the Fourier-Mukai kernels $\overline{\mathcal{Q}}{}'$ on $\bar{P}'\times[\bar{P}'/\Gamma]$. In particular, we also define the Fourier-Mukai kernel $\overline{\mathcal{Q}}$ on $\bar{P}\times[\bar{P}/\Gamma]$. In \S \ref{subsec-RPF}, we prove Theorem \ref{Thm-Pic'-repre-P}, which is only used in Proposition \ref{Prop-phiV-ptil}, leading to to Corollary \ref{Cor-Comp-Poinc-sh}. Another main result of this section is Theorem \ref{Thm-Twisted-Sheaf}. Its proof does not rely on \S \ref{subsec-RPF} or \S \ref{SS-Comp-Poinc-sh}. Corollary \ref{Cor-Comp-Poinc-sh} and Theorem \ref{Thm-Twisted-Sheaf} are the two main ingredients in the proof of Theorem \ref{Thm-B-proof}, and are only used therein.

Let us explain more about the roles of Corollary \ref{Cor-Comp-Poinc-sh} and Theorem \ref{Thm-Twisted-Sheaf}. Theorem \ref{Thm-B-proof} concerns the behaviour of a skyscraper sheaf on $[\bar{P}/\Gamma]$ under the Fourier-Mukai transform defined by $\overline{\mathcal{Q}}$. For our purpose, it suffices to restrict everything to $P\times[\bar{P}/\Gamma]$, where $\overline{\mathcal{Q}}$ becomes a line bundle $\mathcal{Q}$. The core of the problem is to understand the restriction of $\mathcal{Q}$ via the closed immersion $P\times\mathbf{B}H\hookrightarrow P\times[\bar{P}/\Gamma]$ that is defined by the orbifold singularity in question (here, $\mathbf{B}H$ is the classifying stack). Recall that in Section \ref{CPV} we have introduced an intermediate compactified Prym variety $\bar{P}'$. It turns out that it is isomorphic to $\bar{P}^H$ by Lemma \ref{Lem-barP'=barPG}, where $\bar{P}^H$ is the fixed-point locus of $H$. Thus, the closed immersion $\mathbf{B}H\hookrightarrow[\bar{P}/\Gamma]$ factors through $$[\bar{P}'/\Gamma]\cong[\bar{P}^H/\Gamma]\hookrightarrow[\bar{P}/\Gamma].$$Now, Corollary \ref{Cor-Comp-Poinc-sh} tells us that the pullback of $\mathcal{Q}$ to $P\times[\bar{P}'/\Gamma]$ is the same as the pullback of $\mathcal{Q}'=\overline{\mathcal{Q}}{}'|_{P'\times[\bar{P}'/\Gamma]}$ to $P\times[\bar{P}'/\Gamma]$ (let us admit for the moment that there exists a map from $P$ to $P'$).

It remains to understand the restriction of $\mathcal{Q}'$ to $P'\times\mathbf{B}H$. As a coherent sheaf on $P'\times\mathbf{B}H$, it decomposes according to the action of $H$. On each connected component of $P'$, there is a unique character $\xi\in H^{\vee}$, such that the $\xi$-isotypic component of $\mathcal{Q}'$ is nontrivial, since it is a line bundle. Theorem \ref{Thm-Twisted-Sheaf} explicitly determines the character $\xi$. The effect of projecting a sheaf to the first factor of $P'\times\mathbf{B}H$, which is part of the Fourier-Mukai transform operation, amounts to taking the $H$-invariant part. If $\xi$ is nontrivial, we get the zero sheaf. This explains the reason why the Fourier-Mukai transform of a skyscraper sheaf is only support at some particular connected components.

Finally, Theorem \ref{Thm-FHEP} follows easily from Theorem \ref{Thm-B-proof}.

\subsection*{Notations}\hfill

For any finite abelian group $G$, we write $G^{\vee}=\Hom(G,\mathbb{C}^{\ast})$, and by $|G|$ the cardinality of $G$. If $f:X\rightarrow Y$ is a morphism of algebraic stacks, and $\mathcal{M}(X)$ and $\mathcal{M}(Y)$ are the moduli spaces of certain class of sheaves on $X$ and $Y$, e.g. line bundles or torsion-free sheaves, then we denote by $f^{\vee}:\mathbf{M}(Y)\rightarrow\mathcal{M}(X)$ the morphism defined by pulling back sheaves if it is well-defined. However, if $F$ is a sheaf on $Y$, then its pullback via $f$ is denoted by $f^{\ast}F$ as usual. For any algebraic stack $\mathfrak{X}$, we denote by $\mathcal{D}(\mathfrak{X})$ the bounded derived category of coherent sheaves.

\subsection*{Acknowledgement.}\hfill

We would like to thank Zhiyu Liu for answering some questions. We would also like to thank Institute for Advanced Study in Mathematics at Zhejiang University for the wonderful research environment. We thank the anonymous referees for many suggestions that significantly improved the exposition of the article.

\numberwithin{equation}{subsection}
\section{Preliminaries}\label{Sec-P}
\subsection{Jacobians}\label{Sub-Jacob}\hfill

Let $\mathfrak{X}$ be a projective Deligne-Mumford stack over $\mathbb{C}$. For any scheme $S$ over $\mathbb{C}$, an $S$-flat coherent sheaf $F$ on $S\times\mathfrak{X}$ is a \textit{relative torsion-free rank one sheaf} if for any point $s\in S$, the restriction of $F$ to $\mathfrak{X}_s:=\Spec k(s)\times\mathfrak{X}$, denoted by $F_s$, is torsion-free and has generic rank one. We say that two such sheaves $F$ and $F'$ are equivalent if they differ by the pullback of a line bundle on $S$. Let $\overline{\Pic}{}_{\mathfrak{X}}^{pre}$ be the presheaf on $\Sch_{\mathbb{C}}$ that associates to any $S\in\Sch_{\mathbb{C}}$ the equivalence classes of relative torsion-free rank one sheaves on $S\times\mathfrak{X}$. Let $\Pic{}_{\mathfrak{X}}^{pre}\subset\overline{\Pic}{}_{\mathfrak{X}}^{pre}$ be the subpresheaf consisting of line bundles. We will respectively denote by $\Pic_{\mathfrak{X}}$ and $\overline{\Pic}{}_{\mathfrak{X}}$ the \'etale sheafification of $\Pic{}_{\mathfrak{X}}^{pre}$ and $\overline{\Pic}{}_{\mathfrak{X}}^{pre}$. Also, for any algebraic stack $\mathfrak{X}$, we will denote by $\Pic(\mathfrak{X})$ the group of isomorphism classes of line bundles on $\mathfrak{X}$. According to \cite[Th\'eor\`eme 2.2.6]{B}, in nice situations we have $\Pic{}_{\mathfrak{X}}^{pre}=\Pic_{\mathfrak{X}}$ (i.e., the presheaf $\Pic{}_{\mathfrak{X}}^{pre}$ is a sheaf). This is the case when for example $\mathfrak{X}$ is a quotient stack $[X/G]$, where $X$ is a connected projective variety over $\mathbb{C}$ (see \cite[Remarque 2.2.5]{B}). If $\mathfrak{X}$ is an integral projective variety on which rank one torsion-free sheaves are Cohen-Macaulay (for example a curve), then $\overline{\Pic}{}_{\mathfrak{X}}^{pre}=\overline{\Pic}{}_{\mathfrak{X}}$. This follows from \cite[Theorem 3.4 (iii)]{AK2}. As explained in \cite[\S 1.5]{AK2}, the rigidified Picard functor therein coincides with (the nonrigidified) $\overline{\Pic}{}_{\mathfrak{X}}^{pre}$.

Suppose that $X$ is an integral, projective and finitely presented scheme over $\mathbb{C}$. By \cite[n\textdegree{}232, Th\'eor\`eme 3.1]{FGA}, the sheaf $\Pic_X$ is representable by a scheme that is separated and locally of finite presentation over $\mathbb{C}$. (By \cite[\S 8.2, Theorem 2]{BLR}, if $X$ is not irreducible, then $\Pic_X$ is not necessarily separated.) By \cite[Theorem 3.1]{AK2}, the sheaf $\overline{\Pic}_X$ is representable by a scheme and the connected components are proper. We will denote by $\Pic^0_X\subset\Pic_X$ and $\overline{\Pic}{}^0_X\subset\overline{\Pic}_X$ the connected components containing $\mathcal{O}_X$.

Assume further that $X$ is an integral projective curve. For any scheme $S$ over $\mathbb{C}$ and a relative torsion-free rank one sheaf $F$ on $S\times X$, we say that $F$ is of degree $i$ if for any $s\in S$, the restriction $F_s$ satisfies $\chi(F_s)-\chi(\mathcal{O}_{X_s})=i$, where $\chi(-)$ denotes the Euler characteristic of a coherent sheaf.  For any $i\in\mathbb{Z}$, we denote by $\bar{J}^i_{X}\subset\overline{\Pic}_{X}$ the subscheme consisting of degree $i$ sheaves, and by $J^i_{X}\subset\Pic_{X}$ the subscheme consisting of degree $i$ line bundles. We have $J^0_X=\Pic^0_X$ and $\bar{J}^0_X=\overline{\Pic}{}^0_X$. In fact, $\bar{J}^i_X$ is an algebraic variety for any $i$. We will write $\bar{J}_X:=\bar{J}^0_X$ and $J_X=J^0_X$. We call the proper algebraic variety $\bar{J}^i_X$ the \textit{compactified Jacobian} of degree $i$ of $X$, and by the compactified Jacobian of $X$ we mean the degree 0 component $\bar{J}_X$. The open subvariety $J_X\subset\bar{J}_X$ is called the Jacobian of $X$. It is a commutative algebraic group, and its identity element will be written as $0$. We will denote by $\mathcal{U}_X$ (resp. $\overline{\mathcal{U}}_X$) the universal sheaf on $X\times J_X$ (resp. $X\times \bar{J}_X$). If $X$ is a smooth projective curve, then $J_X=\bar{J}_X$ is an abelian variety.

The \textit{Abel-Jacobi map} $\bs\alpha_{X,-1}:X\rightarrow\bar{J}^{-1}_X$ of degree -1 is defined by the ideal sheaf $\mathcal{I}_{\Delta}$ of the diagonal $\Delta:X\hookrightarrow X\times X$. Fix an invertible sheaf of degree $1$ on $X$, say $L_0$. The \textit{Abel-Jacobi map} $\bs\alpha_{X,L_0}:X\rightarrow\bar{J}_X$ is defined by the coherent sheaf $\mathcal{I}_{\Delta}\otimes \pr_1^{\ast}L_0$ on $X\times X$. We have
\begin{equation}\label{eq-AJ-uni}
(\Id\times\bs\alpha_{X,L_0})^{\ast}\overline{\mathcal{U}}_X\cong\mathcal{I}_{\Delta}\otimes \pr_1^{\ast}L_0\otimes\pr_2^{\ast}M
\end{equation}
for some invertible sheaf $M$ on $X$, where $\pr_i$ is the projection from $X\times X$ to the $i$'th factor.

\subsection{Prym Varieties}\label{Sub-Prym}\hfill

Let $X$ be an integral projective curve, $C$ a connected smooth projective curve, and $\pi:X\rightarrow C$ a finite morphism of degree $n$. Let $S$ be a scheme over $\mathbb{C}$ and let $F$ be a relative torsion-free rank one sheaf on $S\times X$. Write $\pi_S:=\Id_S\times\pi$. The norm of $F$ is the coherent sheaf on $S\times C$ defined by
\begin{equation}
\Nm_{\pi_S}(F):=\det(\pi_{S\ast}F)\otimes\det(\pi_{S\ast}\mathcal{O}_{S\times X})^{-1}.
\end{equation}
By \cite[Definition/Lemma 8.1]{Car}, this is well-defined and is a line bundle. These norm maps for varying $S$ define a norm map between the (compactified) Jacobians $\Nm_{\pi}:\bar{J}_X\rightarrow J_C$. By \cite[Corollary 3.12]{HP}, its restriction to $J_X$ coincides with the usual norm map as defined in \cite[\S 6.5]{EGAII}. The norm map $\Nm_{\pi}$ satisfies the following properties
\begin{itemize}
\item[(i)] For any line bundles $L$ and $L'$ on $X$, we have $\Nm_{\pi}(L\otimes L')=\Nm_{\pi}(L)\otimes\Nm_{\pi}(L')$ (see \cite[\S 3.1 Equation (8)]{HP}).
\item[(ii)] For any torsion-free rank one sheaf $F$ on $X$ and any line bundle $L$ on $C$, we have $\Nm_{\pi}(F\otimes\pi^{\ast}L)\cong \Nm_{\pi}(F)\otimes L^{\otimes n}$ (see \cite[Proposition 3.10]{HP}).
\end{itemize}

The compactified Prym variety associated to $\pi$ is defined by $\bar{P}_{\pi}:=\Nm_{\pi}^{-1}(0)\subset\bar{J}_X$. We call the open subvariety $P_{\pi}:=\bar{P}_{\pi}\cap J_X$ the Prym variety associated to $\pi$. 
Regarded as an \'etale sheaf on $\Sch_{\mathbb{C}}$, $\bar{P}_{\pi}$ associates to any scheme $S$ the set
\begingroup
\allowdisplaybreaks
\begin{align*}
\bar{P}_{\pi}(S)=\{F\in\bar{J}_X(S)\mid&\Nm_{\pi_{S}}(F)\cong\pr_1^{\ast}M\text{ for some line bundle $M$ on $S$.}\}
\end{align*}
\endgroup

Let $\nu:\tilde{X}\rightarrow X$ be the normalisation map, and write $\tilde{\pi}=\pi\circ\nu$. Denote by $K\subset\Gamma$ the kernel of $\tilde{\pi}^{\vee}:J_C\rightarrow J_{\tilde{X}}$. According to \cite[Theorem 1.1]{HP}, there is a natural bijection $\pi_0(P_{\pi})\cong K^{\vee}$, where $\pi_0(P_{\pi})$ denotes the group of connected components of $P_{\pi}$. 

\subsection{Poincar\'e Sheaves}\label{Sub-Poincare}\hfill

Let $X$ be as in \S \ref{Sub-Prym} and assume further that $X$ has planar singularities. Then the compactified Jacobian $\bar{J}:=\bar{J}_X$ is irreducible (see \cite{AIK}). This allows us to consider the proper scheme $\overline{\Pic}{}^0_{\bar{J}}$. 
\begin{Thm}(\cite[Theorem A, Theorem B]{Ari2})\label{Thm-Ari2AB}
There exists a coherent sheaf $\overline{\mathcal{P}}$ on $\bar{J}\times\bar{J}$ such that
\begin{itemize}
\item[(i)] The restriction $\overline{\mathcal{P}}|_{J\times\bar{J}\cup\bar{J}\times J}$ is a line bundle.
\item[(ii)] For any closed point $x\in \bar{J}$, the restriction $\overline{\mathcal{P}}|_{x\times\bar{J}}$ is a (Cohen-Macaulay) torsion-free sheaf of rank one and degree zero.
\item[(iii)] Regarded as a family of torsion-free sheaves on $\bar{J}$, the coherent sheaf $\overline{\mathcal{P}}$ induces an isomorphism 
\begin{equation}\label{Eq-rho}
\rho:\bar{J}\lisom\overline{\Pic}{}^0_{\bar{J}}.
\end{equation}
In particular, the connected scheme $\overline{\Pic}{}^0_{\bar{J}}$ is also irreducible and contains $\Pic^0_{\bar{J}}$ as an open dense subscheme.
\end{itemize}
\end{Thm}  

We will call $\overline{\mathcal{P}}$ the Poincar\'e sheaf. The restriction of $\overline{\mathcal{P}}$ to $J\times\bar{J}\cup\bar{J}\times J$, as well as any subvariety thereof, will be denoted by $\mathcal{P}$. By \cite[Theorem 3.4 (iii)]{AK2}, the subscheme of $\overline{\Pic}{}_{\bar{J}}$ parametrising Cohen-Macaulay sheaves admits a universal family. The above theorem says that the torsion-free sheaves that $\overline{\Pic}{}^0_{\bar{J}}$ parametrises are Cohen-Macaulay. Thus we obtain a universal sheaf $\overline{\mathcal{P}}_0$ on $\smash{\overline{\Pic}{}^0_{\bar{J}}\times\bar{J}}$ by restriction. We may require that the restriction of $\overline{\mathcal{P}}_0$ to $\smash{\overline{\Pic}{}^0_{\bar{J}}\times\{0\}}$ is the trivial line bundle. Then we have
\begin{equation}\label{eq-barPbarP0}
(\rho\times\Id)^{\ast}\overline{\mathcal{P}}_0\cong\overline{\mathcal{P}}.
\end{equation}

\begin{Rem}\label{Rem-Ari2AB}
Let $\bs\alpha_L:X\rightarrow\bar{J}$ be the Abel-Jacobi map as in \S \ref{Sub-Jacob}. As is pointed out in the remark in \cite{Ari2} that follows \cite[Theorem B]{Ari2}, the arguments in the proof of \cite[Theorem 2.6]{EK} shows that under $\bs\alpha_L$, a torsion-free sheaf pulls back to a torsion-free sheaf. Thus there is a well-defined morphism
\begin{equation}
\bs\alpha^{\vee}_L:\overline\Pic{}^0_{\bar{J}}\longrightarrow\overline\Pic{}^0_X=\bar{J}.
\end{equation}
According to the proof of \cite[Theorem B]{Ari2}, it is the inverse of $\rho$. In particular, $\bs\alpha^{\vee}_L$ is independent of the choice of $L$. We may therefore write $\bs\alpha^{\vee}=\bs\alpha^{\vee}_L$.
\end{Rem}

Let $\pi:X\rightarrow C$ be as in \S \ref{Sub-Prym}. Suppose that $\bs\alpha$ is defined by some degree 1 line bundle $L$. Let $L_C:=\Nm_{\pi}(L)\in J^1_C$, and let $\bs\alpha_C:C\rightarrow J_C$ be the Abel-Jacobi map defined by $L_C$. Then we have a commutative diagram
\begin{equation}\label{CD-AJ-XaC}
\begin{tikzcd}[row sep=2.5em, column sep=2em]
X \arrow[r, "\bs\alpha"] \arrow[d, swap, "\pi"] & \bar{J} \arrow[d, "\Nm_{\pi}"]\\
C \arrow[r, "\bs\alpha_C"'] & J_C.
\end{tikzcd}
\end{equation}
Taking the compactified Jacobians, we obtain the following commutative diagram:
\begin{equation}\label{CD-PicAJ-XaC}
\begin{tikzcd}[row sep=2.5em, column sep=2em]
\overline{\Pic}{}^0_{\bar{J}} \arrow[r, "\sim", "\bs\alpha^{\vee}"'] & \overline{\Pic}{}^0_{X}=\bar{J} \\
\Pic^0_{J_C} \arrow[r, "\sim", "\bs\alpha_C^{\vee}"'] \arrow[u, "\Nm_{\pi}^{\vee}"] & \Pic^0_C=J_C \arrow[u, swap, "\pi^{\vee}"],
\end{tikzcd}
\end{equation}
where the horizontal arrows are the isomorphisms induced by the Abel-Jacobi maps, and the vertical arrows are the pullback morphisms. Let $\mathcal{P}_{C,0}$ be the universal sheaf on $\Pic^0_{J_C}\times J_C$ and require that its restriction to $\Pic^0_{J_C}\times \{0\}$ is the trivial line bundle. We have the following isomorphism of coherent sheaves on $\Pic^0_{J_C}\times\bar{J}$:
\begin{equation}\label{Eq-NmV}
(\Nm^{\vee}_{\pi}\times\Id)^{\ast}\overline{\mathcal{P}}_0\cong(\Id\times\Nm_{\pi})^{\ast}\mathcal{P}_{C,0}.
\end{equation}
Indeed, a morphism $f$ from $\Pic^0_{J_C}$ to $\overline{\Pic}{}^0_{\bar{J}}$ is equivalent to a family $F$ of torsion free sheaves on $\bar{J}$ parametrised by $\Pic^0_{J_C}$. By definition, $\Nm^{\vee}_{\pi}$ is the morphism $f$ corresponding to the sheaf $F=(\Id\times\Nm_{\pi})^{\ast}\mathcal{P}_{C,0}$. The universal property of $\overline{\mathcal{P}}_0$ means that $(f\times\Id)^{\ast}\overline{\mathcal{P}}_0$ is isomorphic to $F$ up to the pullback of a line bundle $M$ on the first factor. Our normalisations of $\mathcal{P}_{C,0}$ and $\overline{\mathcal{P}}_0$ imply that $M$ must be trivial.

Similarly by the definition of $(\pi^{\vee})^{\vee}$, we have the following isomorphism of line bundles on $\overline{\Pic}{}^0_{\bar{J}}\times J_C$:
\begin{equation}\label{Eq-piVV}
((\pi^{\vee})^{\vee}\times\Id)^{\ast}\mathcal{P}_{C,0}\cong(\Id\times\pi^{\vee})^{\ast}\overline{\mathcal{P}}_0.
\end{equation}
Note that $\pi^{\vee}(J_C)$ is contained in $J$. Arinkin's theorem shows that if $F$ is any torsion-free rank one sheaf on $\bar{J}$, then its restriction to $J$ is a line bundle. Therefore, $(\Id\times\pi^{\vee})^{\ast}\overline{\mathcal{P}}_0$ is a family of line bundles on $J_C$ parametrised by $\overline{\Pic}{}^0_{\bar{J}}$.

Let $\mathcal{P}_C$ be the coherent sheaf on $J_C\times J_C$ that Theorem \ref{Thm-Ari2AB} gives for the smooth curve $C$. Using (\ref{eq-barPbarP0}), (\ref{CD-PicAJ-XaC}) and (\ref{Eq-NmV}), we deduce that
\begin{equation}\label{Eq-Nm-piV}
(\pi^{\vee}\times\Id)^{\ast}\overline{\mathcal{P}}\cong(\Id\times\Nm_{\pi})^{\ast}\mathcal{P}_C.
\end{equation}
By symmetry, we have
\begin{equation}\label{Eq-Nm-piV2}
(\Id\times\pi^{\vee})^{\ast}\overline{\mathcal{P}}\cong(\Nm_{\pi}\times\Id)^{\ast}\mathcal{P}_C.
\end{equation}
Note that $(\pi^{\vee})^{\vee}$ is the unique morphism satisfying (\ref{Eq-piVV}). Using this characterisation of $(\pi^{\vee})^{\vee}$, we deduce from (\ref{eq-barPbarP0}) and (\ref{Eq-Nm-piV2}) that the following diagram commutes
\begin{equation}\label{CD-piVV-Nm}
\begin{tikzcd}[row sep=2.5em, column sep=2em]
\overline{\Pic}{}^0_{\bar{J}} \arrow[r, "\sim", "\bs\alpha^{\vee}"'] \arrow[d, "(\pi^{\vee})^{\vee}"'] & \overline{\Pic}{}^0_X=\bar{J} \arrow[d, "\Nm_{\pi}"] \\
\Pic^0_{J_C} \arrow[r, "\sim", "\bs\alpha_C^{\vee}"'] & \Pic^0_C=J_C.
\end{tikzcd}
\end{equation}
We also deduce from (\ref{Eq-Nm-piV2}) that 
\begin{equation}\label{Eq-P-triv-barP}
\overline{\mathcal{P}}|_{\bar{P}_{\pi}\times\{\pi^{\vee}(g)\}}\cong\mathcal{O}_{\bar{P}_{\pi}}
\end{equation}
for any $g\in J_C$.
\begin{Lem}\label{Lem-Func-Nm}
Let $\pi:X\rightarrow C_1$ and $\phi:C_1\rightarrow C$ be finite morphisms between integral projective curves. Assume that $C_1$ and $C$ are smooth. Then $\Nm_{\phi\circ\pi}=\Nm_{\phi}\circ\Nm_{\pi}$ as morphisms from $\bar{J}_X$ to $J_C$.
\end{Lem}
\begin{proof}
We have $(\phi\circ\pi)^{\vee}=\pi^{\vee}\circ\phi^{\vee}$. It follows from (\ref{CD-piVV-Nm}) that $\Nm_{\phi\circ\pi}=\Nm_{\phi}\circ\Nm_{\pi}$.
\end{proof}

Recall that if $G$ is a finite abelian group, $f:V\rightarrow W$ an \'etale $G$-Galois covering of (connected) projective varieties, and $f^{\vee}:\Pic_W\rightarrow\Pic_V$ the dual morphism, then there is a natural isomorphism $\epsilon_G:G^{\vee}\isom\Ker f^{\vee}$. Let $\chi\in G^{\vee}$. Then $f^{\ast}\epsilon_G(\chi)\cong\mathcal{O}_V$ and $G$ naturally acts on $H^0(V,f^{\ast}\epsilon_G(\chi))\cong\mathbb{C}$ via $\chi$ (see \cite[\S 7, Proposition 3]{M}).

\begin{Lem}\label{Lem-G1G2}
Let $\sigma:G_1\rightarrow G_2$ be a homomorphism of finite abelian groups. For $i=1$, $2$, let $f_i:X_i\rightarrow Y_i$ be an \'etale Galois covering of connected projective varieties with Galois group $G_i$. Suppose that there is a commutative diagram
\begin{equation}\label{CD-Lem-G1G2-0}
\begin{tikzcd}[row sep=2.5em, column sep=2em]
X_1 \arrow[d, "f_1"'] \arrow[r, "u"] & X_2 \arrow[d, "f_2"] \\
Y_1 \arrow[r, "v"'] & Y_2,
\end{tikzcd}
\end{equation}
where $u$ is $G_1$-equivariant, with $G_1$ acting on $X_2$ via $\sigma$. Then, the following diagram commutes
$$
\begin{tikzcd}[row sep=2.5em, column sep=2em]
G_2^{\vee} \arrow[d, "\epsilon_{G_2}"'] \arrow[r, "\sigma^{\vee}"] & G_1^{\vee} \arrow[d, "\epsilon_{G_1}"] \\
\Ker f_2^{\vee} \arrow[r, "v^{\vee}"'] & \Ker f_1^{\vee}.
\end{tikzcd}
$$
\end{Lem}
\begin{proof}
Let $M$ be a line bundle on $Y_2$ such that $f_2^{\ast}M\cong\mathcal{O}_{X_2}$. Write $L=v^{\ast}M$. Then $f_1^{\ast}L\cong\mathcal{O}_{X_1}$. Therefore, the morphism $v^{\vee}$ between Picard varieties restricts to the kernels of $f_1^{\vee}$ and $f_2^{\vee}$. Write $\chi_2=\epsilon_{G_2}^{-1}(M)$ and $\chi_1=\epsilon_{G_1}^{-1}(L)$. For any $g\in G_1$, we have the following commutative diagram
\begin{equation}\label{CD-Lem-G1G2}
\begin{tikzcd}[row sep=2.5em, column sep=2em]
u^{\ast}f_2^{\ast}M \arrow[r, "\sim", "\varphi_2"'] \arrow[d, "\rotatebox{90}{\(\sim\)}"] & u^{\ast}\sigma(g)^{\ast}f_2^{\ast}M \arrow[r, "\sim"] & g^{\ast}u^{\ast}f_2^{\ast}M \arrow[d, "\rotatebox{90}{\(\sim\)}"] \\
f_1^{\ast}v^{\ast}M \arrow[rr, "\sim", "\varphi_1"'] && g^{\ast}f_1^{\ast}v^{\ast}M
\end{tikzcd}
\end{equation}
where all the arrows are the canonical isomorphism. The arrow $\varphi_2$ is the pullback by $u$ of the isomorphism defining the $G_2$-equivariant structure on $f_2^{\ast}M$, and the arrow $\varphi_1$ is the isomorphism defining the $G_1$-equivariant structure on $f_1^{\ast}L$. The right vertical arrow is the pullback of the left vertical arrow. Let $1$ denote the constant function with value 1 on $X_1$. Then $\varphi_2(1)=\chi_2(\sigma(g))$ and $\varphi_1(1)=\chi_1(g)$. Taking the global sections of the sheaves in (\ref{CD-Lem-G1G2}), we see that $\varphi_1(1)=\varphi_2(1)$. Since this holds for any $g\in G_1$, we have $\chi_1=\chi_2\circ\sigma$. This completes the proof.
\end{proof}

We give a reformulation of \cite[Lemma 2.1]{HP} in the case of Jacobian varieties. 
\begin{Lem}\label{Lem-reform-HP2.1}
Let $\pi:X\rightarrow C$ be a finite morphism between smooth projective curves. Let $K$ be the kernel of $\pi^{\vee}:J_C\rightarrow J_X$, and identify $K^{\vee}$ with the kernel of the dual map of the isogeny $J_C\rightarrow\pi^{\vee}(J_C)$. For any $p\in P_{\pi}$, regard $\rho_X(p)$ as a line bundle on $J_X$. Then, the map
\begingroup
\allowdisplaybreaks
\begin{align*}
P_{\pi}&\longrightarrow K^{\vee}\\
p&\longmapsto \rho_X(p)|_{\pi^{\vee}(J_C)}
\end{align*}
\endgroup
is well-defined and induces an isomorphism of groups $\pi_0(P_{\pi})\cong K^{\vee}$.
\end{Lem}
\begin{proof}
Let us first check that $\rho_X(p)|_{\pi^{\vee}(J_C)}$ lies in $K^{\vee}$. We have the following short exact sequence of abelian groups:
$$
0\longrightarrow K^{\vee}\longrightarrow \Pic^0(\pi^{\vee}(J_C))\stackrel{(\pi^{\vee})^{\vee}}{\longrightarrow}\Pic^0(J_C)\longrightarrow 0.
$$
Take $J_X$ for $\bar{J}$ in the upper right corner of (\ref{CD-piVV-Nm}), and regard $p$ as an element of $\bar{J}$. Since $p\in P_{\pi}$, we have $\Nm_{\pi}(p)=0$. As is explained in Remark \ref{Rem-Ari2AB}, the inverse of $\bs\alpha^{\vee}$ is exactly $\rho_X$. 
Then the commutativity of (\ref{CD-piVV-Nm}) implies that $(\pi^{\vee})^{\vee}(\rho_X(p))=0$, so the map in the statement of the lemma is well-defined. Now $i:\pi^{\vee}(J_C)\hookrightarrow J_X$ is an injective homomorphism of abelian varieties, and so $i^{\vee}$ is surjective and $\Ker(i^{\vee})$ is also an abelian variety, and in particular, connected. This implies that if two elements $p_1$ and $p_2$ of $P_{\pi}$ are such that $\rho_X(p_1)|_{\pi^{\vee}(J_C)}$ coincides with $\rho_X(p_2)|_{\pi^{\vee}(J_C)}$, then $p_1$ and $p_2$ must lie in the same connected component of $P_{\pi}$. We conclude that the homomorphism $\pi_0(P_{\pi})\rightarrow K^{\vee}$ is injective. Surjectivity follows from the surjectivity of $i^{\vee}$.
\end{proof}

\section{Hitchin Fibrations}\label{Sec-HF}

\subsection{Moduli spaces of Higgs bundles}\label{SS-MSHB}\hfill

Let $C$ be a connected smooth projective curve of genus $g>1$ over $\mathbb{C}$, and $D$ a line bundle on $C$ that is either isomorphic to the canonical bundle $\Omega_C$ or has degree $\deg D>2g-2$. A $D$-twisted Higgs bundle (or simply Higgs bundle) on $C$ is a pair $(\mathcal{E},\theta)$, where $\mathcal{E}$ is a vector bundle on $C$ and $\theta:\mathcal{E}\rightarrow\mathcal{E}\otimes D$ is a homomorphism of coherent sheaves. We denote by $M_n(C)$ the coarse moduli space of semi-stable rank $n$ (degree 0) Higgs bundles on $C$. We refer to \cite{dC} for the definition of stability condition. The Hitchin map
$$
h_C:M_n(C)\longrightarrow A_n(C):=\bigoplus_{i=1}^nH^0(C,D^{\otimes i})
$$
is surjective and projective, and sends $(\mathcal{E},\theta)$ to the characteristic polynomial of $\theta$. The affine space $A_n(C)$ is called the Hitchin base.

A semi-stable $\SL_n$-Higgs bundle is a semi-stable Higgs bundle $(\mathcal{E},\theta)$ such that $\det(\mathcal{E})\cong\mathcal{O}_C$ and $\Tr(\theta)=0$. We denote by $\check{M}_n(C)$ the moduli space of semi-stable $\SL_n$-Higgs bundles on $C$. The Hitchin map $h_C$ restricts to the Hitchin map for $\SL_n$:
$$
\check{h}_C:\check{M}_n(C)\longrightarrow\check{A}_n(C):=\bigoplus_{i=2}^nH^0(C,D^{\otimes i}),
$$
that is also surjective and projective. See \cite[\S 2.4]{dC}.

Let $\Gamma=\Pic^0(C)[n]$ be the $n$-torsion subgroup of the Jacobian variety of $C$. Then $\Gamma$ acts on $M_n(C)$ by tensorisation. This action preserves the subvariety $\check{M}_n(C)$ and respects the fibration $h_C$. The moduli space $\hat{M}_n(C)$ of semi-stable $\PGL_n$-Higgs bundles on $C$ is defined to be the quotient stack $[\check{M}_n(C)/\Gamma]$. The Hitchin map $\check{h}_C$ induces the Hitchin map 
$$
\hat{h}_C:\hat{M}_n(C)\longrightarrow\check{A}_n(C).
$$
Note that the Hitchin base in this case is the same as the $\SL_n$-case.

Denote by $|D|$ the total space of the line bundle $D$, which is a smooth surface. Let $\pi_D:|D|\rightarrow C$ be the natural map and let $t$ be the canonical section of $\pi_D^{\ast}D$. For any $a=(a_1,\ldots,a_n)\in A_n(C)$, the spectral curve $X_a\subset|D|$ is the zero set of the section 
\begin{equation}\label{eq-sec-spec-cov}
s_a:=t^n+(\pi_D^{\ast}a_1)t^{n-1}+\cdots+\pi_D^{\ast}a_n\in H^0(|D|,\pi_D^{\ast}D^{\otimes n}).
\end{equation}
Notice that the Hitchin base is identified with the sections of $\pi_D^{\ast}D^{\otimes n}$ of the form $s_a$. We will denote by $\pi_a:X_a\rightarrow C$ the restriction of $\pi_D$ to $X_a$, called a spectral cover, and $X_a$ is called a spectral curve. We denote by $A_n^{ell}(C)\subset A_n(C)$ the open subset corresponding to integral spectral curves. We will write $\check{A}_n^{ell}(C)=\check{A}_n(C)\cap A_n^{ell}(C)$, $M_n^{ell}(C)=h_C^{-1}(A_n^{ell}(C))$ and $\check{M}_n^{ell}(C)=\check{M}_n(C)\cap M^{ell}_n(C)$.

The Hitchin fibre $h_C^{-1}(a)$ for $a\in A_n^{ell}(C)$ admits the following modular description. Let $\bar{J}_a^i$ be the compactified Jacobian of degree $i=\binom{n}{2}\cdot\deg D$ of $X_a$ (see \cite[Eq. (5)]{dC}). Then the pushforward along $\pi_a$ defines an isomorphism of algebraic varieties
\begin{equation}\label{eq-pia-wedge}
\pi_a^{\wedge}:\bar{J}_a^{i}\lisom h_C^{-1}(a).
\end{equation}
This is known as the BNR correspondence (see \cite{BNR}).

\subsection{Some subvarieties of the Hitchin base}\label{SS-SSHB}\hfill

We introduce some closed subvarieties of the Hitchin base generalising the endoscopic loci (see for example \cite[\S 2.5]{dC}). Let $G\subset\Gamma$ be a subgroup such that $|G|$ divides $n$. Write $m=n/|G|$. For any $g\in G$, denote by $L_g$ the corresponding line bundle on $C$. The locally free sheaf $\bigoplus_{g\in G}L_g$ defines an \'etale Galois covering $\psi:\tilde{C}\rightarrow C$ with Galois group isomorphic to the dual group $G^{\vee}$.  Then $\{L_g\}_{g\in G}$ are exactly the line bundles on $C$ that pulls back to the trivial bundle on $\tilde{C}$. Note that $\tilde{C}$ is connected. This can be seen by induction on the number of direct factors of $G$. Suppose that $G=G'\oplus H$ where $H$ is a cyclic group, and that the Galois cover $\tilde{C}_{G'}$ is connected. Then $\tilde{C}$ is the fibre product $\tilde{C}_{G'}\times_C\tilde{C}_H$, where $\tilde{C}_H$ is the Galois cover of $C$ associated to $H$. Since the kernel of the pullback morphism $\Pic^0(C)\rightarrow\Pic^0(\tilde{C}_{G'})$ is exactly $G'$, the image of $H$ under this morphism is isomorphic to $H$, which we denote by $H'$. Now $\tilde{C}$ is the cyclic Galois cover of $\tilde{C}_{G'}$ defined by $H'$, and so is connected.

Write $\tilde{D}=\psi^{\ast}D$ and let $|\tilde{D}|$ denote the total space of $\tilde{D}$. Then $G^{\vee}$ acts on $|\tilde{D}|$ in a way that is compatible with the projection $\pi_{\tilde{D}}:|\tilde{D}|\rightarrow\tilde{C}$. Let $A_m(\tilde{C})$ be the Hitchin base associated to $\tilde{C}$ and $\tilde{D}$; that is,
$$
A_m(\tilde{C})=\bigoplus_{i=1}^mH^0(\tilde{C},\tilde{D}^{\otimes i}).
$$  
Thus $G^{\vee}$ also acts on $A_m(\tilde{C})$. For any $\tilde{a}\in A_m(\tilde{C})$, let $s_{\tilde{a}}\in H^0(|\tilde{D}|,\pi_{\tilde{D}}^{\ast}\tilde{D}^{\otimes m})$ be as in (\ref{eq-sec-spec-cov}). The product $\prod_{\chi\in G^{\vee}}\chi(s_{\tilde{a}})\in H^0(|\tilde{D}|,\pi_{\tilde{D}}^{\ast}\tilde{D}^{\otimes n})$ is $G^{\vee}$ invariant, and so descends to a section $s_a\in H^0(|D|,\pi_D^{\ast}D^{\otimes n})$. This defines a finite morphism
\begin{equation}\label{eq-qG}
q_G:A_m(\tilde{C})\longrightarrow A_n(C)
\end{equation}
sending $\tilde{a}$ to $a$. We will denote by $A_{n,G}(C)$ the image of $q_G$. If $G_1\subset G_2$, then $A_{n,G_2}(C)\subset A_{n,G_1}(C)$, since the $G_2^{\vee}$-Galois covering is the composition of a $(G_2/G_1)^{\vee}$-covering and the $G_1^{\vee}$-covering. 
\begin{Lem}\label{Lem-HP5.1}
Let $H^0(\tilde{C},\tilde{D})_{var}$ denote the direct sum of the isotypic components of $H^0(\tilde{C},\tilde{D})$ on which the Galois group $G^{\vee}$ acts via a nontrivial character. Then,
$$
q_G^{-1}(\check{A}_n(C))=H^0(\tilde{C},\tilde{D})_{var}\oplus\bigoplus_{i=2}^mH^0(\tilde{C},\tilde{D}^{\otimes i}).
$$
\end{Lem}
\begin{proof}
This is an analogue of \cite[Lemma 5.1]{HP} for noncyclic groups. For any $\tilde{a}=(\tilde{a}_1,\ldots,\tilde{a}_m)\in A_m(\tilde{C})$, the first component of $q_G(\tilde{a})$ is equal to $\sum_{\chi\in G^{\vee}}\chi(\tilde{a}_1)$. Let $\tilde{a}_1^0$ be the isotypic component of $\tilde{a}_1$ corresponding to the trivial character of $G^{\vee}$. It follows from the orthogonality of irreducible characters of finite groups that $\tilde{a}_1^0=0$ if and only if $\sum_{\chi\in G^{\vee}}\chi(\tilde{a}_1)=0$.
\end{proof}
Let $a\in A_n(C)$ and denote by $\pi_a:X_a\rightarrow C$ the corresponding spectral cover. By the definition of $A_{n,G}(C)$, the element $a$ lies in $A_{n,G}(C)$ if and only if 
\begin{equation}\label{eq-def-AnG}
X_a\times_C\tilde{C}=\bigcup_{\chi\in G^{\vee}}Z_{\tilde{a}}^{\chi}
\end{equation}
where $Z_{\tilde{a}}$ is a spectral cover of $\tilde{C}$ of degree $m$ and $Z_{\tilde{a}}^{\chi}$ is its image under the action of $\chi$ on $|\tilde{D}|$. When $X_a$ is integral, we have a more precise description of $Z_{\tilde{a}}$:
\begin{Lem}\label{Lem-AnG}
Let $a\in A_{n,G}^{ell}(C)$ for some $G$. 
Let $\pi_a:X_a\rightarrow C$ be the spectral cover associated to $a$, $\nu_a:\tilde{X}_a\rightarrow X_a$ the normalisation map, and $\psi:\tilde{C}\rightarrow C$ the $G^{\vee}$-Galois covering associated to $G$. Then the following assertions hold.
\begin{itemize}
\item[(i)] The morphism $\pi_a\circ\nu_a$ factors through $\tilde{C}$.
\item[(ii)] If we denote by $Z$ the image of morphism $\tilde{X}_a\rightarrow X_a\times_C\tilde{C}$ that results from the universal property of fibre product, then the restriction of $\pr_2$ to $Z\rightarrow\tilde{C}$ is an integral spectral cover of degree $m$.  In particular, $X_a\times_C\tilde{C}$ is equal to $\bigcup_{\chi\in G^{\vee}}Z^{\chi}$.
\item[(iii)] The morphism $\tilde{X}_a\rightarrow Z$ in (ii) is the normalisation map for $Z$, and the restriction of $\pr_1$ to $Z\rightarrow X_a$ is birational.
\end{itemize}
\end{Lem}
\begin{proof}
According to the proof of \cite[Theorem 5.3]{HP}, these assertions are true if $G$ is cyclic. We will prove by induction on the number of direct factors of $G$. 

Write $G=G_1\oplus H$, where $H$ is a nontrivial cyclic subgroup. The $G^{\vee}$-Galois covering $\psi:\tilde{C}\rightarrow C$ decomposes as 
\begin{equation}\label{eq-Lem-AnG}
\tilde{C}\stackrel{\psi_1}{\longrightarrow} C_1\stackrel{\psi_0}{\longrightarrow} C,
\end{equation}
where the Galois group of $\psi_1$ is $G_1^{\vee}\cong(G/H)^{\vee}$ and the Galois group of $\psi_0$ is $H^{\vee}$. We may regard $G_1\cong G/H$ as a subgroup of $\Pic^0(C_1)$ since the kernel of $\psi^{\vee}_0$ is $H$. Note that $G_1\subset\Pic^0(C_1)[m_1]$, with $m_1=n/|H|$, since $|G_1|$ divides $m_1$ (recall that $\Pic^0(C_1)[m_1]$ is the $m_1$-torsion subgroup of $\Pic^0(C_1)$). Apply (i) and (ii) to $H$ and $\psi_0$, we obtain the following commutative diagram
$$
\begin{tikzcd}[row sep=2.5em, column sep=2em]
\tilde{X}_a \arrow[dr, "\nu_a"']
\arrow[r, "\nu_1"] & X_{a_1} \arrow[r, "\pi_{a_1}"] \arrow[d]
& C_1 \arrow[d, "\psi_0"] \\
& X_{a} \arrow[r, "\pi_{a}"]
& C,
\end{tikzcd}
$$
where $\nu_1$ is the normalisation map, and $X_{a_1}$ is a spectral cover over $C_1$ corresponding to some $a_1\in A_{m_1}(C_1)$. Moreover, we have $q_H(a_1)=a$, where $q_H$ is the morphism (\ref{eq-qG}) defined for $H$. By assumption, $a\in A_{n,G}(C)$; that is, there exists $a_2\in A_m(\tilde{C})$ such that $q_G(a_2)=a$. Write $q_G$ as a composition
$$
A_m(\tilde{C})\stackrel{q_{G_1}}{\longrightarrow}A_{m_1}(C_1)\stackrel{q_H}{\longrightarrow}A_n(C).
$$
Since $a_1$ and $q_{G_1}(a_2)$ lie in the same fibre of $q_H$, they must lie in the same orbit under the action of $H^{\vee}$. Now $q_{G_1}$ is $H^{\vee}$-equivariant, so $a_1$ lies in the image of $q_{G_1}$ (i.e., $a_1\in A_{m_1,G_1}(C_1)$). By induction assumption, we see that $\pi_{a_1}\circ\nu_1$ factors through $\tilde{C}$. This implies $(i)$ for $G$; also, there exists some $\tilde{a}\in A_m(\tilde{C})$ such that $X_{\tilde{a}}$ is the image of $\tilde{X}_a$ in $X_{a_1}\times_{C_1}\tilde{C}\subset X_a\times_C\tilde{C}$ and $X_{a_1}\times_{C_1}\tilde{C}\cong\bigcup_{\chi\in G_1^{\vee}}X_{\tilde{a}}^{\chi}$. Combined with $X_{a}\times_{C}C_1\cong\bigcup_{\chi\in H^{\vee}}X_{a_1}^{\chi}$, this shows that $X_{a}\times_{C}\tilde{C}\cong\bigcup_{\chi\in G^{\vee}}X_{\tilde{a}}^{\chi}$; that is, (ii) is true for $G$. Again by induction assumption, the map $\tilde{X}_a\rightarrow X_{\tilde{a}}$ is the normalisation map, and the composition $X_{\tilde{a}}\rightarrow X_{a_1}\rightarrow X_a$ is birational. This completes the proof.
\end{proof}

Denote by $A_{n}^{\heartsuit}(C)\subset A_{n}(C)$ the open subset corresponding to reduced spectral curves. If $G\subset\Gamma$ is a subgroup such that $|G|$ divides $n$, then we write $A_{n,G}^{\heartsuit}(C)=A_{n,G}(C)\cap A_n^{\heartsuit}(C)$. There is an alternative characterisation of the subvarieties $A^{\heartsuit}_{n,G}(C)$. We need to introduce a finite subgroup $K_a\subset\Pic^0(C)$ associated to any $a\in A^{\heartsuit}_n(C)$. We may write 
\begin{equation}\label{X_a}
X_a=\sum_{\lambda}X_{a,\lambda},
\end{equation}
where the $X_{a,\lambda}$'s are distinct irreducible components of $X_a$. Let $$\pi_{a,\lambda}:X_{a,\lambda}\longrightarrow C$$ be the restriction of $\pi_a$, and let $n_{a,\lambda}$ be its degree. Then $n=\sum_{\lambda}n_{a,\lambda}$. For each $\lambda$, let $\nu_{a,\lambda}:\tilde{X}_{a,\lambda}\rightarrow X_{a,\lambda}$ be the normalisation map, and let $\tilde{\pi}_{a,\lambda}=\pi_{a,\lambda}\circ \nu_{a,\lambda}$. Put $K_{a,\lambda}:=\Ker(\tilde{\pi}_{a,\lambda}^{\vee})$. In view of \S \ref{Sub-Prym} (ii), $K_{a,\lambda}$ is contained in the $n_{a,\lambda}$-torsion subgroup of $\Pic^0(C)$. Finally, put $$K_a:=\bigcap_{\lambda}K_{a,\lambda}.$$
\begin{Thm}\label{noncyc}
Let $G\subset\Gamma$ be a subgroup such that $|G|$ divides $n$, and let $a\in A_n^{\heartsuit}(C)$. Then $a$ lies in $A_{n,G}(C)$ if and only if $K_a\supset G$.
\end{Thm}
\begin{proof}
We prove by induction on the number of direct factors of $G$, using the results of Hausel-Pauly for cyclic groups (see \cite[Theorem 5.3]{HP}). Suppose that $G=G_1\oplus H$, where $H$ is a nontrivial cyclic group. If $a\in A_{n,G}(C)$, then $a\in A_{n,H}(C)$ and $a\in A_{n,G_1}(C)$. By induction assumption, we have $K_a\supset H$ and $K_a\supset G_1$; therefore, $K_a\supset G$. Conversely, assume that $K_a\supset G$. Consider the composition of Galois coverings 
(\ref{eq-Lem-AnG}). We have $K_a\supset H$, and so $a\in A_{n,H}(C)$; that is, there exists $a_1\in A_{m_1}(C_1)$ such that $q_H(a_1)=a$. We have $G_1\cong G/H\subset K_a/H$. As in the proof of Lemma \ref{Lem-AnG}, we may regard $G_1$ and $K_a/H$ as subgroups of $\Pic^0(C_1)$. Moreover, $G_1\subset\Pic^0(C_1)[m_1]$. We claim that $K_a/H\subset K_{a_1}$. Then, by induction assumption (for the curve $C_1$), $a_1$ lies in the image of $A_m(\tilde{C})\rightarrow A_{m_1}(C_1)$; therefore $a\in A_{n,G}(C)$. 

Now we prove the claim. Let $X_a=\cup_{\lambda}X_{a,\lambda}$ be the decomposition of the reduced spectral curve $X_a$ into its irreducible components. By Lemma \ref{Lem-AnG}, for each $\lambda$, there is a spectral curve $Z_{\lambda}$ over $C_1$ such that $X_{a,\lambda}\times_CC_1\cong\bigcup_{\chi\in H^{\vee}}Z_{\lambda}^{\chi}$. Moreover, the normalisation $\tilde{X}_{a,\lambda}$ is also the normalisation of $Z_{\lambda}$. We have the following commutative diagram:
$$
\begin{tikzcd}[row sep=2.5em, column sep=2em]
\tilde{X}_{a,\lambda} \arrow[dr, "\nu_{\lambda}"']
\arrow[r, "\nu_{1,\lambda}"] & Z_{\lambda} \arrow[r, "\pi_{a_1,\lambda}"] \arrow[d]
& C_1 \arrow[d, "\psi_0"] \\
& X_{a,\lambda} \arrow[r, "\pi_{a,\lambda}"]
& C,
\end{tikzcd}
$$
where $\nu_{\lambda}$ and $\nu_{1,\lambda}$ are normalisation maps. We may take $\cup_{\lambda}Z_{\lambda}$ as the spectral curve over $C_1$ corresponding to $a_1$, and so $K_{a_1}=\cap_{\lambda}K_{a_1,\lambda}$. Recall that $K_{a,\lambda}$ is the kernel of $\nu_{\lambda}^{\vee}\circ\pi_{a,\lambda}^{\vee}$, that $K_{a_1,\lambda}$ is the kernel of $\nu_{1,\lambda}^{\vee}\circ\pi_{a_1,\lambda}^{\vee}$, and that $H$ is the kernel of $\psi_0^{\vee}$. We conclude that $K_{a,\lambda}/H\subset K_{a_1,\lambda}$; therefore, $K_a/H\subset K_{a_1}$.
\end{proof}
\begin{Rem}
Alternatively, we may check that the arguments of \cite[Theorem 5.3]{HP} also work for noncyclic $G$. However, if $a$ corresponds to a nonreduced spectral curve, then we cannot show that $K_a\supset G$ implies $a\in A_{n,G}(C)$. In fact, according to the proof of \cite[Theorem 5.3]{HP}, if $G=H$ is cyclic, $|G|=d$, and $X_a$ is irreducible and nonreduced with $X_a=kX_{a,red}$, then $a_1$ corresponds to a possibly nonreduced spectral curve $X_{a_1}$ over $C_1$ with $X_{a_1}=\frac{k}{\gcd(k,d)}X_{a_1,\red}$. When $G$ is not cyclic, the factor $\gcd(k,d)$ should be replaced by the cardinality of the kernel of the multiplication map $[k]$ on $G$. But this number may not divide $k$, in which case $X_{a_1}$ is not well-defined.
\end{Rem}
Write $\check{A}_{n,G}(C)=A_{n,G}(C)\cap\check{A}_n(C)$.
\begin{Cor}\label{Cor-noncyc}
Let $G\subset\Gamma$ be a subgroup such that $|G|$ divides $n$, and let $a\in \check{A}_n^{\heartsuit}(C)$. Then $a$ lies in $\check{A}_{n,G}(C)$ if and only if $K_a\supset G$.
\end{Cor}
If $D=\Omega_C$, we have $$\dim(\check{A}_{n,G}(C))=(n^2/|G|-1)(g-1).$$ Indeed, Lemma \ref{Lem-HP5.1} reduces the computation to $d_1:=\dim A_m(\tilde{C})$ and $d_2:=\dim H^0(\tilde{C},\tilde{D})^{G^{\vee}}$, since $q_G$ is a finite morphism. It is known that $d_1=m^2(g(\tilde{C})-1)+1$ and $d_2=\dim H^0(C,\Omega_C)=g$, when $D=\Omega_C$. We obtain the desired formula as in \cite[Lemma 7.1]{HP}. If $\deg D>2g-2$, we have $$\dim(\check{A}_{n,G}(C))=d_1-d_2=\frac{1}{2}(\frac{n^2}{|G|^2}+\frac{n}{|G|})\deg D-(n-1)(g-1)-\deg D.$$The formula for $d_1$ is given in \cite[\S 6.1]{dC}, while $d_2$ follows from Riemann-Roch.

Write $A_{n,G}^{ell}(C)=A_{n,G}(C)\cap A_n^{ell}(C)$, and $\check{A}_{n,G}^{ell}(C)=A_{n,G}^{ell}(C)\cap\check{A}_{n}(C)$.
\begin{Lem}\label{Lem-q-1subs-Aell}
The inverse image $q_G^{-1}(A_n^{ell}(C))$ is contained in $A_m^{ell}(\tilde{C})$. Moreover, the action of $G^{\vee}$ on $q_G^{-1}(A_n^{ell}(C))$ is free. In particular, the action of $G^{\vee}$ on $M_m^{ell}(\tilde{C},G):=h_{\tilde{C}}^{-1}(q_G^{-1}(A_n^{ell}(C)))$ is free.
\end{Lem}
\begin{proof}
By Lemma \ref{Lem-AnG}, for any $a\in A_{n,G}^{ell}(C)$, the spectral curve $Z_{\tilde{a}}$ over $\tilde{C}$ as in (\ref{eq-def-AnG}) is integral, and $Z_{\tilde{a}}^{\chi}\ne Z_{\tilde{a}}^{\chi'}$ if $\chi\ne\chi'$. This means that $\tilde{a}\in A_m^{ell}(\tilde{C})$ and the action of $G^{\vee}$ is free.
\end{proof}

Let $\phi:Z\rightarrow X_a$ be the birational morphism above.  Let $\phi^{\vee}:\Pic_{X_a}\longrightarrow \Pic_Z$ be the morphism defined by pulling back line bundles.
\begin{Lem}\label{Lem-phiaV-J}
$\phi^{\vee}$ is surjective.
\end{Lem}
\begin{proof}
Since $\phi$ is a finite morphism and is birational according to the proof of \cite[Theorem 5.3]{HP}, the arguments in the proof of \cite[\S 9.2, Proposition 9]{BLR} apply.
\end{proof}

\subsection{Fixed point locus in the moduli space of Higgs bundles}\label{subsec-FPLMH}\hfill

Let $G$, $m$ and $\psi$ be as in the previous subsection. Let 
$$
\mathbf{q}_G:M_m(\tilde{C})\longrightarrow M_n(C)
$$
be the morphism defined by pushing forward Higgs bundles along $\psi$. Here we use the fact the the pushforward of a semi-stable Higgs bundle along a Galois cover is semi-stable \cite[Lemma 3.2]{MS2}. Note that $M^{ell}_m(\tilde{C},G)$ as defined in Lemma \ref{Lem-q-1subs-Aell} is equal to $\mathbf{q}_G^{-1}(M_n^{ell}(C))$. Let $\check{M}_m(\tilde{C},G)=\mathbf{q}_G^{-1}(\check{M}_n(C))$, and let $
\check{\mathbf{q}}_G$ be the restriction of $\mathbf{q}_G$ to $\check{M}_m(\tilde{C},G)$. Let $\check{A}_m(\tilde{C},G)=q_G^{-1}(\check{A}_n(C))$, and let $\check{q}_G$ be the restriction of $q_G$ to $\check{A}_m(\tilde{C},G)$. We have a commutative diagram
\begin{equation}\label{CD-h-q}
\begin{tikzcd}[row sep=2.5em, column sep=2em]
\check{M}_m(\tilde{C},G) \arrow[r, "\check{\mathbf{q}}_G"] \arrow[d, "\check{h}_{\tilde{C}}", swap] & \check{M}_n(C) \arrow[d, "\check{h}_{C}"] \\
\check{A}_m(\tilde{C},G) \arrow[r, "\check{q}_G"'] & \check{A}_n(C).
\end{tikzcd}
\end{equation}
The Galois group $G^{\vee}$ of $\psi$ acts on $\check{M}_m(\tilde{C},G)$ by pulling back Higgs bundles, and on $\check{A}_m(\tilde{C},G)$ by pulling back sections of powers of $\tilde{D}$. Obviously the Hitchin map $\check{h}_{\tilde{C}}$ is $G^{\vee}$-equivariant. The group $\Gamma$ acts on $\check{M}_m(\tilde{C},G)$ fibrewise with respect to $\check{h}_{\tilde{C}}$, since its action does not affect the Higgs fields. These two actions commute, since an element of $\Gamma$ is first pulled back to $\tilde{C}$ before tensored with Higgs bundles, and the pullback is not affected by covering transformations. 

Let $\check{M}_m^{ell}(\tilde{C},G)=\mathbf{q}_G^{-1}(\check{M}_n^{ell}(C))$. We have $\check{M}_m^{ell}(\tilde{C},G)=\check{M}_m(\tilde{C},G)\cap M^{ell}_m(\tilde{C},G)$. Note that $\check{M}_m^{ell}(\tilde{C},G)\ne \check{M}_m(\tilde{C},G)\cap h_{\tilde{C}}^{-1}(A_m^{ell}(\tilde{C}))$.

\begin{Prop}\label{Prop-NR}
The morphism $\check{\mathbf{q}}_G$ is $\Gamma$-equivariant, and induces an isomorphism 
\begin{equation}\label{eq-Prop-NR}
\check{M}_m^{ell}(\tilde{C},G)/G^{\vee}\cong\check{M}_n^{ell}(C)^G,
\end{equation}
where $\check{M}_n^{ell}(C)^G\subset\check{M}_n^{ell}(C)$ is the subvariety of $G$-fixed points.
\end{Prop}
This is an analogue of \cite[Proposition 7.1]{HT} and \cite[Proposition 3.3]{NR}. Since the author is unable to access \cite{NR}, we give the details of the proof.
\begin{proof}
We will show that $\mathbf{q}_G$ induces a bijection 
\begin{equation}\label{eq-Prop-NR-1}
\tilde{\mathbf{q}}_G:M_m^{ell}(\tilde{C},G)/G^{\vee}\cong M_n^{ell}(C)^G
\end{equation}
at the level of $\mathbb{C}$-points. Since $M_n^{ell}(C)^G$ is smooth as the fixed point locus of a finite group in a smooth variety, it follows from the Zariski Main Theorem that $\tilde{\mathbf{q}}_G$ is an isomorphism. Since $\check{M}_m(\tilde{C},G)$ is by definition the inverse image of $\check{M}_n(C)$ under $\mathbf{q}_G$ and $\check{\mathbf{q}}_G$ is the restriction of $\mathbf{q}_G$, the desired isomorphism (\ref{eq-Prop-NR}) is simply the restriction of $\tilde{\mathbf{q}}_G$. It follows from the projection formula that $\mathbf{q}_G$ is $\Gamma$-equivariant.

Let $S$ be a $\mathbb{C}$-scheme and let $(\tilde{E},\tilde{\theta})$ be a family of semi-stable Higgs bundles on $S\times\tilde{C}$, then $(\Id_S\times\psi)_{\ast}(\tilde{E},\tilde{\theta})$ is a family of Higgs bundles on $S\times C$. Let $g\in G$ and denote by $L_g$ the corresponding line bundle on $C$. Since $\psi^{\ast}L_g\cong\mathcal{O}_{\tilde{C}}$, it follows from the projection formula that tensorisation by $\pr_2^{\ast}L_g$ leaves $(\Id_S\times\psi)_{\ast}(\tilde{E},\tilde{\theta})$ unchanged. This shows that the morphism $\mathbf{q}_G$ factors through $M_n(C)^G$. 

Then we show that $\mathbf{q}_G$ descends to a morphism $\bar{\mathbf{q}}_G:M_m(\tilde{C},G)/G^{\vee}\rightarrow M_n(C)^G$. The Galois group $G^{\vee}$ acts on $\tilde{C}$, and so pulling back $(\tilde{E},\tilde{\theta})$ along $\Id_S\times\tau$, for $\tau\in G^{\vee}$, gives a new Higgs bundle $(\tilde{E}^{\tau},\tilde{\theta}^{\tau})$ on $S\times\tilde{C}$. This defines an automorphism $\varphi_{\tau}$ of $M_m(\tilde{C},G)$. That $\mathbf{q}_G\circ\varphi_{\tau}=\mathbf{q}_G$ is equivalent to saying that pushing forward $(\tilde{E},\tilde{\theta})$ and $(\tilde{E}^{\tau},\tilde{\theta}^{\tau})$ along $\Id\times\psi$ gives isomorphic Higgs bundles on $S\times C$. But this is obvious.

Now we restrict $\bar{\mathbf{q}}_G$ to $M_m^{ell}(\tilde{C},G)/G^{\vee}$ and show that it induces a bijection of $\mathbb{C}$-points. Note that if the characteristic polynomial of a Higgs bundle lies in the elliptic locus, then the Higgs bundle is necessarily stable. Let $(E_1,\theta_1)$ and $(E_2,\theta_2)$ be stable Higgs bundles on $\tilde{C}$ such that $(\psi_{\ast}E_1,\psi_{\ast}\theta_1)$ is isomorphic to $(\psi_{\ast}E_2,\psi_{\ast}\theta_2)$. Then $$\bigoplus_{\tau\in G^{\vee}}(\tau^{\ast}E_1,\tau^{\ast}\theta_1)\cong(\psi^{\ast}\psi_{\ast}E_1,\psi^{\ast}\psi_{\ast}\theta_1)\cong(\psi^{\ast}\psi_{\ast}E_2,\psi^{\ast}\psi_{\ast}\theta_2)\cong\bigoplus_{\tau\in G^{\vee}}(\tau^{\ast}E_2,\tau^{\ast}\theta_2).$$By Krull-Schmidt, there exists some $\tau\in G^{\vee}$ such that $(E_1,\theta_1)\cong(\tau^{\ast}E_2,\tau^{\ast}\theta_2)$. Therefore $\bar{\mathbf{q}}_G$ is injective. Write $\mathcal{A}=\bigoplus_{g\in G}L_g$. It is the sheaf of $\mathcal{O}_C$-algebra that defines the affine morphism $\psi$. If $(E,\theta)$ is a Higgs bundle on $C$ such that $(E,\theta)\cong(E\otimes L_g,\theta\otimes\Id)$ for any $g\in G$, then $E$ is an $\mathcal{A}$-module and $\theta$ is a homomorphism of $\mathcal{A}$-modules. Therefore, there exists a unique Higgs bundle $(\tilde{E},\tilde{\theta})$ on $\tilde{C}$ whose pushforward along $\psi$ is isomorphic to $(E,\theta)$, i.e. $\bar{\mathbf{q}}_G$ is surjective.
\end{proof}

A consequence of the above proposition is  Lemma \ref{Lem-barP'=barPG} below, which will only be used in the proof of Theorem \ref{Thm-B-proof}. We first prepare a lemma:
\begin{Lem}\label{Lem-P'=PG}
Let $\tilde{a}\in\check{A}_m(\tilde{C},G)$ be such that $a:=\check{q}_G(\tilde{a})\in \check{A}^{ell}_n(C)$. Then $\check{\mathbf{q}}_G$ restricts to an isomorphism: $$\check{h}_{\tilde{C}}^{-1}(\tilde{a})\lisom\check{h}_C^{-1}(a)^G,$$where $\check{h}_C^{-1}(a)^G$ is the subvariety of $G$-fixed points. Consequently, we have an isomorphism of quotient stacks: $$[\check{h}_{\tilde{C}}^{-1}(\tilde{a})/\Gamma]\lisom[\check{h}_C^{-1}(a)^G/\Gamma].$$
\end{Lem}
\begin{proof}
It follows from Proposition \ref{Prop-NR} that $\check{\mathbf{q}}_G$ restricts to an isomorphism
$$
\check{\mathbf{q}}_G^{-1}(\check{h}_C^{-1}(a))/G^{\vee}\cong\check{h}_C^{-1}(a)^G.
$$
By Lemma \ref{Lem-q-1subs-Aell}, the action of the Galois group $G^{\vee}$ on $\check{q}_G^{-1}(\check{A}^{ell}_n(C))$ is free. Since $\check{h}_{\tilde{C}}$ is $G^{\vee}$-equivariant, we have $$\check{\mathbf{q}}_G^{-1}(\check{h}_C^{-1}(a))\cong\check{h}_{\tilde{C}}^{-1}(\check{q}_G^{-1}(a))=\bigsqcup_{g\in G^{\vee}}\check{h}_{\tilde{C}}^{-1}(\tilde{a})^g,$$where $\check{h}_{\tilde{C}}^{-1}(\tilde{a})^g$ is the image of $\check{h}_{\tilde{C}}^{-1}(\tilde{a})$ under the action of $g$. The quotient of this variety by $G^{\vee}$ is isomorphic to $\check{h}_{\tilde{C}}^{-1}(\tilde{a})$. 
\end{proof}

By Lemma \ref{Lem-AnG}, we have a commutative diagram
$$
\begin{tikzcd}[row sep=2.5em, column sep=2em]
Z_{\tilde{a}} \arrow[r, "\pi_{\tilde{a}}"] \arrow[d, "\phi", swap] & \tilde{C} \arrow[d, "\psi"] \\
X_a \arrow[r, "\pi_a"'] & C,
\end{tikzcd}
$$ 
where $\pi_{\tilde{a}}:Z_{\tilde{a}}\rightarrow\tilde{C}$ is the spectral cover corresponding to some $\tilde{a}\in A_m(\tilde{C})$. Since $\phi$ is birational, the pushforward along $\phi$ defines a morphism $\phi^{\wedge}_0:\bar{J}^j_{Z_{\tilde{a}}}\rightarrow\bar{J}^i_{X_a}$, such that the following diagram commutes
\begin{equation}\label{CD-Jj-Ji-h}
\begin{tikzcd}[row sep=2.5em, column sep=2em]
\bar{J}_{Z_{\tilde{a}}}^j \arrow[r, "\sim", "\pi_{\tilde{a}}^{\wedge}"'] \arrow[d, "\phi^{\wedge}_0", swap] & h_{\tilde{C}}^{-1}(\tilde{a}) \arrow[d, "\mathbf{q}_G"] \\
\bar{J}_{X_a}^i \arrow[r, "\sim", "\pi^{\wedge}_a"'] & h_C^{-1}(a),
\end{tikzcd}
\end{equation}
where $j=\binom{m}{2}\cdot\deg\psi^{\ast}D$ and $i$ is as in (\ref{eq-pia-wedge}). By the surjectivity of norm maps, there exists line bundles $L$ and $\tilde{L}$ on $X_a$ and $Z_{\tilde{a}}$ respectively such that
\begin{equation}\label{eq-NmL=detpi}
\Nm_{\pi}(L)\cong\det(\pi_{a\ast}\mathcal{O}_{X_a})^{-1},\text{ and }\Nm_{\psi\circ\pi_{\tilde{a}}}(\tilde{L})\cong\det(\psi_{\ast}\pi_{\tilde{a}\ast}\mathcal{O}_{Z_{\tilde{a}}})^{-1}.
\end{equation}
Note that $\deg L=i$ and $\deg \tilde{L}=j$. Then tensorisations by $L$ and $\tilde{L}$ define isomorphisms $\tau_{L}:\bar{J}_{X_a}\isom\bar{J}^i_{X_a}$ and $\tau_{\tilde{L}}:\bar{J}_{Z_{\tilde{a}}}\isom\bar{J}^j_{Z_{\tilde{a}}}$. By Lemma \ref{Lem-phiaV-J}, there exists a line bundle $L'$ on $X_a$ such that $\phi^{\ast}L'\cong \tilde{L}$. Put 
\begin{equation}\label{eq-K}
K=L^{-1}\otimes L',
\end{equation} and define a morphism $\phi^{\wedge}:\bar{J}_{Z_{\tilde{a}}}\rightarrow \bar{J}_{X_a}$, sending $F\in \bar{J}_{Z_{\tilde{a}}}$ to $\phi_{\ast}F\otimes K$. Then the following diagram commutes
\begin{equation}\label{CD-JJJjJi}
\begin{tikzcd}[row sep=2.5em, column sep=2em]
\bar{J}_{Z_{\tilde{a}}} \arrow[r, "\sim", "\tau_{\tilde{L}}"'] \arrow[d, "\phi^{\wedge}", swap] & \bar{J}^j_{Z_{\tilde{a}}} \arrow[d, "\phi^{\wedge}_0"] \\
\bar{J}_{X_a}\arrow[r, "\sim", "\tau_{L}"'] & \bar{J}_{X_a}^i.
\end{tikzcd}
\end{equation}
\begin{Lem}\label{Lem-barP'=barPG}
$\phi^{\wedge}$ induces an isomorphism
$$
\phi^{\wedge}:\bar{P}_{\psi\circ\pi_{\tilde{a}}}\lisom\bar{P}_{\pi_a}^G.
$$
\end{Lem}
\begin{proof}
Combining (\ref{CD-Jj-Ji-h}) and (\ref{CD-JJJjJi}), we obtain the following commutative diagram
$$
\begin{tikzcd}[row sep=2.5em, column sep=2em]
\bar{J}_{Z_{\tilde{a}}} \arrow[r, "\sim", "\pi_{\tilde{a}}^{\wedge}\circ\tau_{\tilde{L}}"'] \arrow[d, "\phi^{\wedge}", swap] & h^{-1}_{\tilde{C}}(\tilde{a}) \arrow[d, "\mathbf{q}_G"] \\
\bar{J}_{X_a} \arrow[r, "\sim", "\pi_a^{\wedge}\circ\tau_{L}"'] & h^{-1}_{C}(a).
\end{tikzcd}
$$
Using (\ref{eq-NmL=detpi}) and \cite[Proposition 3.10]{HP}, we deduce that the above diagram restricts to the following diagram
\begin{equation}\label{CD-PP-hV}
\begin{tikzcd}[row sep=2.5em, column sep=2em]
\bar{P}_{\psi\circ\pi_{\tilde{a}}} \arrow[r, "\sim", "\pi_{\tilde{a}}^{\wedge}\circ\tau_{\tilde{L}}"'] \arrow[d, "\phi^{\wedge}", swap] & \check{h}^{-1}_{\tilde{C}}(\tilde{a}) \arrow[d, "\check{\mathbf{q}}_G"] \\
\bar{P}_{\pi} \arrow[r, "\sim", "\pi_a^{\wedge}\circ\tau_{L}"'] & \check{h}^{-1}_{C}(a).
\end{tikzcd}
\end{equation}
The bottom arrow respects the action of $G$ by the projection formula. By Lemma \ref{Lem-P'=PG}, we get the desired isomorphism.
\end{proof}

\subsection{Orbifold Singularities and Reducible Hitchin Fibres}\label{subsec-OSRHF}\hfill

We prove in this subsection Theorem \ref{M-T}, our first main result.

\begin{Prop}
Let $G$ be a subgroup of $\Gamma$ such that $|G|$ divides $n$, and let $a\in A_{n,G}^{\heartsuit}(C)$. Suppose that for any subgroup $H\subset\Gamma$ strictly containing $G$ such that $|H|$ divides $n$, we have $a\notin A_{n,H}(C)$. Then $K_a=G$.
\end{Prop}
\begin{proof}
Let $a\in A_{n,G}^{\heartsuit}(C)$ and let $X_{a}=\cup_{\lambda}X_{a,\lambda}$ be the decomposition of $X_a$ into its irreducible components, each $X_{a,\lambda}$ being integral. We use the notations of \S \ref{SS-SSHB}. Suppose that $K_a$ strictly contains $G$, and so each $K_{a,\lambda}$ strictly contains $G$. Let $\psi:\tilde{C}\rightarrow C$ be the Galois covering corresponding to $G$. We have $G=\Ker(\psi^{\vee})$. According to the proof of \cite[Proposition 11.4.3]{BL}, $K_{a,\lambda}\supset G$ if and only if $\tilde{\pi}_{a,\lambda}:\tilde{X}_{a,\lambda}\rightarrow C$ factors as $\smash{\tilde{X}_{a,\lambda}\rightarrow \tilde{C}\stackrel{\psi}{\rightarrow}C}$, so by definition $K_{a,\lambda}$ is the kernel of $$\Pic^0(C)\stackrel{\psi^{\vee}}{\longrightarrow}\Pic^0(\tilde{C})\longrightarrow\Pic^0(\tilde{X}_{a,\lambda}).$$ Now $\psi^{\vee}(K_a)$ is a nontrivial finite subgroup of $\Pic^0(\tilde{C})$. Let $G'\subset\psi^{\vee}(K_a)$ be a cyclic subgroup and let $H=\psi^{\vee -1}(G')$. Write $n'_{a,\lambda}=n_{a,\lambda}/|G|$ for each $\lambda$. Since $G'$ is contained in the kernel of $\Pic^0(\tilde{C})\rightarrow\Pic^0(\tilde{X}_{a,\lambda})$, we have $G'\subset\Pic^0(\tilde{C})[n'_{a,\lambda}]$. Therefore $|G'|$ divides $n'_{a,\lambda}$. We deduce that $|H|$ divides $n_{a,\lambda}$. Now $|H|$ divides $n$ since $n=\sum_{\lambda}n_{a,\lambda}$. Since $H\subset K_a$, we have $a\in A_{n,H}(C)$.
\end{proof}
Recall a notation in \cite{HP}. For any integer $i$, let 
\begin{equation}\label{eq-[n]}
[i]: \Pic^0(C)\longrightarrow\Pic^0(C)
\end{equation}
denote the map of multiplication by $i$ (i.e., taking the tensor of $i$ copies of a given line bundle).
\begin{Rem}
The above proof does not work if $a$ corresponds to a non reduced spectral curve. For simplicity, suppose that $X_a$ is irreducible and denote by $X_a^{\red}$ the underlying integral spectral curve (See \cite[Lemma 2.4]{HP}). Denote by $k$ the multiplicity of $X_{a}^{\red}$ in $X_a$ and by $l$ the degree of the spectral cover $X_a^{\red}\rightarrow C$, so that $n=kl$. The spectral curve $X_a^{\red}$ corresponds to a unique element $a_{\red}\in A_{l}(C)$. According to the proof of \cite[Theorem 5.3]{HP}, we have $G\subset K_a$ if and only if $G_{\red}\subset K_{a_{\red}}$, where $G_{\red}$ is the image of $G$ under $[k]$. However, even if $G$ is strictly contained in $K_a$, there is no guarantee that $G_{\red}$ is strictly contained in $K_{a_{\red}}$. We are unable to control which subgroups $K_a$ might contain.
\end{Rem}

Write $\check{A}_{n,G}^{\heartsuit}(C)=\check{A}_{n,G}(C)\cap A_n^{\heartsuit}(C)$.
\begin{Cor}\label{K=G}
Let $G\subset\Gamma$ be a subgroup such that $|G|$ divides $n$, and let $a\in \check{A}_{n,G}^{\heartsuit}(C)$. Suppose that for any subgroup $H\subset\Gamma$ strictly containing $G$ such that $|H|$ divides $n$, we have $a\notin \check{A}_{n,H}(C)$. Then $K_a=G$.
\end{Cor}

\begin{Prop}\label{GsubSt}
Let $G\subset\Gamma$ be a subgroup such that $|G|$ divides $n$. For any $a\in \check{A}_{n,G}(C)$, there exists $x\in \check{h}_C^{-1}(a)$ such that $G\subset\Stab_{\Gamma}(x)$.
\end{Prop}
\begin{proof}
Denote by $\check{M}_{n,G}(C)$ the fixed-point locus of $G$ in $\check{M}_n(C)$. Then we have the following commutative diagram
\begin{equation}\label{SLdiag}
\begin{tikzcd}[row sep=2.5em, column sep=2em]
\check{M}_m(\tilde{C},G) \arrow[r, "\check{\mathbf{q}}_G"] \arrow[d, "\check{h}_{\tilde{C},G}"'] & \check{M}_{n,G}(C) \arrow[r, hookrightarrow] & \check{M}_n(C) \arrow[d, "\check{h}_{C}"] \\
\check{A}_m(\tilde{C},G) \arrow[r, "\check{q}_G"'] & \check{A}_{n,G}(C) \arrow[r, hookrightarrow] & \check{A}_n(C)
\end{tikzcd}
\end{equation}
where $\check{h}_{\tilde{C},G}$ is the restriction of $h_{\tilde{C}}$. This is essentially (\ref{CD-h-q}). We only need to notice that $\Ima\check{\mathbf{q}}_G\subset\check{M}_{n,G}(C)$, which follows from the projection formula.

Let $a\in \check{A}_{n,G}(C)$. By definition, there exists $a'\in\check{A}_{m}(\tilde{C},G)$ that maps to $a$. We will show below that the Hitchin map $\check{h}_{\tilde{C},G}$ is surjective. Let $x'\in\check{h}_{\tilde{C},G}^{-1}(a')$. Put $x=\check{\mathbf{q}}_G(x')$. By the commutativity of the above diagram, $\check{h}_{C}$ maps $x$ to $a$. By the definition of $\check{M}_{n,G}(C)$, we have $G\subset\Stab_{\Gamma}(x)$, and so we are done.

Finally, let us show that $\check{h}_{\tilde{C},G}$ is surjective. We follow the argument of \cite[Proposition 2.4.9]{dC}. The restriction of $h_{\tilde{C}}$ to $\check{A}_m(\tilde{C},G)$ is a surjective morphism, whose domain consists of Higgs bundles $(E,\Phi)$ on $\tilde{C}$ such that $\psi_{\ast}\Phi$ has zero trace. We only need to show that each fibre contains some Higgs bundle  satisfying $\det(\psi_{\ast}E)\cong\mathcal{O}_C$. We consider the action of $\Pic^0(C)$ on $M_m(\tilde{C})$. If $L\in\Pic^0(C)$, then it sends $(E,\Phi)$ to $(E\otimes\psi^{\ast}L,\Phi\otimes\Id)$. This action preserves the fibres. The natural morphism $M_m(\tilde{C})\rightarrow\Pic^0(C)$, sending $(E,\Phi)$ to $\det(\psi_{\ast}E)$ is equivariant with respect to the actions of $\Pic^0(C)$, where $L\in\Pic^0(C)$ sends $M\in\Pic^0(C)$ to $M\otimes L^{n}$. Since $[n]$ is surjective onto $\Pic^0(C)$ (see (\ref{eq-[n]}) for the notation), we can always modify $E$ by some $\psi^{\ast}L$ with $L\in\Pic^0(C)$ in such a way that $E$ has the required determinant.
\end{proof}

\begin{Thm}\label{M-T}
The following assertions hold:
\begin{itemize}
\item[(i)] For any $a\in \check{A}_n^{\heartsuit}(C)$, and any $x\in 
\check{h}_C^{-1}(a)$, we have $\Stab_{\Gamma}(x)\subset K_a$.
\item[(ii)] For any $a\in \check{A}^{\heartsuit}_n(C)$, there exists $x\in 
\check{h}_C^{-1}(a)$ such that $\Stab_{\Gamma}(x)=K_a$.
\end{itemize}
\end{Thm}
\begin{proof}
(i). Let $a\in\check{A}_n^{\heartsuit}(C)$, and let $\pi:X_a\rightarrow C$ be the corresponding spectral curve. By assumption, $X_a$ is reduced. Let $X_a=\bigcup_{\lambda} X_{a,\lambda}$ be the decomposition of $X_a$ into irreducible components. For each $\lambda$, let $\nu_{\lambda}:\tilde{X}_{a,\lambda}\rightarrow X_{a,\lambda}$ be the normalisation map. Let $\tilde{X}_a=\bigsqcup_{\lambda}\tilde{X}_{a,\lambda}$ be the disjoint union. Let $\nu:\tilde{X}_a\rightarrow X_a$ be the union of the maps $\nu_{\lambda}$. Write $\tilde{\pi}=\pi\circ \nu$. Recall that $\check{h}^{-1}(a)$ consists of torsion-free sheaves on $X_a$. It suffices to show that for any torsion-free sheaf $M$ of generic rank 1 on $X_a$ and any invertible sheaf $L$ on $C$ such that $M\otimes\pi^{\ast}L\cong M$, we have $\tilde{\pi}^{\ast}L\cong \mathcal{O}_{\tilde{X}_a}$, i.e. $L\in K_a$.

We have $\nu^{\ast}M\otimes\tilde{\pi}^{\ast}L\cong \nu^{\ast}M$. In general, $\nu^{\ast}M$ is not torsion-free, since $\nu$ may not be flat. Denote by $(\nu^{\ast}M)_{tor}$ the torsion subsheaf. We have $(\nu^{\ast}M)_{tor}\otimes\tilde{\pi}^{\ast}L\cong (\nu^{\ast}M)_{tor}$, and so $(\nu^{\ast}M/(\nu^{\ast}M)_{tor})\otimes\tilde{\pi}^{\ast}L\cong \nu^{\ast}M/(\nu^{\ast}M)_{tor}$. Now $\nu^{\ast}M/(\nu^{\ast}M)_{tor}$ is invertible since $\tilde{X}_a$ is smooth. We deduce that $\tilde{\pi}^{\ast}L\cong\mathcal{O}_{\tilde{X}_a}$.

(ii). Let $G\subset\Gamma$ be a subgroup such that $|G|$ divides $n$. Put $$\check{A}^{\circ}_{n,G}(C)=\check{A}_{n,G}^{\heartsuit}(C)\setminus\big(\bigcup_H\check{A}_{n,H}^{\heartsuit}(C)\big),$$where $H\subset\Gamma$ runs over subgroups strictly containing $G$ such that $|H|$ divides $n$. Then $\check{A}^{\circ}_{n,G}(C)$ is nonempty for dimension reason. In view of Corollary \ref{K=G}, we have a disjoint union $$\check{A}^{\heartsuit}_n(C)=\bigsqcup_{G}\check{A}^{\circ}_{n,G}(C),$$where $G\subset\Gamma$ runs over subgroups such that $|G|$ divides $n$. By Proposition \ref{GsubSt} and part (i) of the theorem, for any $a\in\check{A}^{\circ}_{n,G}(C)$, there exists $x\in\check{h}^{-1}_C(a)$ such that $$G\subset\Stab_{\Gamma}(x)\subset K_a,$$while $G=K_a$ according to Corollary \ref{K=G}.
\end{proof}
\begin{Cor}\label{Cor-M-T}
Let $a\in \check{A}_n^{\heartsuit}(C)$, then $|K_a|$ divides $n$. For any $x\in 
\check{h}_C^{-1}(a)$, $|\Stab_{\Gamma}(x)|$ divides $n$.
\end{Cor}
\begin{Rem}
The proof of Theorem \ref{M-T} (ii) shows more than what is stated in the theorem: it shows that $K_a$ (and so $\Stab_{\Gamma}(x)$) comes from some \'etale Galois covering of $C$ whose degree divides $n$. This is necessary for our approach to the fractional Hecke eigenproperty in later sections. In fact, there is a short proof without giving much information on $K_a$. Consider $J_a:=\Pic^0(X_a)$, the variety of line bundles that have degree 0 on every irreducible components of $X_a$. It is a commutative connected group variety, and acts on the $\GL_n$-Hitchin fibre $h_C^{-1}(a)$ by tensorisation. There is a short exact sequence of group varieties
$$
0\longrightarrow J_a^{\text{aff}}\longrightarrow J_a\longrightarrow J_a^{\text{proj}}\longrightarrow 0,
$$
where $J_a^{\text{aff}}$ is an affine subgroup and $J_a^{\text{proj}}$ is an abelian variety. By the Borel fixed point theorem, there is a $J_a^{\text{aff}}$-fixed point, say $x$, in $h_C^{-1}(a)$. Modifying $x$ by some element of $J_a$ if necessary, we may assume that $x$ lies in the $\SL_n$-Hitchin fibre $\check{h}_C^{-1}(a)$. Let us show that $\Stab_{\Gamma}(x)=K_a$. The pullback via $\nu$ induces an isomorphism $J_a^{\text{proj}}\cong\Pic^0(\tilde{X_a})$. In particular, the kernel of $\nu^{\vee}$ is exactly $J_a^{\text{aff}}$. By definition, $K_a$ is the subgroup of $\Gamma$ that is mapped into $J_a^{\text{aff}}$ via $\pi^{\vee}$; therefore, $K_a$ fixes $x$.
\end{Rem}
\begin{Rem}
When $x\in\check{h}^{-1}(a)$ corresponds to a stable $\SL_n$-Higgs bundle, the automorphism group of the $\PGL_n$-Higgs bundle obtained from $x$ by extension of structure group, is isomorphic to $\Stab_{\Gamma}(x)$. It is known that $\check{M}_n^{ell}(C)$ is contained in $\check{M}_n^s(C)$, the open subset of stable $\SL_n$-Higgs bundles. Therefore, the above theorem implies Theorem \ref{Thm-FW}. We may also consider the moduli space of $\SL_n$-Higgs bundles with coprime degree, then Theorem \ref{Thm-FW} holds in the larger open subset $\check{M}_n^{\heartsuit}(C)$.\footnote{I thank the referee for pointing this out.}
\end{Rem}
\begin{Rem}
\cite[Corollary 1.3]{Ho} does not seem to be correct, since it says that there are Hitchin fibres $\check{h}^{-1}_C(a)$ for $a\in \check{A}_n^{ell}(C)$ which has four irreducible components, in contradiction with the corollary above. Besides, in \cite[\S 5.2.1, \S 5.2.2]{FW} Frenkel and Witten have considered spectral covers associated to quadratic differentials that only have double zeros, and the corresponding Hitchin fibres have two irreducible components. This also contradicts \cite[Corollary 1.3]{Ho}.
\end{Rem}

\section{Compactified Prym Varieties}\label{CPV}
\subsection{Prym Varieties}\label{subsec-prym}\hfill

We give the setup of the rest of this article. We fix an integral spectral cover $\pi:X\rightarrow C$ of degree $n$. Let $\nu:\tilde{X}\rightarrow X$ be the normalisation. Write $\tilde{\pi}:=\pi\circ\nu$. Let $G=\Ker\tilde{\pi}^{\vee}$, where $\tilde{\pi}^{\vee}:J_C\rightarrow J_{\tilde{X}}$ is the pullback morphism. Let $\psi_G:C''\rightarrow C$ be the Galois covering defined by $G$, then we have $G=\Ker\psi_G^{\vee}$. By Lemma \ref{Lem-AnG}, the map $\tilde{\pi}$ factors through $C''$. The Galois group of $C''\rightarrow C$ is isomorphic to the dual group $G^{\vee}$. 

Let $H\subset G$ be a subgroup. Then $H^{\vee}$ corresponds to a Galois covering $\psi:C'\rightarrow C$ such that $\psi_G$ factors as $C''\stackrel{\psi'}{\rightarrow} C'\stackrel{\psi}{\rightarrow} C$. We have $\psi=\psi_G$ if $H=G$. By Lemma \ref{Lem-AnG}, we have a Cartesian diagram
\begin{equation}\label{CD-Za'-1}
\begin{tikzcd}[row sep=2.5em, column sep=2em]
\bigcup_{\xi\in H^{\vee}}Y^{\xi} \arrow[r] \arrow[d, swap, "\phi"] \arrow[dr, phantom, "\square"] & C' \arrow[d, "\psi"] \\
X \arrow[r, swap, "\pi"] & C,
\end{tikzcd}
\end{equation}
where $Y$ is a spectral curve over $C'$ and $\phi$ is defined to be the base change of $\psi$. By abuse of notation, the restriction of $\phi$ to $Y$ is also denoted by $\phi$. Similarly, we have $C''\times_{C'}Y\cong\bigcup_{\chi\in K^{\vee}}Z^{\chi}$, where $K:=G/H$, and $Z$ is a spectral curve over $C''$. By Lemma \ref{Lem-AnG}, the normalisation map $\nu$ factors as 
$$
\tilde{X}\stackrel{\nu''}{\longrightarrow}Z\longrightarrow X.
$$
We have a commutative diagram
\begin{equation}\label{CD-Za'-Ya''-Xa}
\begin{tikzcd}[row sep=2.5em, column sep=2em]
\tilde{X} \arrow[dr, "\nu''"] \arrow[drr, "\tilde{\pi}_G", bend left] \arrow[dddr, "\nu"', bend right] \arrow[ddr, "\tilde{\phi}'", bend right] & & \\
&Z \arrow[r, "\pi_G"] \arrow[d, "\phi'"'] & C'' \arrow[d, "\psi'"] \\
&Y \arrow[r, "\pi_H"] \arrow[d,"\phi"'] & C' \arrow[d, "\psi"]\\
&X \arrow[r, "\pi"'] & C,
\end{tikzcd}
\end{equation}
where $\phi'$, $\pi_G$ and $\pi_H$ are the restrictions of the projection maps from the fibre products, and $\tilde{\phi}'=\phi'\circ\nu''$, $\tilde{\pi}_G=\pi_G\circ\nu''$. We write:
\begin{itemize}
\item $\tilde{\pi}_H=\pi_H\circ\tilde{\phi}'$, $\pi'=\psi\circ\pi_H$, and $\pi''=\psi_G\circ\pi_G$
\end{itemize}

We also make the following simplifications of notations:
\begin{itemize}
\item For $X$ and $\pi$, write $\bar{J}=\bar{J}_X$, $J=J_X$, $\bar{P}=\bar{P}_{\pi}$, and $P=P_{\pi}$;
\item For $Y$ and $\pi'$, write $\bar{J}'=\bar{J}_Y$, $J'=J_Y$, $\bar{P}'=\bar{P}_{\pi'}$, and $P'=P_{\pi'}$;
\item For $Z$ and $\pi''$, write $\bar{J}''=\bar{J}_Z$, $J''=J_Z$,  $\bar{P}''=\bar{P}_{\pi''}$, and $P''=P_{\pi''}$.
\end{itemize}
And write
\begin{itemize}
\item $P_{C'}:=P_{\psi}$ and $P_{C''}:=P_{\psi_G}$.
\end{itemize}
These notations will be used in the rest of this article. The objects associated to $X$ and $Z$ will be thought to be fixed, while those intermediate objects associated to $Y$ depend on the choice of the subgroup $H$. Many of our results will be stated for $Y$ and $\pi'$. Of course, they include the particular case when $H=0$ and $H=G$.

The open subset $J'\subset\bar{J}'$ consisting of line bundles on $Y$ acts on $\bar{J}'$ by tensorisation. Then $\Gamma$ acts on $\bar{J}'$ via the homomorphism $\pi^{\prime\vee}:J_C\rightarrow J'$. The subvariety $\bar{P}'$ is preserved under this action, since $\Nm_{\pi'}(F\otimes\pi^{\prime\ast}L)\cong \Nm_{\pi'}(F)\otimes L^{\otimes n}\cong\Nm_{\pi'}(F)$ for any torsion-free rank one sheaf $F$ on $Y$ and any $L\in\Gamma$. The group $\Gamma$ also acts on $J_C$ by multiplication. For any $\gamma\in\Gamma$, we let $\gamma$ acts on $\bar{P}'\times J_C$ as $\gamma\times\gamma^{-1}$.
\begin{Lem}\label{Lem-JCP=barJ}
The action of $\Gamma$ on $\bar{P}'\times J_C$ is free, and there is an isomorphism of algebraic varieties:
$$
(\bar{P}'\times J_C)/\Gamma\lisom\bar{J}'.
$$
\end{Lem}
\begin{proof}
For simplicity, we only consider the $\mathbb{C}$-points of these varieties. The lemma will follow once we notice that the arguments work in families (of torsion-free sheaves), using the modular descriptions of (compactified) Jacobians and compactified Prym varieties as in \S \ref{Sub-Jacob} and \S \ref{Sub-Prym}.

Let $F$ be a torsion-free sheaf on $Y$ representing a point of $\bar{P}'$ and $L$ a degree 0 line bundle on $C$. Then $F\otimes\pi^{\prime\ast}L$ defines a point of $\bar{J}'$. This defines a map $\bar{P}'\times J_C\rightarrow\bar{J}'$. We first check that it is surjective. Let $F\in\bar{J}'$ be a torsion-free sheaf on $Y$. Since the multiplication map $[n]$ is surjective on $J_C$, we may find a line bundle $L$ such that $L^{\otimes n}\cong\Nm_{\pi'}(F)^{-1}$. Then $F_1:=F\otimes\pi^{\prime\ast}L$ lies in $\bar{P}'$. Therefore, $(F_1,L^{-1})$ maps to $F$. Now we check injectivity. Let $(F_1,L_1)$ and $(F_2,L_2)$ be elements of $\bar{P}'\times J_C$, such that $F_1\otimes\pi^{\prime\ast}L_1\cong F_2\otimes\pi^{\prime\ast}L_2$. Put $L=L_2L_1^{-1}$. Then $F_1\cong F_2\otimes\pi^{\prime\ast}L$. Applying $\Nm_{\pi'}$ to both sides, we deduce that $L^{\otimes n}\cong\mathcal{O}_C$ (i.e., $L\in\Gamma$). We see that $(F_1,L_1)$ and $(F_2,L_2)$ lie in the same $\Gamma$-orbit.
\end{proof}

\subsection{Connected components of compactified Prym varieties}\label{SS-CCPV}\hfill

\begin{Lem}\label{Lem-conn-comp-P'}
There are natural isomorphisms of groups:
\begin{itemize}
\item[(i)] $\pi_0(P_{C''})\cong G^{\vee}$.
\item[(ii)] $\pi_0(P')\cong\pi_0(P_{\tilde{\pi}})\cong G^{\vee}$.
\end{itemize}
\end{Lem}
\begin{proof}
The first part follows from Lemma \ref{Lem-reform-HP2.1}. By Lemma \ref{Lem-AnG}, the normalisation $\tilde{X}$ of $X$ is also the normalisation of $Y$. By \cite[Lemma 4.1 (4)]{HP}, we have $\pi_0(P')\cong \pi_0(P_{\tilde{\pi}})$,
which is identified with $G^{\vee}$ by Lemma \ref{Lem-reform-HP2.1}.
\end{proof}

\begin{Lem}\label{Lem-conn-comp-barP'}
There are natural isomorphisms of groups: 
\begin{itemize}
\item[(i)] $\pi_0(P_{C'})\cong H^{\vee}$.
\item[(ii)] $\pi_0(\bar{P}')\cong H^{\vee}$.
\end{itemize}
\end{Lem}
\begin{proof}
The first part is given by Lemma \ref{Lem-reform-HP2.1}. By definition, $\bar{P}'$ is the inverse image of $0$ under $\Nm_{\pi'}$. By Lemma \ref{Lem-Func-Nm}, $\Nm_{\pi'}$ is equal to the following map $$\bar{J}'\stackrel{\Nm_{\pi_H}}{\longrightarrow} J_{C'}\stackrel{\Nm_{\psi}}{\longrightarrow} J_C.$$Since $Y\rightarrow C'$ is a spectral cover, the fibre of $\Nm_{\pi_H}$ is isomorphic to a Hitchin fibre of an $\SL_m$-Hitchin fibration, with $m=n/|H|$. By \cite[Proposition 2.4.9]{dC}, such a Hitchin fibre is connected. Now $\Nm_{\psi}^{-1}(0)=P_{C'}$, whose connected components are parametrised by $H^{\vee}$.
\end{proof}
\begin{Rem}\label{Rem-piH-inj}
The fact that the fibre of $\Nm_{\pi_H}$ is connected implies that $\pi_H^{\vee}$ is injective. If $\pi_H^{\vee}$ is not injective, then the arguments of \cite[Proposition 11.4.3]{BL} show that $\pi_H$ must factor through some nontrivial \'etale Galois covering of $C'$. But then the fibre of $\Nm_{\pi_H}$ can not be connected. Contradiction.
\end{Rem}
\begin{Rem}
It follows from Lemma \ref{Lem-barP'=barPG} and Lemma \ref{Lem-conn-comp-barP'} (ii) that there is an isomorphism $\pi_0(\bar{P}^H)\cong H^{\vee}$.
\end{Rem}

The Poincar\'e bundle $\mathcal{P}_C$ on $J_C\times J_C$ induces an isomorphism $\rho:J_C\isom\Pic^0_{J_C}$, which restricts to an isomorphism between the $n$-torsion points $\rho:\Gamma=J_C[n]\isom\Pic^0_{J_C}[n]$. Denote by $\Gamma^{\vee}$ the group of multiplicative characters of $\Gamma$, then we have a natural identification $\Pic^0_{J_C}[n]\cong\Gamma^{\vee}$. Indeed, the multiplication by $n$ defines a Galois covering $[n]:J_C\rightarrow J_C$ with Galois group $\Gamma$, and the kernel of the dual $[n]^{\vee}:\Pic^0_{J_C}\rightarrow\Pic^0_{J_C}$ is identified with the isotypic components of $[n]_{\ast}\mathcal{O}_{J_C}$, which are in bijection with the irreducible characters of $\Gamma$.

Since $[n]=\Nm_{\pi'}\circ\pi^{\prime\vee}$ according to \S \ref{Sub-Prym} (ii), we have $\pi^{\prime\vee}(\Gamma)\subset \bar{P}'$. By abuse of notation, the induced morphism $\Gamma\rightarrow\pi_0(\bar{P}')$ is also denoted by $\pi^{\prime\vee}$.
\begin{Lem}\label{Lem-GaV-HV}
The following diagram commutes
\begin{equation}\label{CD-Lem-GaV-HV}
\begin{tikzcd}[row sep=2.5em, column sep=2em]
\Gamma \arrow[r, "\pi^{\prime\vee}"] \arrow[d, "\rotatebox{90}{\(\sim\)}", "\rho"'] & \pi_0(\bar{P}') \arrow[d, "\rotatebox{90}{\(\sim\)}"] \\
\Gamma^{\vee} \arrow[r] & H^{\vee},
\end{tikzcd}
\end{equation}
where the right vertical arrow is given by Lemma \ref{Lem-conn-comp-barP'} (ii), and the bottom arrow is the restriction of a character of $\Gamma$ to $H$. 
\end{Lem}
\begin{proof}
Note that $[n]=\Nm_{\pi'}\circ\pi^{\prime\vee}$, and that $\pi^{\prime\vee}:J_C\rightarrow\bar{J}'$ factors as
$$
J_C\stackrel{\psi^{\vee}}{\longrightarrow}\psi^{\vee}(J_C)\stackrel{i}{\hookrightarrow} J_{C'}\stackrel{\pi_H^{\vee}}{\longrightarrow}\bar{J}'.
$$
Therefore $[n]$ factors as 
$$
J_C\stackrel{\psi^{\vee}}{\longrightarrow}\psi^{\vee}(J_C)\stackrel{j}{\longrightarrow} J_C,
$$
where $j=\Nm_{\pi'}\circ\pi_H^{\vee}\circ i$. Note that $\psi^{\vee}$ is an $H$-Galois covering. Applying Lemma \ref{Lem-G1G2}, taking 
$$
\begin{tikzcd}[row sep=2.5em, column sep=2em]
J_C \arrow[d, "\psi^{\vee}"'] \arrow[r, "\Id"] & J_C \arrow[d, "\lbrack n\rbrack"] \\
\psi^{\vee}(J_C) \arrow[r, "j"'] & J_C,
\end{tikzcd}
$$
for (\ref{CD-Lem-G1G2-0}), we see that the restriction map $\Gamma^{\vee}\rightarrow H^{\vee}$ is the same as $j^{\vee}$. It remains to show that the diagram (\ref{CD-Lem-GaV-HV}) commutes if we replace the bottom arrow by $j^{\vee}$.

The right hand side vertical arrow in (\ref{CD-Lem-GaV-HV}) is the composition $\pi_0(\bar{P}')\cong \pi_0(P_{C'})\cong H^{\vee}$, where the first isomorphism is induced by $\Nm_{\pi_H}$ and the second isomorphism is given by Lemma \ref{Lem-conn-comp-barP'} (i). Therefore, for any $\gamma\in\Gamma$, the image of $\pi^{\prime\vee}(\gamma)$ in $H^{\vee}$ is $i^{\vee}\circ\rho_{C'}\circ\Nm_{\pi_H}\circ\pi^{\prime\vee}(\gamma)$. We must show that 
\begin{equation}\label{eq-Lem-GaV-HV}
j^{\vee}\circ\rho_C(\gamma)=i^{\vee}\circ\rho_{C'}\circ\Nm_{\pi_H}\circ\pi^{\prime\vee}(\gamma).
\end{equation}
The left hand side of the equation is equal to $i^{\vee}\circ(\pi_H^{\vee})^{\vee}\circ\Nm_{\pi'}^{\vee}\circ\rho_C(\gamma)$. Applying (\ref{CD-PicAJ-XaC}) to $\pi':Y\rightarrow C$, we get 
$$
\Nm_{\pi'}^{\vee}\circ\rho_C=\rho'\circ\pi^{\prime\vee},
$$
where $\rho'$ is the isomorphism (\ref{Eq-rho}) for $\bar{J}'$. Applying (\ref{CD-piVV-Nm}) to $\pi_H:Y\rightarrow C'$, we get
$$
(\pi_H^{\vee})^{\vee}\circ\rho'=\rho_{C'}\circ\Nm_{\pi_H}.
$$
The above two equations combined give the desired equation (\ref{eq-Lem-GaV-HV}).
\end{proof}

For any $\xi\in H^{\vee}$, we will denote by $\bar{P}'_{\xi}$ the corresponding connected component of $\bar{P}'$. In particular, the identity component is $\bar{P}'_0$. For any $\xi\in H^{\vee}$, we will denote by $\Gamma_{\xi}\subset\Gamma$ the subset that is mapped into $\bar{P}'_{\xi}$ under $\pi^{\prime\vee}$. In particular $\Gamma_0$ is the subgroup that is mapped into $\bar{P}'_0$. For any $\chi\in G^{\vee}$, we will denote by $P'_{\chi}$ the corresponding connected component of $P'$.
\begin{Rem}\label{Rem-JCP=barJ-0}
It follows from Lemma \ref{Lem-JCP=barJ} that there is an isomorphism of algebraic varieties
$$
(\bar{P}'_0\times J_C)/\Gamma_0\lisom\bar{J}'.
$$
\end{Rem}
\begin{Rem}\label{Rem-gamma_chi}
It follows from Lemma \ref{Lem-GaV-HV} that if $\gamma\in\Gamma_{\xi}$, then $\rho(\gamma)$ is a character of $\Gamma$ that restricts to $\xi$ on $H$.
\end{Rem}

\begin{Prop}\label{Prop-IndGH}
Let $\xi\in H^{\vee}$ and $\chi\in G^{\vee}$. Then $P'_{\chi}$ is contained in $\bar{P}'_{\xi}$ if and only if $\chi\in\Ind^G_H\xi$ (i.e., $\chi$ is an irreducible component of the induced character $\Ind^G_H\xi$).
\end{Prop}
\begin{proof}
Denote by $P_{C'',\chi}$ the connected component of $P_{C''}$ corresponding to $\chi$, and by $P_{C',\xi}$ the connected component of $P_{C'}$ corresponding to $\xi$. The proposition follows from the following two assertions:
\begin{itemize}
\item[(1)] $P'_{\chi}\subset\bar{P}'_{\xi}$ if and only if $\Nm_{\psi'}(P_{C'',\chi})\subset P_{C',\xi}$.
\item[(2)] $\Nm_{\psi'}(P_{C'',\chi})\subset P_{C',\xi}$ if and only if $\Res^G_H\chi=\xi$. 
\end{itemize}
Note that $\Res^G_H\chi=\xi$ if and only if $\chi\in\Ind^G_H\xi$ by Frobenius reciprocity.

Recall the arrows in (\ref{CD-Za'-Ya''-Xa}). We claim that the following diagram 
\begin{equation}\label{CD-phia''-NmNm}
\begin{tikzcd}[row sep=2.5em, column sep=2em]
P' \arrow[r, "\tilde{\phi}^{\prime\vee}"] \arrow[dr, swap, "\Nm_{\pi_H}"] & P_{\tilde{\pi}} \arrow[r, "\Nm_{\tilde{\pi}_G}"] \arrow[d] & P_{C''} \arrow[dl, "\Nm_{\psi'}"] \\
&P_{C'} &
\end{tikzcd}
\end{equation}
commutes, and that if we identify the component groups of $P'$, $P_{\tilde{\pi}}$ and $P_{C''}$ with $G^{\vee}$ using the isomorphisms in Lemma \ref{Lem-conn-comp-P'}, then $\tilde{\phi}^{\prime\vee}$ and $\Nm_{\tilde{\pi}_G}$ induce the identity maps between them. Then the assertion (1) follows. Indeed, let $p\in P'$, then $p\in P'_{\chi}$ if and only if $\Nm_{\tilde{\pi}_G}\circ\tilde{\phi}^{\prime\vee}(p)\in P_{C'',\chi}$, while $p\in \bar{P}'_{\xi}$ if and only if $\Nm_{\pi_H}(p)\in P_{C',\xi}$.

To get the commutative diagram (\ref{CD-phia''-NmNm}), we consider the following diagram 
\begin{equation}\label{CD-phia''-NmNm-Jac}
\begin{tikzcd}[row sep=2.5em, column sep=2em]
J' \arrow[r, "\tilde{\phi}^{\prime\vee}"] \arrow[dr, swap, "\Nm_{\pi_H}"] & J_{\tilde{\pi}} \arrow[r, "\Nm_{\tilde{\pi}_G}"] \arrow[d] & J_{C''} \arrow[dl, "\Nm_{\psi'}"] \\
&J_{C'} &,
\end{tikzcd}
\end{equation}
where the vertical arrow is $\Nm_{\tilde{\pi}_H}$ (the map $\tilde{\pi}_H$ is defined after the diagram (\ref{CD-Za'-Ya''-Xa})). Note that the morphism $\tilde{\pi}:\tilde{X}\rightarrow C$ factors as
$$
\tilde{X}\stackrel{\tilde{\pi}_G}{\longrightarrow}C''\stackrel{\psi_G}{\longrightarrow}C.
$$
We have $\tilde{\pi}_H=\psi'\circ\tilde{\pi}_G$. By \cite[Lemma 3.4]{HP}, the left hand side of (\ref{CD-phia''-NmNm-Jac}) commutes.  By Lemma \ref{Lem-Func-Nm}, the right hand side of (\ref{CD-phia''-NmNm-Jac}) commutes. Now all of the Jacobians in (\ref{CD-phia''-NmNm-Jac}) admit norm maps to $J_C$. We can then apply Lemma \ref{Lem-Func-Nm} and \cite[Lemma 3.4]{HP} to these maps, and conclude that (\ref{CD-phia''-NmNm-Jac}) restricts to the commutative diagram (\ref{CD-phia''-NmNm}). 

Since $\tilde{X}$ is also the normalisation of $Y$, we have that $\tilde{\phi}^{\prime\vee}$ induces a natural identification of the connected components by \cite[Lemma 4.1]{HP}. To show that $\Nm_{\tilde{\pi}_G}$ induces the identity map between the component groups of Prym varieties, we consider the following commutative diagram
\begin{equation}\label{CD-Jpi-JC''-Vee}
\begin{tikzcd}[row sep=2.5em, column sep=2em]
J_{C''} \arrow[r, "\tilde{\pi}_G^{\vee}"] & J_{\tilde{\pi}} \\
\psi_G^{\vee}(J_C) \arrow[u, hook] \arrow[r, "\tilde{\pi}_G^{\vee}"', "\sim"] & \tilde{\pi}^{\vee}(J_C) \arrow[u, hook] \\
J_C \arrow[u, "\psi_G^{\vee}"] \arrow[r, "="'] & J_C \arrow[u, "\tilde{\pi}^{\vee}"'].
\end{tikzcd}
\end{equation}
The map $\tilde{\pi}_G^{\vee}$ restricts to an isomorphism in the second row because $\psi_G^{\vee}$ and $\tilde{\pi}^{\vee}$ have the same kernel. Note that the two vertical arrows in the lower half of the diagram are $G$-Galois coverings. Let $p\in P_{\tilde{\pi}}$. Let $\chi\in G^{\vee}$ and denote by $L_{\chi}$ the line bundle on $\tilde{\pi}^{\vee}(J_C)$ corresponding to $\chi$. By abuse of notation, the pullback of $L_{\chi}$ to $\psi_G^{\vee}(J_C)$ is also denoted by $L_{\chi}$. By Lemma \ref{Lem-reform-HP2.1}, $p$ lies in $P_{\tilde{\pi},\chi}$, the connected component of $P_{\tilde{\pi}}$ corresponding to $\chi$, if and only if $\rho_{\tilde{X}}(p)|_{\tilde{\pi}^{\vee}(J_C)}\cong L_{\chi}$. Applying (\ref{CD-piVV-Nm}) to $\tilde{X}\rightarrow C''$, we see that $\rho_{C''}(\Nm_{\tilde{\pi}_G}(p))\cong\tilde{\pi}_G^{\vee\ast}\rho_{\tilde{X}}(p)$. The commutativity of the upper half of (\ref{CD-Jpi-JC''-Vee}) shows that $\rho_{C''}(\Nm_{\tilde{\pi}_G}(p))|_{\psi_G^{\vee}(J_C)}\cong L_{\chi}$. Using Lemma \ref{Lem-reform-HP2.1} again, we see that $\Nm_{\tilde{\pi}_G}$ induces the identity map between the component groups. 

Now we prove (2). Since $\Ker(\psi^{\vee})=H$, we have that $J_C\rightarrow\psi^{\vee}(J_C)$ is an $H$-Galois covering. The line bundles on $\psi^{\vee}(J_C)$ that pulls back to the trivial bundle are parametrised by $H^{\vee}$. We denote by $M_{\xi}$ the line bundle on $\psi^{\vee}(J_C)$ corresponding to $\xi\in H^{\vee}$. The $G$-Galois covering $J_C\rightarrow\psi_G^{\vee}(J_C)$ factors as
$$
J_C\stackrel{\psi^{\vee}}{\longrightarrow}\psi^{\vee}(J_C)\stackrel{\psi^{\prime\vee}}{\longrightarrow}\psi_G^{\vee}(J_C).
$$ 
Let $p\in P_{C''}$. By Lemma \ref{Lem-reform-HP2.1}, $p$ lies in $P_{C'',\chi}$ if and only if $\rho_{C''}(p)|_{\psi_G^{\vee}(J_C)}\cong L_{\chi}$, and $\Nm_{\psi'}(p)$ lies in $P_{C',\xi}$ if and only if $\rho_{C'}(\Nm_{\psi'}(p))|_{\psi^{\vee}(J_C)}\cong M_{\xi}$. But $\rho_{C'}(\Nm_{\psi'}(p))$ is isomorphic to $(\psi^{\prime\vee})^{\ast}\rho_{C''}(p)$ by (\ref{CD-piVV-Nm}), which in turn is isomorphic to $(\psi^{\prime\vee})^{\ast}L_{\chi}$. We see that $\Nm_{\psi'}(p)\in P_{C',\xi}$ if and only if $(\psi^{\prime\vee})^{\ast}L_{\chi}\cong M_{\xi}$. The line bundle $L_{\chi}$ is equivalent to $\mathcal{O}_{J_C}$ equipped with a $G$-equivariant structure, and $G$ acts on the global sections of  $\mathcal{O}_{J_C}$ via $\chi$. Since the pullback of $(\psi^{\prime\vee})^{\ast}L_{\chi}$ to $J_C$ is the trivial line bundle, the line bundle $(\psi^{\prime\vee})^{\ast}L_{\chi}$ is equivalent to $\mathcal{O}_{J_C}$ equipped with an $H$-equivariant structure, and the $H$-action on the global sections of  $\mathcal{O}_{J_C}$ is simply restricting $\chi$ to $H$. We conclude that $(\psi^{\prime\vee})^{\ast}L_{\chi}\cong M_{\xi}$ if and only if $\xi=\Res^G_H\chi$.
\end{proof}

\begin{Lem}\label{Lem-phiaV}
The morphism $\phi^{\vee}$ restricts to a surjective morphism
$$
\phi^{\vee}:P\longrightarrow P'.
$$
Moreover, if we identify $\pi_0(P)$ and $\pi_0(P')$ with $G^{\vee}$ using the isomorphisms in Lemma \ref{Lem-conn-comp-P'} (ii), then $\phi^{\vee}$ induces the identity map between them.
\end{Lem}
\begin{proof}
Since $\tilde{X}$ is the normalisation of both $X$ and $Y$, we can use \cite[Lemma 3.4]{HP} twice and show that $\phi^{\vee}$ maps $P$ to $P'$. By Lemma \ref{Lem-phiaV-J}, the morphism $\phi^{\vee}:J\rightarrow J'$ is surjective. Now let $L'$ be a line bundle on $Y$ such that $\Nm_{\pi'}(L')\cong\mathcal{O}_C$ and let $L\in J$ be such that $\phi^{\ast}L\cong L'$. The same arguments as above show that $\Nm_{\pi}(L)\cong\mathcal{O}_C$. Therefore, $\phi^{\vee}$ is also surjective between Prym varieties. Finally, $\phi^{\vee}$ induces the identity map between the component groups because it is compatible with the pullback morphisms $P\rightarrow P_{\tilde{\pi}}$ and $P'\rightarrow P_{\tilde{\pi}}$.
\end{proof}

We have in fact proved that $P\subset J$ is the inverse image of $P'$ under $\phi^{\vee}$.

\subsection{Connected components of Prym varieties via Galois actions}\hfill

Now we may slightly relax the assumption by allowing $G\subset\Gamma$ to be any subgroup. The goal is to prove Theorem \ref{Thm-Conn-Prym-New}, which will be used in \S \ref{subsec-TOE}. Lemma \ref{Lem-reform-HP2.1} gives a parametrisation of the connected components of $P_{C''}=P_{\psi_G}$.  Theorem \ref{Thm-Conn-Prym-New} will give an alternative parametrisation. For simplicity, we will write $P_0=P_{C'',0}$.

Recall that $J^i_{C''}$ is the space of degree $i$ line bundles on $C''$, and $J_{C''}=J^0_{C''}$. Let $\chi\in G^{\vee}$, regarded as a covering transformation of $\psi_G$. Consider the morphism $f_{\chi}^{(i)}:J^i_{C''}\rightarrow J_{C''}$ sending $L$ to $L^{-1}\otimes\chi^{\ast}L$. Note that $\Nm_{\psi_G}(L)=\Nm_{\psi_G}(\chi^{\ast}L)$ and so the image of $f_{\chi}^{(i)}$ is contained in $P_{C''}$. Moreover, the image of $f_{\chi}^{(i)}$ is connected since $J^i_{C''}$ is connected. Therefore $\Ima f_{\chi}^{(i)}$ is contained in $P_{C'',\chi'}$ for some $\chi'\in G^{\vee}$.
\begin{Lem}\label{Lem-auto-sigma-0}
Suppose that $i=0$. Then $\chi'$ is the trivial character for any $\chi\in G^{\vee}$.
\end{Lem}
\begin{proof}
For any $\chi$, the trivial line bundle is contained in $\Ima f^{(0)}_{\chi}$.
\end{proof}
\begin{Lem}\label{Lem-auto-sigma-1}
Suppose that $i=-1$. Then the map $\chi\mapsto\chi'$ defines a homomorphism $\sigma:G^{\vee}\rightarrow G^{\vee}$.
\end{Lem}
\begin{proof}
Let $\chi_1$, $\chi_2\in G^{\vee}$ and $L\in J^{-1}_{C''}$, then $L^{-1}\otimes\chi_1^{\ast}L$ and $L^{-1}\otimes\chi_2^{\ast}L$ lie in $P_{C'',\chi_1'}$ and $P_{C'',\chi_2'}$ respectively. We have 
$$
L^{-1}\otimes(\chi_1\chi_2)^{\ast}L=L^{-1}\otimes\chi_2^{\ast}L\otimes\chi_2^{\ast}(L^{-1}\otimes\chi_1^{\ast}L).
$$
By Lemma \ref{Lem-auto-sigma-0}, $F^{-1}\otimes\chi^{\ast}F$ lies in $P_{0}$ for any $\chi\in G^{\vee}$ and any $F\in J_{C''}$. We deduce that $\chi_2^{\ast}(L^{-1}\otimes\chi_1^{\ast}L)$ and $L^{-1}\otimes\chi_1^{\ast}L$ lie in the same connected component of $P_{C''}$. Therefore $L^{-1}\otimes(\chi_1\chi_2)^{\ast}L$ lies in $P_{C'',\chi_1'\chi_2'}$. 
\end{proof}

We will show that the homomorphism $\sigma$ as in the above lemma is in fact the identity map on $G^{\vee}$. To this end, we will construct a Galois covering on which $\pi_0(P_{C''})$ naturally acts, and then apply Lemma \ref{Lem-G1G2}. The Galois group $G^{\vee}$ acts naturally on $J^{-1}_{C''}$ by pulling back line bundles. It also acts on $J_{C''}/P_0$ by multiplication. Indeed, we may identify $\pi_0(P_{C''})\cong P_{C''}/P_0$ with $G^{\vee}$, and thus regard $G^{\vee}$ as a subgroup of $J_{C''}/P_0$.  Let $\chi\in G^{\vee}$ and let $L_0\in J^{-1}_{C''}$. Then the translation of $J_{C''}$ by $L_0^{-1}\otimes\chi^{\ast}(L_0)$ induces a morphism $T_{\sigma(\chi)}:J_{C''}/P_0\rightarrow J_{C''}/P_0$, which is none other than the translation by $\sigma(\chi)\in\pi_0(P_{C''})$.

Fix $M\in J^{1}_{C''}$. We have an isomorphism $\tau_M:J^{-1}_{C''}\isom J_{C''}$ sending $F\in J^{-1}_{C''}$ to $F\otimes M$. Its composition with the natural projection $J_{C''}\rightarrow J_{C''}/P_0$ will be denoted by $\bar{\tau}_{M}$. Recall that $\bs\alpha_{C'',-1}:C''\rightarrow J_{C''}^{-1}$ is the Abel-Jacobi map of degree -1.
\begin{Lem}\label{Lem-tauM-equiv}
For any $\chi\in G^{\vee}$, we have a commutative diagram
\begin{equation}
\begin{tikzcd}[row sep=2.5em, column sep=4em]
C'' \arrow[r, "\bs\alpha_{C'',-1}"] \arrow[d, "\chi^{-1}"'] & J^{-1}_{C''} \arrow[d, "\chi^{\vee}"] \arrow[r, "\bar{\tau}_M"] & J_{C''}/P_0 \arrow[d, "T_{\sigma(\chi)}"]\\
C'' \arrow[r, "\bs\alpha_{C'',-1}"'] & J^{-1}_{C''} \arrow[r, "\bar{\tau}_M"'] & J_{C''}/P_0,
\end{tikzcd}
\end{equation}
where $\chi$ is regarded as a covering transformation of $C''$, and $\chi^{\vee}$ denotes the pullback morphism.
\end{Lem}
\begin{proof}
The left hand side of the diagram obviously commutes. Let $L\in J^{-1}_{C''}$. Then
\begingroup
\allowdisplaybreaks
\begin{align*}
\bar{\tau}_M\circ\chi^{\vee}(L)&=\chi^{\ast}L\otimes M\text{ $\mod P_0$,}\\
T_{\sigma(\chi)}\circ\bar{\tau}_M&= L\otimes (L_0^{-1}\otimes\chi^{\ast}L_0)\otimes M\text{ $\mod P_0$.}
\end{align*}
\endgroup 
But the difference between $L^{-1}\otimes\chi^{\ast}L$ and $L_0^{-1}\otimes\chi^{\ast}L_0$ lies in $P_0$, so the right hand side also commutes.
\end{proof}

The norm map $\Nm_{\psi_G}:J_{C''}\rightarrow J_C$ factors through $J_{C''}/P_0$, and we denote by $\overline{\Nm}_{\psi_G}:J_{C''}/P_0\rightarrow J_C$ the induced morphism. Note that $\overline{\Nm}_{\psi_G}$ is a $G^{\vee}$-Galois covering. Write $M_C=\Nm_{\psi_G}(M)$, which defines the translation morphism $\tau_{M_C}:J^{-1}_C\rightarrow J_C$.
\begin{Lem}\label{Lem-CD-phi-tauM-Nm}
The following diagram commutes
\begin{equation}\label{CD-Lem-CD-phi-tauM-Nm}
\begin{tikzcd}[row sep=2.5em, column sep=4em]
C'' \arrow[r, "\bs\alpha_{C'',-1}"] \arrow[d, "\psi_G"'] & J^{-1}_{C''} \arrow[r, "\bar{\tau}_M"] \arrow[d, "\Nm_{\psi_G}"'] & J_{C''}/P_0 \arrow[d, "\overline{\Nm}_{\psi_G}"]\\
C \arrow[r, "\bs\alpha_{C,-1}"'] & J_C^{-1} \arrow[r, "\tau_{M_C}"'] & J_C.
\end{tikzcd}
\end{equation}
\end{Lem}
\begin{proof}
The left hand side of the diagram commutes by \cite[Proposition 3.1]{HP}. The right hand side commutes because the norm map is a group homomorphism.
\end{proof}
 
\begin{Lem}\label{Lem=sigmachi(g)}
Let $\chi\in G^{\vee}$ and regard it as a covering transformation of the $G^{\vee}$-Galois covering $\overline{\Nm}_{\psi_G}:J_{C''}/P_0
\rightarrow J_C$. Then, for any $g\in G$, $\rho_C(g)$ is a line bundle on $J_C$ such that $\overline{\Nm}{}_{\psi_G}^{\ast}\rho_C(g)\cong\mathcal{O}_{J_{C''}/P_0}$. Moreover, if $\mathcal{O}_{J_{C''}/P_0}$ is equipped with a $G^{\vee}$-equivariant structure so that it descends to $\rho_C(g)$, then $\chi$ acts as the multiplication by $\chi(g)^{-1}$ on $H^0(J_{C''}/P_0,\mathcal{O}_{J_{C''}/P_0})$.
\end{Lem}
The dual of the morphism $\Gamma\rightarrow\pi_0(P_{C''})\cong G^{\vee}$ identifies $(G^{\vee})^{\vee}$ with a subgroup of $\Gamma^{\vee}$. The lemma says that the map $G\hookrightarrow\Gamma\stackrel{\rho_C}{\rightarrow}\Gamma^{\vee}$ induces a map $G\rightarrow(G^{\vee})^{\vee}$, and the latter one sends $g\in G$ to the character of $G^{\vee}$ that takes the value $\chi(g)^{-1}$ at every $\chi\in G^{\vee}$.
\begin{proof}
Since $P_0$ is an abelian subvariety of $J_{C''}$, the map $\Pic^0(J_{C''}/P_0)\rightarrow\Pic^0(J_{C''})$ induced by the quotient morphism is injective. Therefore $\Ker\Nm_{\psi_G}^{\vee}=\Ker\overline{\Nm}{}^{\vee}_{\psi_G}$. Then (\ref{CD-PicAJ-XaC}) implies that $\overline{\Nm}{}_{\psi_G}^{\ast}\rho_C(g)\cong\mathcal{O}_{J_{C''}/P_0}$. We can therefore regard $\rho_C(g)$ as a character $\theta_g$ of $G^{\vee}$. The goal is to show that 
\begin{equation}\label{eq-Lem=sigmachi(g)}
\theta_g(\chi)=\chi(g)^{-1}.
\end{equation}

Denote by $\bar{\psi}_G^{\vee}:J_C\rightarrow J_{C''}/P_0$ the composition of $\psi_G^{\vee}$ and the quotient morphism $J_{C''}\rightarrow J_{C''}/P_0$. Denote by $\Gamma_1\subset\Gamma$ the subgroup that is mapped into $P_0$ by $\psi_G^{\vee}$. Obviously, $[n]=\overline{\Nm}{}_{\psi_G}\circ\bar{\psi}_G^{\vee}$. That is, the $\Gamma$-Galois covering $[n]$ is the composition of the $\Gamma_1$-Galois covering $\bar{\psi}_G^{\vee}$ and the $G^{\vee}$-Galois covering $\overline{\Nm}{}_{\psi_G}$. Since $[n]^{\ast}\rho_C(g)\cong\mathcal{O}_{J_C}$, we can regard $\rho_C(g)$ as a character $\tilde{\theta}_g$ of $\Gamma$. We apply Lemma \ref{Lem-G1G2} to the commutative diagram of Galois coverings
$$
\begin{tikzcd}
J_C \arrow[d, "\lbrack n\rbrack"'] \arrow[r, "\bar{\psi}_G^{\vee}"] & J_{C''}/P_0 \arrow[d, "\overline{\Nm}{}_{\psi_G}"] \\
J_C \arrow[r, "\Id"'] & J_C,
\end{tikzcd}
$$
and obtain
$$
\begin{tikzcd}
(G^{\vee})^{\vee} \arrow[d, "\rotatebox{90}{\(\sim\)}"] \arrow[r] & \Gamma^{\vee} \arrow[d, "\rotatebox{90}{\(\sim\)}"] \\
\Ker\overline{\Nm}{}_{\psi_G}^{\vee} \arrow[r, hook] & \Ker[n]^{\vee},
\end{tikzcd}
$$
so that $\theta_g\in(G^{\vee})^{\vee}$ is mapped to $\tilde{\theta}_g$ under the top arrow. This implies that for any $\gamma\in\Gamma$ that is mapped to $\chi\in G^{\vee}$ under the map $\Gamma\rightarrow \pi_0(P_{C''})\cong G^{\vee}$, we have $\tilde{\theta}_g(\gamma)=\theta_g(\chi)$. Now $\tilde{\theta}_g(\gamma)=\tilde{\theta}_{\gamma}(g)^{-1}$ by the skew-symmetry of the Weil pairing
\begingroup
\allowdisplaybreaks
\begin{align*}
\Gamma\times\Gamma&\longrightarrow\mathbb{C}^{\ast}\\
(g,\gamma)&\longmapsto \tilde{\theta}_{\gamma}(g)
\end{align*}
\endgroup
(see \cite[Corollary 11.22]{EGM} in the preliminary version of the book of Edixhoven-van der Geer-Moonen; see also \cite[\S 23]{M}, though the pairing there is formulated differently.)
That $\gamma$ is mapped to $\chi$ means that $\tilde{\theta}_{\gamma}|_G=\chi$, by Lemma \ref{Lem-GaV-HV}. The desired equality (\ref{eq-Lem=sigmachi(g)}) follows.
\end{proof}

The following theorem gives a characterisation of the connected components of Prym varieties in terms of the image of $f^{(-1)}_{\chi}$.
\begin{Thm}\label{Thm-Conn-Prym-New}
The homomorphism $\sigma$ is the identity on $G^{\vee}$, i.e., for any $\chi\in G^{\vee}$, the image of $f^{(-1)}_{\chi}$ lies in $P_{C'',\chi}$.
\end{Thm}
\begin{proof}
Consider the commutative diagram (\ref{CD-Lem-CD-phi-tauM-Nm}), where $\psi_G$ and $\overline{\Nm}_{\psi_G}$ are both $G^{\vee}$-coverings. By Lemma \ref{Lem-tauM-equiv}, the top arrow $\bar{\tau}_M\circ\bs\alpha_{C'',-1}$ in (\ref{CD-Lem-CD-phi-tauM-Nm}) is $G^{\vee}$-equivariant, with $\chi\in G^{\vee}$ acting on $C''$ as the covering transformation and on $J_{C''}/P_0$ as the multiplication by $\sigma(\chi^{-1})$. We apply Lemma \ref{Lem-G1G2} to (\ref{CD-Lem-CD-phi-tauM-Nm}) and obtain the following commutative diagram
$$
\begin{tikzcd}
(G^{\vee})^{\vee} \arrow[d, "\rotatebox{90}{\(\sim\)}"] \arrow[r] & (G^{\vee})^{\vee} \arrow[d, "\rotatebox{90}{\(\sim\)}"] \\
\Ker\overline{\Nm}{}_{\psi_G}^{\vee} \arrow[r] & \Ker\psi_G^{\vee}.
\end{tikzcd}
$$
The top arrow sends a character of $G^{\vee}$ to its composition with $\sigma$ and the inversion map. The bottom arrow is the restriction of the isomorphism $\rho_C$, and sends $\rho_C(g)$ to $g$, for any $g\in \Ker\psi_G^{\vee}$. Under the vertical arrows, $g$ is identified with the character of $G^{\vee}$ whose value at $\chi\in G^{\vee}$ is equal to $\chi(g)$, and $\rho_C(g)$ is identified with the character whose value at $\chi$ is equal to $\chi(g)^{-1}$ by Lemma \ref{Lem=sigmachi(g)}. We conclude that $\chi(g)=\sigma(\chi)(g)$ for any $g\in G$.
\end{proof}

\section{Computations of Determinant of Cohomology}\label{S-CDC}
\subsection{Review of determinant of cohomology}\label{SS-DoC}\hfill

We recall some facts about \textit{determinant of cohomology}. Let $f:X\rightarrow S$ be a flat projective morphism of schemes whose geometric fibres are curves. Let $F$ be a coherent sheaf on $X$ that is flat over $S$. According to \cite[Expos\'e III, Proposition 4.8]{ILL}, the complex $R^{\bullet}f_{\ast}F$ is perfect, i.e. there is an open covering $S=\bigcup_iU_i$ such that $R^{\bullet}f_{\ast}F|_{U_i}$ is quasi-isomorphic to a complex of locally free sheaves for every $i$. Suppose that $R^{\bullet}f_{\ast}F|_{U_i}$ is quasi-isomorphic to $E_i^{\bullet}=[E^0_i\rightarrow E^1_i]$ for some complex of locally free sheaves $E_i^{\bullet}$ (see \cite[Observation 43]{E} for the construction of $E_i^{\bullet}$). Then $\det E_i^{\bullet}:=\det E_i^1\otimes(\det E_i^0)^{-1}$ is a line bundle on $U_i$. The quasi-isomorphisms $R^{\bullet}f_{\ast}F|_{U_i}\cong E^{\bullet}_i$ and $R^{\bullet}f_{\ast}F|_{U_j}\cong E^{\bullet}_j$ induce an isomorphism $\det E^{\bullet}_i\cong\det E^{\bullet}_j$ on $U_i\cap U_j$. These data define a line bundle $\mathbf{D}_f(F)$ on $S$, which is called the determinant of cohomology. 

For the convenience of the reader, we collect below some properties of the determinant of cohomology. 

(i). If $f$ factorises as $X\stackrel{\phi}{\rightarrow}Y\stackrel{g}{\rightarrow} S$ where $\phi$ is a finite morphism, then $R^{\bullet}f_{\ast}F\cong R^{\bullet}g_{\ast}\phi_{\ast}F$ for any coherent sheaf $F$ on $X$. Therefore $\mathbf{D}_f(F)\cong\mathbf{D}_g(\phi_{\ast}F)$.

(ii). (Base change property) The determinant of cohomology is compatible with any base change. Let $g:T\rightarrow S$ be a morphism of schemes so that we have a Cartesian diagram
\begin{equation}
\begin{tikzcd}[row sep=2.5em, column sep=2em]
X\times_ST \arrow[r, "\pr_1"] \arrow[d, swap, "\pr_2"] \arrow[dr, phantom, "\square"] & X \arrow[d, "f"] \\
T \arrow[r, swap, "g"] & S.
\end{tikzcd}
\end{equation}
Then there is an isomorphism $g^{\ast}\mathbf{D}_f(F)\cong\mathbf{D}_{\pr_2}(\pr_1^{\ast}F)$ (see \cite[Proposition 44 (1)]{E}).

(iii). (Projection property) For any line bundle $L$ on $S$, there is an isomorphism
$$
\mathbf{D}_f(F\otimes f^{\ast}L)\cong\mathbf{D}_f(F)\otimes L^{\chi(F/S)}
$$
(see \cite[Proposition 44 (3)]{E}).

(iv). (Additive property) For every short exact sequence 
$$
0\longrightarrow F_1\longrightarrow F_2\longrightarrow F_3\longrightarrow 0
$$
of coherent sheaves on $X$ that are flat over $S$, there is an isomorphism
$$
\mathbf{D}_f(F_2)\cong\mathbf{D}_f(F_1)\otimes\mathbf{D}_f(F_3)
$$
(see \cite[Proposition 44 (4)]{E}).

Let $X$ be a curve with planar singularities, and $J$ (resp. $\bar{J}$) the Jacobian (resp. compactified Jacobian) of $X$. Let $\overline{\mathcal{U}}$ be the universal sheaf on $X\times\bar{J}$. Its fibre over any geometric point of $\bar{J}$ is a rank one torsion-free sheaf on $X$. Let $\mathcal{U}$ be the restriction of $\overline{\mathcal{U}}$ to $X\times J$. It is a line bundle. The Poincar\'e bundle on $\bar{J}\times J$ is defined to be
\begin{equation}\label{eq-Poin-sh}
\mathcal{P}:=\mathbf{D}_{\pr_{23}}(\pr_{12}^{\ast}\overline{\mathcal{U}}\otimes \pr_{13}^{\ast}\mathcal{U})^{-1}\otimes\mathbf{D}_{\pr_{23}}(\pr_{12}^{\ast}\overline{\mathcal{U}})\otimes\mathbf{D}_{\pr_{23}}(\pr_{13}^{\ast}\mathcal{U}),
\end{equation}
where $\pr_{ij}$ is the projection from $X\times\bar{J}\times J$ to the $i$'th and $j$'th factors. We may similarly define a line bundle $\mathcal{P}'$ on $J\times\bar{J}$, reversing the roles of $J$ and $\bar{J}$ in the definition of $\mathcal{P}$. Then $\mathcal{P}$ coincides with $\mathcal{P}'$ on $J\times J\subset\bar{J}\times\bar{J}$. Let $U=J\times\bar{J}\cup\bar{J}\times J$, and let $i:U\hookrightarrow\bar{J}\times\bar{J}$ be the inclusion. Then there is a line bundle $\mathcal{P}''$ on $U$ which restricts to $\mathcal{P}$ (resp. $\mathcal{P}'$) on $\bar{J}\times J$ (resp. $J\times\bar{J}$). The Poincar\'e sheaf $\overline{\mathcal{P}}$ on $\bar{J}\times\bar{J}$ given by Theorem \ref{Thm-Ari2AB} is the extension $i_{\ast}\mathcal{P}''$.

\subsection{A Commutative Diagram of Jacobians}\hfill

Now we go back to the setting of \S \ref{subsec-prym}, where $X\rightarrow C$ is an integral spectral curve, $C'\rightarrow C$ an $H^{\vee}$-Galois covering defined by some subgroup $H\subset G\subset\Pic^0(C)[n]$, and $Y$ some integral spectral curve over $C'$ that fits into the commutative diagram (\ref{CD-Za'-Ya''-Xa}). In particular, we have a finite birational morphism $\phi:Y\rightarrow X$. If $F$ is a rank one torsion-free sheaf of degree 0 on $Y$, then we introduce a line bundle $K$ on $X$ as in (\ref{eq-K}), so that $\phi_{\ast}F\otimes K$ has degree 0. This defines a morphism 
\begin{equation}\label{eq-phi-W}
\phi^{\wedge}:\bar{J}'\longrightarrow\bar{J}.
\end{equation}
It is easy to see that $\phi^{\wedge}$ factors through the fixed point locus $\bar{J}^H\subset\bar{J}$. Let $\overline{\mathcal{U}}{}'$ be the universal sheaf on $Y\times\bar{J}'$, then $(\phi\times\Id)_{\ast}\overline{\mathcal{U}}{}'\otimes\pr_1^{\ast}K$ is a family of rank one torsion-free sheaves of degree 0 on $X$. Let $\overline{\mathcal{U}}$ be the universal sheaf on $X\times\bar{J}$. We have
\begin{equation}\label{eq-U-phiW}
(\phi\times\Id)_{\ast}\overline{\mathcal{U}}{}'\otimes\pr_1^{\ast}K\cong(\Id\times\phi^{\wedge})^{\ast}\overline{\mathcal{U}}\otimes \pr_2^{\ast}M
\end{equation}
for some invertible sheaf $M$ on $\bar{J}'$. Recall that $\phi^{\vee}:J\longrightarrow J'$ is the morphism defined by pulling back line bundles.
\begin{Prop}\label{Prop-phiWV-phiV}
The following diagram commutes
\begin{equation}
\begin{tikzcd}[row sep=2.5em, column sep=2em]
\Pic^0_{\bar{J}} \arrow[r, "\sim", "\bs\alpha^{\vee}"'] \arrow[d, swap, "(\phi^{\wedge})^{\vee}"] & J \arrow[d, "\phi^{\vee}"]\\
\Pic^0_{\bar{J}'} \arrow[r, "\sim", "\bs\alpha^{\prime\vee}"'] & J',
\end{tikzcd}
\end{equation}
where the horizontal arrows are the pullback morphisms induced by the Abel-Jacobi maps associated to $X$ and $Y$ respectively. In particular, the morphism $(\phi^{\wedge})^{\vee}$ is independent of the choice of $K$.
\end{Prop}
\begin{Rem}
The above commutative diagram seems to be the dual of a commutative diagram like
\begin{equation}
\begin{tikzcd}[row sep=2.5em, column sep=3em]
Y \arrow[r, "\bs\alpha'"] \arrow[d, "\phi"'] & \bar{J}' \arrow[d, "\phi^{\wedge}"] \\
X \arrow[r, "\bs\alpha"'] & \bar{J}.
\end{tikzcd}
\end{equation}
However, this diagram can not be commutative. Indeed, for a smooth point of $X$, its image under $\bs\alpha$ represents an invertible sheaf on $X$. But every point in the image of $\phi^{\wedge}$ has non trivial stabiliser under the action of $\Gamma$, and so does not lie in $J$. It is surprising that after applying the Picard functor we get a commutative diagram.
\end{Rem}
\begin{proof}
We will show that for any $L\in J$, we have
\begin{equation}\label{eq-Prop-phiWV-phiV-0}
\phi^{\ast}L\cong(\phi^{\wedge}\circ\bs\alpha')^{\ast}\mathcal{P}|_{\bar{J}\times\{L\}}.
\end{equation}
We will use the idea of the proof of \cite[Proposition 2.2]{EGK}.

Now
\begin{equation}\label{eq-Prop-phiWV-phiV-1}
\mathcal{P}_{\bar{J}\times\{L\}}\cong\mathbf{D}_{\pr_2}(\overline{\mathcal{U}}\otimes\pr_1^{\ast}L)^{-1}\otimes\mathbf{D}_{\pr_2}(\overline{\mathcal{U}})\otimes\mathbf{D}_{\pr_{2}}(\pr_{1}^{\ast}L).
\end{equation}
Let us compute
\begingroup
\allowdisplaybreaks
\begin{align}
\nonumber
&(\phi^{\wedge})^{\ast}\mathbf{D}_{\pr_{2}}(\overline{\mathcal{U}}\otimes \pr_1^{\ast}L)\\\nonumber
\stackrel{\circled{1}}{\cong}&\mathbf{D}_{\pr_{2}}\big((\Id\times\phi^{\wedge})^{\ast}\overline{\mathcal{U}}\otimes\pr_1^{\ast}L\big)\\
\nonumber
\stackrel{\circled{2}}{\cong}&\mathbf{D}_{\pr_{2}}\big((\phi\times\Id)_{\ast}\overline{\mathcal{U}}{}'\otimes \pr_1^{\ast}K\otimes \pr_1^{\ast}L\big)\otimes M_1\\
\nonumber
\stackrel{\circled{3}}{\cong}&\mathbf{D}_{\pr_{2}}\big((\phi\times\Id)_{\ast}(\overline{\mathcal{U}}{}'\otimes \pr_1^{\ast}\phi^{\ast}K\otimes \pr_{1}^{\ast}\phi^{\ast}L)\big)\otimes M_1\\
\label{eq-Prop-phiWV-phiV-2}
\stackrel{\circled{4}}{\cong}&\mathbf{D}_{\pr_{2}}\big(\overline{\mathcal{U}}{}'\otimes \pr_1^{\ast}\phi^{\ast}K\otimes \pr_{1}^{\ast}\phi^{\ast}L\big)\otimes M_1.
\end{align}
\endgroup
In the above equations, the domains of the projections $\pr_i$ can vary. In equality \circled{1}, we have used the base-change property of the determinant of cohomology. In equality \circled{2}, we have used (\ref{eq-U-phiW}) and the projection property of the determinant of cohomology, and $M_1$ is some line bundle on $\bar{J}'$. In equality \circled{3}, we have used the projection formula. In equality \circled{4}, we have used the fact that determinant of cohomology is preserved under finite morphisms. Similarly, we have
\begingroup
\allowdisplaybreaks
\begin{align}
\label{eq-Prop-phiWV-phiV-3}
(\phi^{\wedge})^{\ast}\mathbf{D}_{\pr_{2}}(\overline{\mathcal{U}})&\cong\mathbf{D}_{\pr_{2}}\big(\overline{\mathcal{U}}{}'\otimes \pr_1^{\ast}\phi^{\ast}K\big)\otimes M_1,\\
\label{eq-Prop-phiWV-phiV-4}
(\phi^{\wedge})^{\ast}\mathbf{D}_{\pr_{2}}(\pr_{1}^{\ast}L)&\cong\mathcal{O}_Y.
\end{align}
\endgroup
Then we compute
\begingroup
\allowdisplaybreaks
\begin{align*}
&\bs\alpha^{\prime\ast}\mathbf{D}_{\pr_{2}}\big(\overline{\mathcal{U}}{}'\otimes \pr_1^{\ast}\phi^{\ast}K\otimes \pr_{1}^{\ast}\phi^{\ast}L\big)\\
\stackrel{\circled{1}}{\cong}&\mathbf{D}_{\pr_{2}}\big((\Id\times\bs\alpha')^{\ast}\overline{\mathcal{U}}{}'\otimes \pr_1^{\ast}\phi^{\ast}K\otimes \pr_{1}^{\ast}\phi^{\ast}L\big)\\
\stackrel{\circled{2}}{\cong}&\mathbf{D}_{\pr_{2}}\big(\mathcal{I}_{\Delta}\otimes \pr_1^{\ast}L_0\otimes \pr_1^{\ast}\phi^{\ast}K\otimes \pr_{1}^{\ast}\phi^{\ast}L\big)\otimes M_2\\
\end{align*}
\endgroup
where $M_2$ is some line bundle on $Y$. In equality \circled{1}, we have used the base-change property of the determinant of cohomology. In equality \circled{2}, we have used (\ref{eq-AJ-uni}) and the projection property.

Recall that $\mathcal{I}_{\Delta}$ is the ideal sheaf of the diagonal $\Delta:Y\rightarrow Y\times Y$. We take the tensor product of the exact sequence
$$
0\longrightarrow\mathcal{I}_{\Delta}\longrightarrow\mathcal{O}_{Y\times Y}\longrightarrow\mathcal{O}_{\Delta}\longrightarrow 0,
$$
with the line bundle $\pr_1^{\ast}L_0\otimes \pr_1^{\ast}\phi^{\ast}K\otimes \pr_{1}^{\ast}\phi^{\ast}L$, and use the additive property of the determinant of cohomology to obtain
\begingroup
\allowdisplaybreaks
\begin{align*}
&\mathbf{D}_{\pr_{2}}\big(\mathcal{I}_{\Delta}\otimes \pr_1^{\ast}L_0\otimes \pr_1^{\ast}\phi^{\ast}K\otimes \pr_{1}^{\ast}\phi^{\ast}L\big)\\
\cong&\mathbf{D}_{\pr_{2}}\big(\pr_1^{\ast}L_0\otimes \pr_1^{\ast}\phi^{\ast}K\otimes \pr_{1}^{\ast}\phi^{\ast}L\big)\otimes (L_0\otimes \phi^{\ast}K\otimes \phi^{\ast}L)^{-1}\\
\cong&(L_0\otimes \phi^{\ast}K\otimes \phi^{\ast}L)^{-1}.
\end{align*}
\endgroup

Therefore, we have
\begingroup
\allowdisplaybreaks
\begin{align}
\label{eq-Prop-phiWV-phiV-5}
\bs\alpha^{\prime\ast}\mathbf{D}_{\pr_{2}}\big(\overline{\mathcal{U}}{}'\otimes \pr_1^{\ast}\phi^{\ast}K\otimes \pr_{1}^{\ast}\phi^{\ast}L\big)\cong (L_0\otimes \phi^{\ast}K\otimes \phi^{\ast}L)^{-1}\otimes M_2.
\end{align}
\endgroup
Similarly, we have
\begingroup
\allowdisplaybreaks
\begin{align}
\label{eq-Prop-phiWV-phiV-6}
\bs\alpha^{\prime\ast}\mathbf{D}_{\pr_{2}}\big(\overline{\mathcal{U}}{}'\otimes \pr_1^{\ast}\phi^{\ast}K\big)&\cong (L_0\otimes \phi^{\ast}K)^{-1}\otimes M_2.
\end{align}
\endgroup
Combining (\ref{eq-Prop-phiWV-phiV-1}), (\ref{eq-Prop-phiWV-phiV-2}), (\ref{eq-Prop-phiWV-phiV-3}), (\ref{eq-Prop-phiWV-phiV-4}), (\ref{eq-Prop-phiWV-phiV-5}) and (\ref{eq-Prop-phiWV-phiV-6}), we obtain (\ref{eq-Prop-phiWV-phiV-0}).
\end{proof}

\section{Poincar\'e Sheaves}\label{Sec-PS}
\subsection{Equivariant Structures on Poincar\'e Sheaves}\label{subsec-ESPS}\hfill

As before, the fixed spectral curve $X$ determines $G\subset\Gamma$, and $H$ is an arbitrary subgroup of $G$. Up to the action of $H^{\vee}$, this gives us a spectral curve $Y$ over $C'$, and in particular the compactified Jacobian variety $\bar{J}'$ and the compactified Prym variety $\bar{P}'$ as in \S \ref{subsec-prym}. We will define the Poincar\'e sheaf $\overline{\mathcal{Q}}{}'$ on $\bar{P}'\times[\bar{P}'/\Gamma]$.

We begin with a lemma relating various equivariant structures. Recall that $\Gamma_0\subset\Gamma$ is the subgroup that is mapped into the connected component $\bar{P}'_0$. The group $\Gamma_0$ acts on $\bar{J}'$ naturally. It also acts on $\bar{P}'_0\times J_C$ in three different ways. For any $\gamma\in\Gamma_0$, we let it act on $\bar{P}'_0\times J_C$ as $\gamma\times\gamma^{-1}$, then $r:\bar{P}'_0\times J_C\rightarrow\bar{J}'$ denotes the quotient morphism for this action. See Remark \ref{Rem-JCP=barJ-0}. Besides, the group $\Gamma_0$ may also act on one of the two factors of $\bar{P}'_0\times J_C$ in the natural way, and acts trivially on the other.
\begin{Lem}\label{Lem-3-equiv}
We have an equivalence between the following three sets of data:
\begin{itemize}
\item[(i)] A coherent sheaf $L$ on $\bar{J}'$ with a $\Gamma_0$-equivariant structure.
\item[(ii)] A coherent sheaf $M$ on $\bar{P}'_0\times J_C$ with 
\begin{itemize}
\item[(a)] a $\Gamma_0$-equivariant structure $$\theta_{\gamma}:M\lisom(\gamma\times\gamma^{-1})^{\ast}M,\text{ for any $\gamma\in\Gamma_0$, and}$$
\item[(b)] a $\Gamma_0$-equivariant structure $$\zeta_{J,\gamma}:M\lisom(\Id\times\gamma)^{\ast}M,\text{ for any $\gamma\in\Gamma_0$}$$ such that the two equivariant structures are compatible, i.e. for any $\gamma_1$ and $\gamma_2\in\Gamma_0$, we have a commutative diagram
\begin{equation}\label{eq-Lem-3-equiv-comp}
\begin{tikzcd}[row sep=2.5em, column sep=3em]
M \arrow[r, "\zeta_{J,\gamma_2}"] \arrow[d, "\theta_{\gamma_1}"'] & (\Id\times\gamma_2)^{\ast}M \arrow[d, "(\Id\times\gamma_2)^{\ast}\theta_{\gamma_1}"]\\
(\gamma_1\times\gamma_1^{-1})^{\ast}M \arrow[r, "(\gamma_1\times\gamma_1^{-1})^{\ast}\zeta_{J,\gamma_2}"'] & (\gamma_1\times\gamma_1^{-1})^{\ast}(\Id\times\gamma_2)^{\ast}M,
\end{tikzcd}
\end{equation}
where we have used the canonical isomorphism $$(\Id\times\gamma_2)^{\ast}(\gamma_1\times\gamma_1^{-1})^{\ast}M\cong(\gamma_1\times\gamma_1^{-1})^{\ast}(\Id\times\gamma_2)^{\ast}M.$$
\end{itemize}
\item[(iii)] A coherent sheaf $M$ on $\bar{P}'_0\times J_C$ with 
\begin{itemize}
\item[(a)] a $\Gamma_0$-equivariant structure $$\theta_{\gamma}:M\lisom(\gamma\times\gamma^{-1})^{\ast}M,\text{ for any $\gamma\in\Gamma_0$, and}$$
\item[(b)] a $\Gamma_0$-equivariant structure $$\zeta_{P,\gamma}:M\lisom(\gamma\times\Id)^{\ast}M,\text{ for any $\gamma\in\Gamma_0$}$$ such that the two equivariant structures are compatible, i.e. a similar diagram like (\ref{eq-Lem-3-equiv-comp}) commutes.
\end{itemize}
\end{itemize}
Moreover, under the equivalence between (ii) and (iii), we have 
\begingroup
\allowdisplaybreaks
\begin{align}
\label{Lem-3-equiv-eq01}
(\Id\times\gamma)^{\ast}\theta_{\gamma}\circ\zeta_{J,\gamma}&=\zeta_{P,\gamma}\\
\label{Lem-3-equiv-eq02}
(\gamma\times\Id)^{\ast}\theta_{\gamma^{-1}}\circ\zeta_{P,\gamma}&=\zeta_{J,\gamma}.
\end{align}
\endgroup
\end{Lem}
\begin{proof}
Since $r:\bar{P}'_0\times J_C\rightarrow\bar{J}'$ is a $\Gamma_0$-Galois covering, a coherent sheaf $M$ on $\bar{P}'_0\times J_C$ with an equivariant structure $(\theta_{\gamma})_{\gamma\in\Gamma_0}$ is equivalent to a coherent sheaf $L$ on $\bar{J}'$, and $M\cong r^{\ast}L$. If $\{\Phi_{\gamma}:L\isom\gamma^{\ast}L\}_{\gamma\in\Gamma_0}$ is an equivariant structure on $L$. Then we define $\zeta_{J,\gamma}$ to be
\begin{equation}\label{Lem-3-equiv-eq1}
r^{\ast}L\stackrel{r^{\ast}\Phi_{\gamma}}{\longrightarrow}r^{\ast}\gamma^{\ast}L\lisom(\Id\times\gamma)^{\ast}r^{\ast}L
\end{equation}
where the second isomorphism is the canonical one. This defines an equivariant structure on $r^{\ast}L$. And it is routine to verify the compatibility with the equivariant structure $\{\theta_{\gamma}\}_{\gamma\in\Gamma_0}$. Conversely, given $M$ as in (ii). We let $L$ be the invariant subsheaf $(r_{\ast}M)^{\Gamma_0}$ with the $\Gamma_0$-action defined by $\{\theta_{\gamma}\}_{\gamma\in\Gamma_0}$. Then $\{\zeta_{J,\gamma}\}_{\gamma\in\Gamma_0}$ defines an equivariant structure $\{\Phi_{\gamma}\}_{\gamma\in\Gamma_0}$ on $r_{\ast}M$. The compatibility of the two equivariant structures implies that $\Phi_{\gamma}$ preserves the invariant subsheaf $L\subset r_{\ast}M$. The two constructions are inverse to each other. In (iii) we put $\zeta_{P,\gamma}$ to be
\begin{equation}\label{Lem-3-equiv-eq2}
r^{\ast}L\stackrel{r^{\ast}\Phi_{\gamma}}{\longrightarrow}r^{\ast}\gamma^{\ast}L\lisom(\gamma\times\Id)^{\ast}r^{\ast}L
\end{equation}
and the proof is completely analogous. Finally, it is easy to check (\ref{Lem-3-equiv-eq01}) and (\ref{Lem-3-equiv-eq02}), using the fact that $M\cong r^{\ast}L$.
\end{proof}

\begin{Prop}\label{Prop-equiv-barP-a'}
Let $\xi\in H^{\vee}$. Regard $\overline{\mathcal{P}}{}'$ as a coherent sheaf on $\bar{P}'_{\xi}\times\bar{J}'$ by restriction. Consider the $\Gamma_0$-action on $\bar{P}'_{\xi}\times\bar{J}'$ by $\Id\times\gamma$, for any $\gamma\in\Gamma_0$. Then $\overline{\mathcal{P}}{}'$ admits a $\Gamma_0$-equivariant structure.
\end{Prop}
Note that $H\subset\Gamma_0$ acts trivially on $\bar{J}'$. By \cite[\S 4.1]{GS}, there exists an isomorphism $\Phi_{\gamma}:\overline{\mathcal{P}}{}'\isom(\Id\times\gamma)^{\ast}\overline{\mathcal{P}}{}'$ on $\bar{P}'_{\xi}\times \bar{J}'$ for every $\gamma\in\Gamma_0$. For purpose of computation in \S \ref{subsec-TS}, we would like to be specific about the choice of the isomorphisms $\Phi_{\gamma}$.

\begin{Lem}\label{Lem-HomP=C}
Regard $\overline{\mathcal{P}}{}'$ as a coherent sheaf on $\bar{P}'_{\xi}\times \bar{J}'$ by restriction. Then $\Hom(\overline{\mathcal{P}}{}',\overline{\mathcal{P}}{}')=\mathbb{C}$.
\end{Lem}
\begin{proof}
By \cite[Lemma 4.3 (a)]{GS}, the coherent sheaf $\overline{\mathcal{P}}{}'$ is Cohen-Macaulay, in particular saturated in codimension 2, and is isomorphic to $i_{\ast}\mathcal{P}'$, where $\mathcal{P}'$ is an invertible sheaf on an open subset $i:U\rightarrow\bar{P}'_{\xi}\times \bar{J}'$ whose complement has codimension larger than 2. According to \cite[Lemma 30.11.3]{S-P} the sheaf of local homomorphisms $\SHom(\overline{\mathcal{P}}{}',\overline{\mathcal{P}}{}')$ is also saturated in codimension 2. The restriction of this sheaf to $U$ is then $\mathcal{O}_U$, and so itself must be isomorphic to $\mathcal{O}_{\bar{P}'_{\xi}\times \bar{J}'}$. Since $\bar{P}'_{\xi}\times \bar{J}'$ is connected and projective, the global sections are the constant functions.
\end{proof}

\begin{proof}[Proof of Proposition \ref{Prop-equiv-barP-a'}]
We must choose the isomorphisms $\{\Phi_{\gamma}\}_{\gamma\in\Gamma_0}$ properly so that they are compatible.
The above lemma shows that each $\Phi_{\gamma}$ can only be modified by a scalar. Consider the $\Gamma_0$-equivariant morphism 
$$
\nu:\bar{P}'_{\xi}\times J_C\stackrel{\Id\times \pi^{\prime\vee}}{\longrightarrow} \bar{P}'_{\xi}\times\bar{J}',
$$
where $\Gamma_0$ acts by multiplication on $J_C$ and on $\bar{J}'$, and acts trivially on $\bar{P}'_{\xi}$. For any $p\in \bar{P}'_{\xi}$, we have $\nu^{\ast}\overline{\mathcal{P}}{}'|_{\{p\}\times J_C}\cong\mathcal{O}_{J_C}$ in view of (\ref{CD-piVV-Nm}). Note that $\nu^{\ast}\overline{\mathcal{P}}{}'$ is a line bundle on $\bar{P}'_{\xi}\times J_C$, since the image of $J_C$ under $\pi^{\prime\vee}$ is contained in $J'$. Therefore, $\nu^{\ast}\overline{\mathcal{P}}{}'\cong\pr_1^{\ast}M$ for some line bundle $M$ on $\bar{P}'_{\xi}$. There is a canonical isomorphism $\Psi_{\gamma}:\pr_1^{\ast}M\isom(\Id\times\gamma)^{\ast}\pr_1^{\ast}M$ for every $\gamma\in\Gamma_0$. And these isomorphisms define a $\Gamma_0$-equivariant structure on $\pr_1^{\ast}M$. Note that an automorphism of $\nu^{\ast}\overline{\mathcal{P}}{}'$ is simply a non zero scalar, and so we can always modify $\Phi_{\gamma}$ by a scalar so that the induced isomorphism
$$
\Phi'_{\gamma}:\nu^{\ast}\overline{\mathcal{P}}{}'\stackrel{\nu^{\ast}\Phi_{\gamma}}{\longrightarrow}\nu^{\ast}(\Id\times\gamma)^{\ast}\overline{\mathcal{P}}{}'\lisom(\Id\times\gamma)^{\ast}\nu^{\ast}\overline{\mathcal{P}}{}'
$$ 
coincides with $\Psi_{\gamma}$, for any $\gamma$. This uniquely determines $\Phi_{\gamma}$ since any difference by a scalar must induce the multiplication by the same scalar on $\nu^{\ast}\overline{\mathcal{P}}{}'$. Then the isomorphism
$$
\overline{\mathcal{P}}{}'\stackrel{\Phi_{\gamma_2}}{\longrightarrow}(\Id\times\gamma_2)^{\ast}\overline{\mathcal{P}}{}'\stackrel{(\Id\times\gamma_2)^{\ast}\Phi_{\gamma_1}}{\longrightarrow}(\Id\times\gamma_2)^{\ast}(\Id\times\gamma_1)^{\ast}\overline{\mathcal{P}}{}'
$$
must agree with $\Phi_{\gamma_1\gamma_2}$, since these two isomorphisms induce the same (canonical) isomorphism on $\bar{P}'_{\xi}\times J_C$. Therefore, the isomorphisms $\{\Phi_{\gamma}\}_{\gamma\in\Gamma_0}$ indeed define a $\Gamma_0$-equivariant structure on $\overline{\mathcal{P}}{}'$.
\end{proof}
\begin{Rem}\label{Rem-equiv-barP-a'}
Fix some $p\in \bar{P}'_{\xi}$. We have seen that $\nu^{\ast}\overline{\mathcal{P}}{}'|_{\{p\}\times J_C}\cong \mathcal{O}_{J_C}$. The equivariant structure $\{\Psi_{\gamma}\}_{\gamma\in\Gamma_0}$ induces an equivariant structure $\{\psi_{\gamma}\}_{\gamma\in\Gamma_0}$ on $\mathcal{O}_{J_C}$. For any $\gamma\in\Gamma_0$, the isomorphism $\psi_{\gamma}:\mathcal{O}_{J_C}\isom\gamma^{\ast}\mathcal{O}_{J_C}$ sends a function $f$ to $f\circ \gamma^{-1}$, thus $\Gamma_0$ acts trivially on the global sections.
\end{Rem}

Consequently, for any $\xi\in H^{\vee}$, the restriction of $\overline{\mathcal{P}}{}'$ to $\bar{P}'_{\xi}\times\bar{P}'_0$ is $\Gamma_0$-equivariant, and so descends to a coherent sheaf $\overline{\mathcal{Q}}{}'$ on $\bar{P}'_{\xi}\times[\bar{P}'/\Gamma]$. For $H=0$, and so $\bar{P}'=\bar{P}$, the sheaf $\overline{\mathcal{Q}}{}=\overline{\mathcal{Q}}{}'$ will be the kernel of the Fourier-Mukai transform between $\bar{P}$ and $[\bar{P}/\Gamma]$. 

\begin{Rem}\label{Rem-p'VQ'=P'}
In \cite[\S 4.1]{GS}, the authors have defined a $J_C$-equivariant structure on $\overline{\mathcal{P}}{}'$ regarded as a sheaf on $\bar{P}'_{\xi}\times\bar{J}'$, and $\overline{\mathcal{Q}}{}'$ is equivalently defined as the descending of $\overline{\mathcal{P}}{}'$. Then we have $(\Id\times p')^{\ast}\overline{\mathcal{Q}}{}'\cong\overline{\mathcal{P}}{}'$ as sheaves on $\bar{P}'_{\xi}\times\bar{J}'$, where $p':\bar{J}'\rightarrow[\bar{P}'/\Gamma]$ is the projection given by Lemma \ref{Lem-JCP=barJ}. Lemma \ref{Inj-pV} below shows that $\overline{\mathcal{Q}}{}'$ is the unique coherent sheaf on $\bar{P}'_{\xi}\times[\bar{P}'/\Gamma]$ such that $(\Id\times p')^{\ast}\overline{\mathcal{Q}}{}'\cong\overline{\mathcal{P}}{}'$.
\end{Rem}

\subsection{Representability of the Picard Functor of $[\bar{P}'/\Gamma]$}\label{subsec-RPF}\hfill
 
Let $p':\bar{J}'\rightarrow[\bar{P}'/\Gamma]$ be the morphism as in Remark \ref{Rem-p'VQ'=P'}. We also denote by $r_P:\bar{P}'_0\rightarrow[\bar{P}'_0/\Gamma_0]$ the quotient morphism. Note that there is a natural isomorphism $[\bar{P}'/\Gamma]\cong[\bar{P}'_0/\Gamma_0]$.
\begin{Lem}\label{Inj-pV}
The homomorphism of \'etale sheaves on $\Sch_{\mathbb{C}}$ $$p^{\prime\vee}:\Pic_{[\bar{P}'/\Gamma]}\rightarrow\Pic_{\bar{J}'}$$ is injective.
\end{Lem}
\begin{proof}
It suffices to prove the injectivity at the level of Picard functors (i.e. the presheaves on $\Sch_{\mathbb{C}}$ that associate to any $S\in\Sch_{\mathbb{C}}$ the groups $\Pic(S\times[\bar{P}'/\Gamma])$ and $\Pic(S\times\bar{J}')$ respectively). And it suffices to check the injectivity on the subcategory of affine schemes. We have a commutative diagram
\begin{equation}\label{CD-rPrJ-0}
\begin{tikzcd}[row sep=2.5em, column sep=2em]
\bar{P}'_0\times J_C \arrow[r, "\pr_1"] \arrow[d, "r"'] & \bar{P}'_0 \arrow[d, "r_P"] \\
\bar{J}' \arrow[r, "p'"'] & \lbrack\bar{P}'_0/\Gamma_0\rbrack,
\end{tikzcd}
\end{equation}
where in the bottom arrow we have identified $[\bar{P}'/\Gamma]$ with $[\bar{P}'_0/\Gamma_0]$. The above diagram induces a commutative diagram of abelian groups for every commutative $\mathbb{C}$-algebra $R$:
\begin{equation}\label{CD-rPrJ}
\begin{tikzcd}[row sep=2.5em, column sep=2em]
\Pic(\lbrack\bar{P}'_0/\Gamma_0\rbrack_R) \arrow[r, "p^{\prime\vee}(R)"] \arrow[d, "r_P^{\vee}(R)"'] & \Pic(\bar{J}'_R) \arrow[d, "r^{\vee}(R)"] \\
\Pic(\bar{P}'_{0,R}) \arrow[r, "\pr_1^{\vee}(R)"'] & \Pic(\bar{P}'_{0,R}\times_R J_{C,R}),
\end{tikzcd}
\end{equation}
where the subscript $R$ indicates the base change to $\Spec R$. Since $r_P$ is the quotient morphism for the $\Gamma_0$-action, the kernel of $r_P^{\vee}(R)$ is identified with the set of $\Gamma_0$-equivariant structures on $\mathcal{O}_{\bar{P}'_{0,R}}$. And similarly for $r$. The proof given below is essentially Lemma \ref{Lem-G1G2} over $\Spec R$.

Note that $H^0(\bar{P}'_{0,R},\mathcal{O}_{\bar{P}'_{0,R}})\cong R$. Suppose that the isomorphisms $\theta_{\gamma}:\mathcal{O}_{\bar{P}'_{0,R}}\isom\gamma^{\ast}\mathcal{O}_{\bar{P}'_{0,R}}$, $\gamma\in\Gamma_0$, define an equivariant structure. Each $\theta_{\gamma}$ sends the unit of $H^0(\bar{P}'_{0,R},\mathcal{O}_{\bar{P}'_{0,R}})$ to an element of $R^{\ast}$, the subset of invertible elements of $R$. Clearly, this defines a group homomorphism $\Gamma_0\rightarrow R^{\ast}$. Suppose that $\{\theta'_{\gamma}\}_{\gamma\in\Gamma_0}$ is a different equivariant structure. There is some $\gamma\in\Gamma_0$ such that $\theta_{\gamma}\ne\theta'_{\gamma}$. Their difference is an automorphism of $\mathcal{O}_{\bar{P}'_{0,R}}$, which is given by an element of $R^{\ast}$ not equal to $1$. Thus we get a different homomorphism $\Gamma_0\rightarrow R^{\ast}$. This shows that the map from the set of equivariant structures to $\Hom(\Gamma_0,R^{\ast})$ is injective. Given a homomorphism $\chi:\Gamma_0\rightarrow R^{\ast}$, we define for each $\gamma$ an isomorphism $\theta_{\gamma}:\mathcal{O}_{\bar{P}'_{0,R}}\isom\gamma^{\ast}\mathcal{O}_{\bar{P}'_{0,R}}$ sending a function $s$ on some open subset to $\chi(\gamma)\cdot s\circ \gamma^{-1}$. This defines a $\Gamma_0$-equivariant structure on $\mathcal{O}_{\bar{P}'_{0,R}}$. Thus we have a bijection between the set of $\Gamma_0$-equivariant structures on $\mathcal{O}_{\bar{P}'_{0,R}}$ and $\Hom(\Gamma_0,R^{\ast})$. Similarly, the set of $\Gamma_0$-equivariant structures on $\mathcal{O}_{\bar{P}'_{0,R}\times_RJ_{C,R}}$ is in bijection with $\Hom(\Gamma_0,R^{\ast})$. The projection $\bar{P}'_{0,R}\times_RJ_{C,R}\rightarrow\bar{P}'_{0,R}$ induces a $\Gamma_0$-equivariant isomorphism of global sections of structure sheaves. As $r$ and $r_P$ are quotient morphisms by $\Gamma_0$, $p^{\prime\vee}(R)$ restricts to an isomorphism $\Ker(r_P^{\vee})(R)\isom\Ker(r^{\vee})(R)$.

For any $R$-point of $J_C$, the composition $\bar{P}'_{0,R}\times_R\Spec R\hookrightarrow \bar{P}'_{0,R}\times_RJ_{C,R}\stackrel{\pr_1}{\rightarrow}\bar{P}'_{0,R}$ is the identity. Therefore, $\pr_1^{\vee}(R)$ is injective. Let $L\in\Pic([\bar{P}'_0/\Gamma_0]_R)$ be such that $p^{\prime\vee}(R)(L)=\mathcal{O}_{\bar{J}'_R}$. By the commutativity of (\ref{CD-rPrJ}), we have $L\in\Ker(r_P^{\vee})(R)$. However, the restriction of $p^{\prime\vee}(R)$ to  $\Ker(r_P^{\vee})(R)$ is an isomorphism, we conclude that $L$ must be the trivial line bundle.
\end{proof}

Recall the map $\pi'=\psi\circ\pi_H$ (see \ref{CD-Za'-Ya''-Xa}). We have a commutative diagram
\begin{equation}\label{CD-PrymAJ}
\begin{tikzcd}[row sep=2.5em, column sep=2em]
\Pic'_{\bar{P}'_0} \arrow[r, "\pr_1^{\vee}"] \arrow[dr, phantom, "\circled{1}"] & \Pic'_{\bar{P}'_0\times J_C} \arrow[r, "i_C^{\vee}"] & \Pic^0_{J_C} \arrow[r, "\sim" inner sep=.3mm, "\bs\alpha_C^{\vee}"'] \arrow[dr, phantom, "\circled{3}"] & J_C \\
\Pic'_{[\bar{P}'_0/\Gamma_0]} \arrow[r, "p^{\prime\vee}"'] \arrow[u, "r_P^{\vee}"] & \Pic^0_{\bar{J}'} \arrow[rr, "\sim", "\bs\alpha^{\prime\vee}"'] \arrow[u, "r^{\vee}"] \arrow[ur, "\circled{2}", "(\pi^{\prime\vee})^{\vee}"'] & & J' \arrow[u, "\Nm_{\pi'}"'],
\end{tikzcd}
\end{equation}
where $\Pic'_{\bar{P}'_0\times J_C}\subset\Pic_{\bar{P}'_0\times J_C}$ is defined to be the inverse image of $\Pic^0_{J_C}\subset\Pic_{J_C}$ under $i_C^{\vee}$, the subsheaf $\Pic'_{\bar{P}'_0}\subset\Pic_{\bar{P}'_0}$ is defined to be the inverse image of $\Pic'_{\bar{P}'_0\times J_C}\subset\Pic_{\bar{P}'_0\times J_C}$ under $\pr_1^{\vee}$, and $\Pic'_{[\bar{P}'_0/\Gamma_0]}\subset\Pic_{[\bar{P}'_0/\Gamma_0]}$ is defined to be the inverse image of $\Pic^0_{\bar{J}'}\subset\Pic_{\bar{J}'}$ under $p^{\prime\vee}$.
Diagram \circled{1} commutes because (\ref{CD-rPrJ-0}) commutes. Diagram \circled{2} is obtained by taking the Picard sheaves of the following commutative diagram
\begin{equation}\label{CD-JaPJCJC}
\begin{tikzcd}[row sep=2.5em, column sep=1.5em]
J_C \arrow[rr, "i_C"] \arrow[dr, swap, "\pi^{\prime\vee}"] && \bar{P}'_0\times J_C \arrow[dl, "r"]\\
& \bar{J}', &
\end{tikzcd}
\end{equation}
where $i_C$ is the inclusion identifying $J_C$ with $\{0\}\times J_C$. Note that the image of $\Pic^0_{\bar{J}'}$ under $r^{\vee}$ necessarily lies in the subsheaf $\Pic'_{\bar{P}'_0\times J_C}$, since $(\pi^{\prime\vee})^{\vee}$ maps $\Pic^0_{\bar{J}'}$ into $\Pic^0_{J_C}$. Diagram \circled{3} is the restriction of (\ref{CD-piVV-Nm}) to the smooth loci.

\begin{Thm}\label{Thm-Pic'-repre-P}
The \'etale sheaf $\Pic'_{[\bar{P}'/\Gamma]}$ is representable by $P'$.
\end{Thm}
\begin{proof}
We will show that the image of the injective morphism $p^{\prime\vee}$ coincides with $\rho'(P')$. Since $\bs\alpha^{\prime\vee-1}=\rho':J'\rightarrow\Pic^0_{\bar{J}'}$ is an isomorphism, this implies that $\Pic'_{[\bar{P}'/\Gamma]}$ is isomorphic to $P'$. From (\ref{CD-PrymAJ}) we deduce that $\Nm_{\pi'}\circ\bs\alpha^{\prime\vee}\circ p^{\prime\vee}$ is the zero morphism, and so $\bs\alpha^{\prime\vee}\circ p^{\prime\vee}$ factors through $P'\hookrightarrow J'$, or equivalently, the image of $p^{\prime\vee}$ is contained in $\rho'(P')$. Let $\tilde{\rho}':P'\rightarrow\Pic_{[\bar{P}'/\Gamma]}$ be the morphism defined by the Poincar\'e line bundle $\mathcal{Q}'$ on $P'\times[\bar{P}'/\Gamma]$ (on each connected component, the sheaf $\mathcal{Q}'$ is the restriction of the sheaf $\overline{\mathcal{Q}}{}'$ defined in \S \ref{subsec-ESPS}). A priori, we do not know that the image of $\tilde{\rho}'$ is contained in $\Pic'_{[\bar{P}'/\Gamma]}$. But note that $p^{\prime\vee}$ is defined on the entire Picard sheaf $\Pic_{[\bar{P}'/\Gamma]}$. The composition $p^{\prime\vee}\circ\tilde{\rho}'$ is represented by the line bundle $(\Id\times p')^{\ast}\mathcal{Q}'$, which is isomorphic to $\mathcal{P}'$ by Remark \ref{Rem-p'VQ'=P'}. Since $\mathcal{P}'$ is a family of degree zero line bundles, the image of the morphism $p^{\prime\vee}\circ\tilde{\rho}'$ is contained in $\Pic^0_{\bar{J}'}$, and so the image of $\tilde{\rho}'$ is contained in the subsheaf $\Pic'_{[\bar{P}'/\Gamma]}$. Now the morphism $P'\rightarrow\Pic^0_{\bar{J}'}$ induced by $\mathcal{P}'$ is exactly the restriction of $\rho'$ to $P'$. We conclude that $p^{\prime\vee}\circ\tilde{\rho}'=\rho'|_{P'}$, therefore $p^{\prime\vee}$ surjects onto $\rho'(P')$.
\end{proof}

Let $\Pic^0_{[\bar{P}'/\Gamma]}$ be the subsheaf of $\Pic'_{[\bar{P}'/\Gamma]}$ representable by $P'_0$. Define the following subsheaf of $\Pic_{[\bar{P}'/\Gamma]}$ 
$$\Pic^{\tau}_{[\bar{P}'/\Gamma]}:=\{s\in\Pic_{[\bar{P}'/\Gamma]}\mid ks\in\Pic^0_{[\bar{P}'/\Gamma]}\text{ for some $0\ne k\in\mathbb{Z}$}\}.$$ 
\begin{Cor}
We have an equality of sheaves $\Pic^{\tau}_{[\bar{P}'/\Gamma]}=\Pic'_{[\bar{P}'/\Gamma]}$.
\end{Cor}
\begin{proof}
Since $\Pic'_{[\bar{P}'/\Gamma]}$ is representable by $P'$ which has finitely many connected components, we have $\Pic'_{[\bar{P}'/\Gamma]}\subset\Pic^{\tau}_{[\bar{P}'/\Gamma]}$. Now suppose that there is a section $s$ of $\Pic^{\tau}_{[\bar{P}'/\Gamma]}$ that does not lie in $\Pic'_{[\bar{P}'/\Gamma]}$. Then by definition $p^{\prime\vee}(s)\notin\Pic^0_{\bar{J}'}$. But $\Pic^0_{\bar{J}'}=\Pic^{\tau}_{\bar{J}'}$, and so no multiple of $p^{\prime\vee}(s)$ lies in $\Pic^0_{\bar{J}'}$, which contradicts the assumption that a multiple of $s$ lies in $\Pic^0_{[\bar{P}'/\Gamma]}$.
\end{proof}

If $\bar{P}'=\bar{P}''$ (see \S \ref{subsec-prym} for the definitions of these compactified Prym varieties) and so the subgroup $G\subset\Gamma$ acts trivially on $\bar{P}'$, then Theorem \ref{Thm-Pic'-repre-P} is similar to the fact that the dual of $\mathbf{B}G$ is isomorphic to $G^{\vee}$. However, this theorem shows that even if $\bar{P}'=\bar{P}$ and so $[\bar{P}/\Gamma]$ is generically a scheme, its dual has nevertheless connected components parametrised by $G^{\vee}$.

\subsection{Compatibility of Poincar\'e Sheaves}\label{SS-Comp-Poinc-sh}\hfill

By the proof of Lemma \ref{Lem-barP'=barPG}, the morphism $\phi^{\wedge}:\bar{J}'\rightarrow\bar{J}$ restricts to a morphism $\bar{P}'\rightarrow\bar{P}$. Obviously, it is $\Gamma$-equivariant, and so descends to a morphism
\begin{equation}\label{eq-Def-iotaH}
\iota_H:[\bar{P}'/\Gamma]\longrightarrow[\bar{P}/\Gamma].
\end{equation}
Then the following diagram commutes
\begin{equation}\label{CD-phi-iota-p}
\begin{tikzcd}[row sep=2.5em, column sep=2em]
\bar{J}' \arrow[r, "\phi^{\wedge}"] \arrow[d, "p'"'] & \bar{J} \arrow[d, "p"]\\
\lbrack\bar{P}'/\Gamma\rbrack \arrow[r, "\iota_H"'] & \lbrack\bar{P}/\Gamma\rbrack.
\end{tikzcd}
\end{equation}

\begin{Prop}\label{Prop-phiV-ptil}
The following diagram commutes
\begin{equation}\label{Prop-CD-phiV-ptil}
\begin{tikzcd}[row sep=2.5em, column sep=2em]
\Pic^{\tau}_{\lbrack\bar{P}/\Gamma\rbrack} \arrow[r, "\sim", "\tilde{p}"'] \arrow[d, "\iota_H^{\vee}"'] & P \arrow[d, "\phi^{\vee}"]\\
\Pic^{\tau}_{\lbrack\bar{P}'/\Gamma\rbrack} \arrow[r, "\sim", "\tilde{p}'"'] & P',
\end{tikzcd}
\end{equation}
where $\phi^{\vee}$ is defined as in Lemma \ref{Lem-phiaV}, and the horizontal arrows are the isomorphisms given by Theorem \ref{Thm-Pic'-repre-P}.
\end{Prop}
\begin{proof}
Applying $\Pic^{\tau}(-)$ to the diagram (\ref{CD-phi-iota-p}) and noticing that $\Pic^{\tau}=\Pic^0$ for compactified Jacobians, we have
\begin{equation}\label{CD-phiV-ptil-1}
\begin{tikzcd}[row sep=2.5em, column sep=2em]
\Pic^{\tau}_{\lbrack\bar{P}/\Gamma\rbrack} \arrow[r, "\iota_H^{\vee}"] \arrow[d, "p^{\vee}"'] & \Pic^{\tau}_{\lbrack\bar{P}'/\Gamma\rbrack} \arrow[d, "p^{\prime\vee}"] \\
\Pic^0_{\bar{J}} \arrow[r, "(\phi^{\wedge})^{\vee}"'] & \Pic^0_{\bar{J}'}.
\end{tikzcd}
\end{equation}
By (\ref{CD-PrymAJ}) and the proof of Theorem \ref{Thm-Pic'-repre-P}, we have a commutative diagram
\begin{equation}\label{CD-phiV-ptil-2}
\begin{tikzcd}[row sep=2.5em, column sep=2em]
\Pic^{\tau}_{\lbrack\bar{P}'/\Gamma\rbrack} \arrow[r, "\sim", "\tilde{p}'"'] \arrow[d, "p^{\prime\vee}"'] & P' \arrow[d, hook] \\
\Pic^0_{\bar{J}'} \arrow[r, "\sim", "\bs\alpha^{\prime\vee}"'] & J'.
\end{tikzcd}
\end{equation}
We have a similar commutative diagram for $\bar{P}$ and $\bar{J}$, which together with (\ref{CD-phiV-ptil-2}) and (\ref{CD-phiV-ptil-1}) gives a commutative diagram
\begin{equation}\label{CD-phiV-ptil-3}
\begin{tikzcd}[row sep=2.5em, column sep=4em]
P \arrow[r, "\tilde{p}'\circ\iota_H^{\vee}\circ\tilde{p}^{-1}"] \arrow[d, hook] & P' \arrow[d, hook]\\
J \arrow[r, "\bs\alpha^{\prime\vee}\circ(\phi^{\wedge})^{\vee}\circ\bs\alpha^{\vee -1}"'] & J'.
\end{tikzcd}
\end{equation}
By Proposition \ref{Prop-phiWV-phiV}, we have $\bs\alpha^{\prime\vee}\circ(\phi^{\wedge})^{\vee}\circ\bs\alpha^{\vee -1}=\phi^{\vee}$. Therefore, the top arrow is the restriction of $\phi^{\vee}$; that is, we have the desired commutative diagram.
\end{proof}

Recall that $\mathcal{Q}$ and $\mathcal{Q}'$ are the Poincar\'e line bundles on $P\times[\bar{P}/\Gamma]$ and $P'\times[\bar{P}'/\Gamma]$ respectively.
\begin{Cor}\label{Cor-Comp-Poinc-sh}
We have an isomorphism of line bundles on $P\times[\bar{P}'/\Gamma]$:
$$
(\Id\times\iota_H)^{\ast}\mathcal{Q}\cong(\phi^{\vee}\times\Id)^{\ast}\mathcal{Q}'\otimes\pr_1^{\ast}M,
$$
where $M$ is some line bundle on $P$.
\end{Cor}
\begin{proof}
The sheaf $\mathcal{Q}$ defines the inverse of $\tilde{p}$ in (\ref{Prop-CD-phiV-ptil}). The pullback $(\Id\times\iota_H)^{\ast}\mathcal{Q}$ defines $\iota_H^{\vee}\circ\tilde{p}^{-1}$. By Proposition \ref{Prop-phiV-ptil}, we have $\iota_H^{\vee}\circ\tilde{p}^{-1}=\tilde{p}^{\prime -1}\circ\phi^{\vee}$. But $\tilde{p}^{\prime -1}\circ\phi^{\vee}$ is defined by the sheaf $(\phi^{\vee}\times\Id)^{\ast}\mathcal{Q}'$. Therefore, the two sheaves $(\Id\times\iota_H)^{\ast}\mathcal{Q}$ and $(\phi^{\vee}\times\Id)^{\ast}\mathcal{Q}'$ represent the same element of $\Pic^{\tau}_{[\bar{P}'/\Gamma]}(P)$, and so they differ by the pullback of a line bundle on $P$.
\end{proof}

\subsection{Twisted Sheaves}\label{subsec-TS}\hfill

\begin{Lem}\label{Lem-NmGa=Ga1V}
Regard $\Gamma^{\vee}$ as a subgroup of $\Pic^0(J_C)$. Then, we have a natural identification $\Nm_{\pi'}^{\vee}(\Gamma^{\vee})\cong\Gamma_0^{\vee}$. Moreover, under this identification, the norm map $\Nm_{\pi'}^{\vee}:\Gamma^{\vee}\rightarrow\Gamma_0^{\vee}$ is restricting a character of $\Gamma$ to $\Gamma_0$ composed with the inversion on $\Gamma_0$.
\end{Lem}
\begin{proof}
Recall that we have a $\Gamma_0$-covering $\bar{P}'_0\times J_C\rightarrow \bar{J}'$. We can therefore identify $\Gamma_0^{\vee}$ with the kernel of $\Pic^0_{\bar{J}'}\rightarrow\Pic^0_{\bar{P}'_0\times J_C}$. According to \cite[\S 5, Corollary 6]{M}, a line bundle $L$ on $\bar{P}'_0\times J_C$ is trivial if and only if 
\begin{itemize}
\item[(i)] for every $x\in J_C$, the restriction $L|_{\bar{P}'_0\times\{x\}}$ is trivial, and
\item[(ii)] the restriction $L_{\{0\}\times J_C}$ is trivial.
\end{itemize}
Note that for every $x\in J_C$, the image of $\bar{P}'_0\times\{x\}$ in $\bar{J}'$ is contained in $\Nm^{-1}_{\pi'}(nx)$. Therefore, a line bundle $L$ on $\bar{J}'$ pulls back to the trivial bundle on $\bar{P}'_0\times J_C$ if
\begin{itemize}
\item[(i')] for every $x\in J_C$, the restriction $L|_{\Nm_{\pi'}^{-1}(nx)}$ is trivial.
\item[(ii')] $(\pi^{\prime\vee})^{\ast}L\cong\mathcal{O}_{J_C}$.
\end{itemize}
Let us verify (i') and (ii') for $L=\Nm_{\pi'}^{\ast}M$ with $M\in\Pic^0_{J_C}[n]=\Gamma^{\vee}$. The pullback $\Nm_{\pi'}^{\ast}M$ is obviously trivial on each fibre $\Nm_{\pi'}^{-1}(nx)$. This verifies (i'). Since $(\pi^{\prime\vee})^{\vee}\circ\Nm_{\pi'}^{\vee}=[n]^{\vee}$, and $M$ is $n$-torsion, we have $(\pi^{\prime\vee})^{\ast}(\Nm_{\pi'}^{\ast}M)\cong\mathcal{O}_{J_C}$. This verifies (ii'). 

We have shown that $\Nm_{\pi'}^{\vee}(\Gamma^{\vee})\subset\Gamma_0^{\vee}$. By Remark \ref{Rem-piH-inj}, the morphism $\pi_H^{\vee}$ is injective, so the kernel of $\pi^{\prime\vee}:J_C\rightarrow\bar{J}'$ is equal to $H$. We see from the commutative diagram (\ref{CD-PicAJ-XaC}) that the cardinality of $\Nm_{\pi'}^{\vee}(\Gamma^{\vee})$ is equal to $|\Gamma|/|H|=|\Gamma_0|$; therefore, $\Nm_{\pi'}^{\vee}(\Gamma^{\vee})=\Gamma_0^{\vee}$.

To show the last statement, we apply Lemma \ref{Lem-G1G2} to the following commutative diagram
$$
\begin{tikzcd}[row sep=2.5em, column sep=2em]
\bar{P}'_0\times J_C \arrow[r, "\pr_2"] \arrow[d, "r"'] & J_C \arrow[d, "\lbrack n\rbrack"]\\
\bar{J}' \arrow[r, "\Nm_{\pi'}"'] & J_C.
\end{tikzcd}
$$
Note that $\pr_2$ is $\Gamma_0$-equivariant if $\gamma\in\Gamma_0$ acts on $\bar{P}'_0\times J_C$ as $\gamma^{-1}\times\gamma$, the inverse of the usual action.
\end{proof}

Recall that in Proposition \ref{Prop-equiv-barP-a'}, we have defined a $\Gamma_0$-equivariant structure $\{\Phi_{\gamma}\}_{\gamma\in\Gamma_0}$ on $\overline{\mathcal{P}}{}'$, regarded as a sheaf on $\bar{P}'_{\xi}\times\bar{J}'$ where $\Gamma_0$ acts on the first factor trivially. The Poincar\'e sheaf $\overline{\mathcal{Q}}{}'$ is obtained by first restricting $\overline{\mathcal{P}}{}'$ to $\bar{P}'_{\xi}\times\bar{P}'_0$ and then making it descend to the $\Gamma_0$-quotient. 

Since $H$ acts trivially on $\bar{P}'_0$, there is a natural morphism $[\bar{P}'_0/\Gamma_0]\rightarrow[\bar{P}'_0/(\Gamma_0/H)]$ defined by sending a $\Gamma_0$-torsor $\mathcal{E}$ on any test scheme $T$ to the extension of structure group $\mathcal{E}\times^{\Gamma_0}\Gamma_0/H$. There is a Cartesian diagram
$$
\begin{tikzcd}[row sep=2.5em, column sep=2em]
\lbrack\bar{P}'_0/H\rbrack \arrow[d] \arrow[r] & \lbrack\bar{P}'_0/\Gamma_0\rbrack \arrow[d]\\
\bar{P}'_0 \arrow[r] & \lbrack\bar{P}'_0/(\Gamma_0/H)\rbrack,
\end{tikzcd}
$$ 
where the bottom arrow is an \'etale atlas. Note that $[\bar{P}'_0/H]\cong\bar{P}'_0\times\mathbf{B}H$. In this situation, we say that $[\bar{P}'_0/\Gamma_0]$ is an $H$-gerbe over $[\bar{P}'_0/(\Gamma_0/H)]$, or simply an $H$-gerbe. If $F$ is a coherent sheaf on $[\bar{P}'_0/\Gamma_0]$, then it decomposes according to the action of $H$. If $F$ is equal to its $\xi$-isotypic component for some $\xi\in H^{\vee}$, then we say that $F$ is a $\xi$-twisted sheaf. This is equivalent to saying that the pullback of $F$ to $\bar{P}'_0$ has such an $H$-equivariant structure that $H$ acts on its sections via $\xi$. We are now ready to prove the following theorem.

\begin{Thm}\label{Thm-Twisted-Sheaf}
Let $\xi\in H^{\vee}$ and let $\bar{P}'_{\xi}$ be the corresponding connected component of $\bar{P}'$. Then, the Poincar\'e sheaf $\overline{\mathcal{Q}}{}'$ on the $H$-gerbe $\bar{P}'_{\xi}\times[\bar{P}'_0/\Gamma_0]$ is a $\xi^{-1}$-twisted sheaf.
\end{Thm}
\begin{proof}
By Lemma \ref{Lem-HomP=C}, for any $g\in H$, the isomorphism $\Phi_g:\overline{\mathcal{P}}{}'\isom\overline{\mathcal{P}}{}'$ is the multiplication by a scalar. We need to show that these isomorphisms define an action of $H$ on $\overline{\mathcal{P}}{}'$ via the character $\xi$. It suffices to show that for some $p\in\bar{P}'_{\xi}$, the restrictions of the isomorphisms $\Phi_g$ to $\{p\}\times\bar{J}'$ define an action of $H$ on $\overline{\mathcal{P}}{}'|_{\{p\}\times\bar{J}'}$ via $\xi$.

We have seen in the proof of the above lemma that the kernel of $\pi^{\prime\vee}:J_C\rightarrow\bar{J}'$ is equal to $H$. The commutative diagram (\ref{CD-PicAJ-XaC}) gives rise to a commutative diagram 
\begin{equation}
\begin{tikzcd}[row sep=2.5em, column sep=3em]
H \arrow[r, hook] \arrow[d, "\rotatebox{90}{\(\sim\)}"] & J_C \arrow[r, "\pi^{\prime\vee}"] \arrow[d, "\rotatebox{90}{\(\sim\)}", "\rho_C"'] & J' \arrow[d, "\rotatebox{90}{\(\sim\)}", "\rho'"']\\
K \arrow[r, hook] & \Pic^0_{J_C} \arrow[r, "\Nm^{\vee}_{\pi'}"'] & \Pic^0_{\bar{J}'}
\end{tikzcd}
\end{equation}
where $K$ is defined to be the kernel of $\Nm_{\pi'}^{\vee}$ and the middle vertical arrow restricts to $H\isom K$. This diagram restricts to the following commutative diagram
\begin{equation}\label{CD-GaGaG-GaVGa1V}
\begin{tikzcd}[row sep=2.5em, column sep=3em]
0 \arrow[r] & H \arrow[r] \arrow[d, "\rotatebox{90}{\(\sim\)}"] & \Gamma \arrow[r] \arrow[d, "\rotatebox{90}{\(\sim\)}", "\rho_C"'] & \Gamma_H \arrow[d, "\rotatebox{90}{\(\sim\)}", "\rho'"'] \arrow[r] & 0\\
0 \arrow[r] & K \arrow[r] & \Gamma^{\vee} \arrow[r, "\Nm_{\pi'}^{\vee}"'] & \Gamma_0^{\vee} \arrow[r] & 0
\end{tikzcd}
\end{equation}
where $\Gamma_H:=\Gamma/H$, and we have used Lemma \ref{Lem-NmGa=Ga1V}. 

Let $\bar{\gamma}\in\Gamma_{\xi}$ and so $\gamma:=\pi^{\prime\vee}(\bar{\gamma})$ lies in $\bar{P}'_{\xi}$ (the notation $\Gamma_{\xi}$ was introduced before Remark \ref{Rem-JCP=barJ-0}). We take $p=\gamma$. It follows from (\ref{CD-GaGaG-GaVGa1V}) that $\rho'(\Gamma_H)=\Gamma_0^{\vee}$. This allows us to regard $\rho'(\gamma)$ as a character $\kappa\in\Gamma_0^{\vee}$. Consequently, the line bundle $\rho'(\gamma)=\mathcal{P}'|_{\{\gamma\}\times\bar{J}'}$ is equivalent to $\mathcal{O}_{\bar{P}'_0\times J_C}$ equipped with a $\Gamma_0$-equivariant structure $\{\theta_g\}_{g\in\Gamma_0}$, so that $\Gamma_0$ acts on $H^0(\bar{P}'_0\times J_C,\mathcal{O}_{\bar{P}'_0\times J_C})$ via $\kappa$. By Remark \ref{Rem-gamma_chi} and Lemma \ref{Lem-NmGa=Ga1V}, we have $\Nm_{\pi'}^{\vee}(\rho_C(\gamma))|_{H}=\xi^{-1}$; therefore, $\kappa|_H=\xi^{-1}$. In other words, we have $\theta_g(1)=\xi(g)^{-1}1$, for any $g\in H$, where 1 is the constant function with value 1 on $\bar{P}'_0\times J_C$.

We need to determine the isomorphisms $\Phi_g$ on $\{p\}\times\bar{J}'$ from the isomorphisms $\theta_g$, for $g\in H$. Applying Lemma \ref{Lem-3-equiv} to $M=\mathcal{O}_{\bar{P}'_0\times J_C}$ and $L=\rho'(\gamma)$, we see that the equivariant structure $\{\Phi_g\}_{g\in H}$ on $L$ is equivalent to the equivariant structure $\{\zeta_{P,g}\}_{g\in H}$ on $M$, where $\zeta_{P,g}:M\rightarrow(g\times\Id)^{\ast}M$ are the isomorphisms defined by (\ref{Lem-3-equiv-eq2}). Now, the actions of $g$ on $\bar{P}'_0$ and on $\bar{J}'$ are both trivial so $(g\times\Id)^{\ast}M=M$ and $g^{\ast}L=L$. The second (canonical) isomorphism in (\ref{Lem-3-equiv-eq2}) becomes the identity morphism, and we have $\zeta_{P,g}=r^{\ast}\Phi_g$. However, $\zeta_{P,g}=(\Id\times g)^{\ast}\theta_{g}\circ\zeta_{J,g}$ by (\ref{Lem-3-equiv-eq01}). We have already seen that $\theta_g(1)=\xi(g)^{-1}1$. Recall that the equivariant structure $\{\Phi_{\gamma}\}_{\gamma\in\Gamma_0}$ on $L$ is defined in such a way that its pullback to $J_C$ consists of the canonical isomorphisms $\mathcal{O}_{J_C}\isom g^{\ast}\mathcal{O}_{J_C}$ (see the proof of Proposition \ref{Prop-equiv-barP-a'} and Remark \ref{Rem-equiv-barP-a'}). By definition, the restrictions of $\zeta_{J,g}$, $g\in H$, to $\{0\}\times J_C$ are exactly the pullbacks of $\{\Phi_{\gamma}\}$ to $J_C$. We deduce that $\zeta_{J,g}(1)=1$, and so $\zeta_{P,g}(1)=\xi(g)^{-1}1$, for any $g\in H$. We conclude that $\Phi_g$ is also the multiplication by $\xi(g)^{-1}$.
\end{proof}
\begin{Rem}
While $\Phi_g$ multiplies sections of $\rho'(\gamma)$ by $\xi(g)$, the isomorphism $\mathcal{O}_{J_C}\rightarrow g^{\ast}\mathcal{O}_{J_C}$ induced by $\Phi_g$ multiplies global sections simply by 1. This is not a contradiction. In fact, since $g$ acts trivially on $\bar{P}'_0$ and $\bar{J}'$, the second (canonical) isomorphism in (\ref{Lem-3-equiv-eq1}) coincides with $$\theta_{g^{-1}}:\mathcal{O}_{\bar{P}'_0\times J_C}\lisom(g^{-1}\times g)^{\ast}\mathcal{O}_{\bar{P}'_0\times J_C}=(\Id\times g)^{\ast}\mathcal{O}_{\bar{P}'_0\times J_C}.$$ And $\theta_{g^{-1}}$ acts as $\xi(g)$ on global sections, so we get the correct action of $H$.
\end{Rem}

\section{Fibrewise Fractional Hecke Eigenproperties}\label{FHC}

\subsection{Support of the mirror}\label{SS-Supp-Mirr}\hfill

As in \S \ref{subsec-prym}, we have a fixed integral spectral cover $\pi:X\rightarrow C$, which determines a subgroup $G\subset \Gamma$, and we have the associated compactified Prym variety $\bar{P}$ and its quotient $[\bar{P}/\Gamma]$. Since $\bar{P}$ is isomorphic to a Hitchin fibre, by \cite[Fact 2.5.1]{dC}, it is connected, and its irreducible components are in bijection with the connected components of $P$, which are in turn parametrised by $G^{\vee}$. According to Theorem \ref{M-T}, the action of $\Gamma$ on $\bar{P}$ is not free, as long as $G$ is nontrivial. A coherent sheaf $\bar{Q}$ on $\bar{P}\times[\bar{P}/\Gamma]$ has been defined in \S \ref{subsec-ESPS}. The Fourier-Mukai transform with kernel $\overline{\mathcal{Q}}$ is the functor
\begingroup
\allowdisplaybreaks
\begin{align*}
\mathscr{FM}_{\overline{\mathcal{Q}}}:\mathcal{D}([\bar{P}/\Gamma])&\longrightarrow\mathcal{D}(\bar{P})\\
K^{\bullet}&\longmapsto R\pr_{1\ast}(\overline{\mathcal{Q}}\otimes^L\pr_2^{\ast}K^{\bullet}),
\end{align*}
\endgroup
where $\pr_i$ is the projection from $\bar{P}\times[\bar{P}/\Gamma]$ to the $i$-th factor.

Let $x\in\bar{P}$ be a closed point. Suppose that $\Stab_{\Gamma}(x)=H$. Then there is a closed immersion $i_x:\mathbf{B}H\longrightarrow[\bar{P}/\Gamma]$ that fits into the following commutative diagram
\begin{equation}\label{CD-ix}
\begin{tikzcd}[row sep=2.5em, column sep=4em]
\Spec\mathbb{C} \arrow[r, "e"] \arrow[d, "x"'] & \mathbf{B}H \arrow[d, "i_x"] \\
\bar{P} \arrow[r, "r_P"'] & \lbrack\bar{P}/\Gamma\rbrack,
\end{tikzcd}
\end{equation}
where $e$ is the quotient by $H$. The line bundles on $\mathbf{B}H$ are identified with $H^{\vee}$. For any $\xi\in H^{\vee}$, we denote by $F_{\xi}$ the line bundle on $\mathbf{B}H$ with $H$ acting on its stalk via $\xi$. We will write $$\mathbf{A}_x(\xi)=\mathscr{FM}_{\overline{\mathcal{Q}}}(i_{x\ast}F_{\xi}).$$ We will denote by $\mathbf{A}_x=\overline{\mathcal{P}}|_{\bar{P}\times\{x\}}$ the torsion-free sheaf on $\bar{P}$ corresponding to $x$.
\begin{Lem}\label{Lem-Decomp}
There is an isomorphism of coherent sheaves on $\bar{P}$
$$
\mathbf{A}_x\cong\bigoplus_{\xi\in H^{\vee}}\mathbf{A}_x(\xi),
$$
and each summand $\mathbf{A}_x(\xi)$ is torsion-free and simple (in the sense that the endomorphism space is isomorphic to $\mathbb{C}$).
\end{Lem}
\begin{proof}
Write $F_R=\bigoplus_{\xi\in H^{\vee}}F_{\xi}$, where $R$ means the regular representation of $H$. The right hand side of the desired isomorphism is isomorphic to $\mathscr{FM}_{\overline{\mathcal{Q}}}(i_{x\ast}F_R)$. According to the commutative diagram (\ref{CD-ix}), we have $$i_{x\ast}F_R\cong i_{x\ast}e_{\ast}\underline{\mathbb{C}}\cong r_{P\ast}x_{\ast}\underline{\mathbb{C}}\cong r_{P\ast}\underline{\mathbb{C}}_x,$$ where $\underline{\mathbb{C}}_x$ is the skyscraper sheaf on $\bar{P}$ supported at $x$. Using the fact that $\overline{\mathcal{P}}\cong(\Id\times r_P)^{\ast}\overline{\mathcal{Q}}$ and the projection formula, we deduce that
$$
(\Id\times r_P)_{\ast}(\overline{\mathcal{P}}\otimes^L\pr_2^{\ast}\underline{\mathbb{C}}_x)\cong\overline{\mathcal{Q}}\otimes^L\pr_2^{\ast}i_{x\ast}F_R.
$$
Let $i_{\bar{P}}:\bar{P}\times\{x\}\rightarrow\bar{P}\times\bar{P}$ be the closed immersion. We apply $R\pr_{1\ast}$ to both sides of the above equation. Since $\pr_1=\pr_1\circ(\Id\times r_P)$ and $$\overline{\mathcal{P}}\otimes^L\pr_2^{\ast}\underline{\mathbb{C}}_x\cong i_{\bar{P}\ast}\overline{\mathcal{P}}|_{\bar{P}\times\{x\}},$$ we get $\overline{\mathcal{P}}|_{\bar{P}\times\{x\}}\cong\mathscr{FM}_{\overline{\mathcal{Q}}}(i_{x\ast}F_R)$, as desired. Since the Fourier-Mukai functor $\mathscr{FM}_{\overline{\mathcal{Q}}}$ is an equivalence of categories by \cite[Theorem A]{FHR}, each summand $\mathbf{A}_x(\xi)$ is simple. Finally, each direct summand $\mathbf{A}_x(\xi)$ is torsion-free, since $\mathbf{A}_x$ is torsion-free.
\end{proof}

For any $\chi\in G^{\vee}$, we denote by $\bar{P}_{\chi}^{irr}$ the corresponding irreducible component of $\bar{P}$.
\begin{Thm}\label{Thm-B-proof}
Let $H\subset G$ be a subgroup and let $\xi\in H^{\vee}$. Let $x\in\bar{P}$ be such that $\Stab_{\Gamma}(x)=H$. Then
$$
\Supp\mathbf{A}_x(\xi)=\bigcup_{\chi\in\Ind^G_H\xi}\bar{P}_{\chi}^{irr},
$$
where $\chi\in\Ind^G_H\xi$ means that $\chi$ is an irreducible component of the induced character $\Ind^G_H\xi$.
\end{Thm}
\begin{proof}
We will show that 
\begin{equation}\label{eq-Thm-A-proof}
P\cap\Supp\mathbf{A}_x(\xi)=\bigsqcup_{\chi\in\Ind^G_H\xi}P_{\chi}.
\end{equation}
Since $\mathbf{A}_x(\xi)$ is a torsion-free sheaf on $\bar{P}$, its support can not have irreducible components of lower dimension, and so (\ref{eq-Thm-A-proof}) implies that its support is the union of the irreducible components $\bar{P}_{\chi}^{irr}$. 

We have
\begingroup
\allowdisplaybreaks
\begin{align*}
&(\overline{\mathcal{Q}}\otimes^L\pr_2^{\ast}i_{x\ast}F_{\xi})|_{P\times[\bar{P}/\Gamma]}\\
\cong{}&\mathcal{Q}\otimes\pr_2^{\ast}i_{x\ast}F_{\xi}\\
\cong{}&(\Id\times i_x)_{\ast}((\Id\times i_x)^{\ast}\mathcal{Q}\otimes\pr_2^{\ast}F_{\xi}).
\end{align*}
\endgroup
Note that $i_x$ factors through $\iota_x:\mathbf{B}H\longrightarrow[\bar{P}^H/\Gamma]$. By Lemma \ref{Lem-barP'=barPG}, we can identify $[\bar{P}'/\Gamma]$ with $[\bar{P}^H/\Gamma]$, and the morphism $\iota_H$ defined in (\ref{eq-Def-iotaH}) is identified with the closed immersion $[\bar{P}^H/\Gamma]\rightarrow[\bar{P}/\Gamma]$. Now 
$$
(\Id\times i_x)^{\ast}\mathcal{Q}\cong(\Id\times\iota_x)^{\ast}(\Id\times\iota_H)^{\ast}\mathcal{Q}\cong(\Id\times\iota_x)^{\ast}\big((\phi^{\vee}\times\Id)^{\ast}\mathcal{Q}'\otimes\pr_1^{\ast}M\big)
$$
for some line bundle $M$ on $P$ according to Corollary \ref{Cor-Comp-Poinc-sh}, where $\phi^{\vee}:P\rightarrow P'$ is as in Lemma \ref{Lem-phiaV}. By Theorem \ref{Thm-Twisted-Sheaf}, and Proposition \ref{Prop-IndGH}, the restriction of $\mathcal{Q}'$ to $P'_{\chi}\times[\bar{P}'/\Gamma]$ is a $\Res^G_H\chi^{-1}$-twisted sheaf. We deduce that the restriction of $(\Id\times i_x)^{\ast}\mathcal{Q}$ to $P_{\chi}\times\mathbf{B}H$ is a $\Res^G_H\chi^{-1}$-twisted sheaf. Therefore $(\Id\times i_x)^{\ast}\mathcal{Q}\otimes\pr_2^{\ast}F_{\xi}$ is $0$-twisted exactly on the connected components $P_{\chi}\times\mathbf{B}H$ with $\chi\in\Ind^G_H\xi$. Now the functor $\pr_{1\ast}=R\pr_{1\ast}=R\pr_{1\ast}\circ(\Id\times i_x)_{\ast}$ takes a sheaf on $P_{\chi}\times\mathbf{B}H$ to its $0$-twisted component, we see that
$$
P_{\chi}\subset\Supp\mathscr{FM}_{\overline{\mathcal{Q}}}(i_{x\ast}F_{\xi})
$$
if and only if $\chi\in\Ind^G_H\xi$. This proves (\ref{eq-Thm-A-proof}).
\end{proof}

\subsection{Translation Operators and Eigensheaves}\label{subsec-TOE}\hfill

Recall that $\pi:X\rightarrow C$ is an integral spectral cover and $\bar{P}=\bar{P}_{\pi}$. Let $p_1$ and $p_2$ be smooth points of of $X$ such that $\pi(p_1)=\pi(p_2)=q\in C$. Then $\mathcal{O}_{X}(p_1-p_2)\in P$. Define the translation map
\begingroup
\allowdisplaybreaks
\begin{align*}
T_{p_1,p_2}:\bar{P}&\longrightarrow \bar{P},\\
F&\longmapsto F\otimes\mathcal{O}_{X}(p_1-p_2).
\end{align*}
\endgroup
Since $p_1$ and $p_2$ are smooth points, there are unique points $\tilde{p}_1$ and $\tilde{p}_2$ of the normalisation $\tilde{X}$ such that $\tilde{p}_i$ lies over $p_i$ for $i=1$, $2$. Let $\tilde{\pi}:\tilde{X}\rightarrow C$ be as in \S \ref{SS-CCPV} and $G=\Ker\tilde{\pi}^{\vee}$. Tensoring by $\mathcal{O}_{\tilde{X}}(\tilde{p}_1-\tilde{p}_2)$ defines a translation map $T_{\tilde{p}_1,\tilde{p}_2}$ on $\bar{P}_{\tilde{\pi}}$. Let $q_1$ and $q_2$ be the images of $\tilde{p}_1$ and $\tilde{p}_2$ respectively under the map $\tilde{\pi}_G:\tilde{X}\rightarrow C''$, where $\tilde{\pi}_G$ is as in the proof of Proposition \ref{Prop-IndGH}. Then $\psi_G(q_1)=\psi_G(q_2)=q$ and so $\mathcal{O}_{C''}(q_1-q_2)\in P_{C''}$. Tensoring by $\mathcal{O}_{C''}(q_1-q_2)$ defines a translation map $T_{q_1,q_2}$ on $P_{C''}$. There exists a unique $\chi_{p_1,p_2}\in G^{\vee}$, regarded as a covering transformation, such that $q_2=\chi_{p_1,p_2}(q_1)$.

\begin{Prop}\label{Prop-T-pi0}
The restriction of $T_{p_1,p_2}$ to $P$ induces the multiplication by $\chi_{p_1,p_2}^{-1}$ on $\pi_0(P)$.
\end{Prop}
\begin{proof}
We have a commutative diagram
\begin{equation}
\begin{tikzcd}[row sep=2.5em, column sep=4em]
\pi_0(P) \arrow[r, "\sim"] \arrow[d, "T_{p_1,p_2}"'] & \pi_0(P_{\tilde{\pi}}) \arrow[r, "\sim"] \arrow[d, "T_{\tilde{p}_1,\tilde{p}_2}"'] & \pi_0(P_{C''}) \arrow[d, "T_{q_1,q_2}"]\\
\pi_0(P) \arrow[r, "\sim"] & \pi_0(P_{\tilde{\pi}}) \arrow[r, "\sim"] & \pi_0(P_{C''})
\end{tikzcd}
\end{equation}
where the isomorphism $\pi_0(P)\cong\pi_0(P_{\tilde{\pi}})$ is induced by the pullback along the normalisation map $\tilde{X}\rightarrow X$. (See \cite[Lemma 4.1 (3), (4)]{HP}.) We have seen in the proof of Proposition \ref{Prop-IndGH} that $\Nm_{\tilde{\pi}_G}$ induces an isomorphism $\pi_0(P_{\tilde{\pi}})\cong\pi_0(P_{C''})$. The left hand side of the diagram commutes because $\mathcal{O}_{X}(p_1-p_2)$ pulls back to $\mathcal{O}_{\tilde{X}}(\tilde{p}_1-\tilde{p}_2)$. The right hand side of the diagram commutes because $\Nm_{\tilde{\pi}_G}(\mathcal{O}_{\tilde{X}}(\tilde{p}_1-\tilde{p}_2))=\mathcal{O}_{C''}(q_1-q_2)$. Now $\mathcal{O}_{C''}(q_1-q_2)$ lies in the image of $f^{(-1)}_{\chi_{p_1,p_2}^{-1}}$ by the definition of $\chi_{p_1,p_2}$. It follows from Theorem \ref{Thm-Conn-Prym-New} that the bijection from $\pi_0(P_{C''})$ to itself induced by $T_{q_1,q_2}$ is the multiplication by $\chi_{p_1,p_2}^{-1}$.
\end{proof} 

\begin{Thm}\label{Thm-FHEP}
Let $p_1$ and $p_2$ be two smooth points of $X$ lying in a fibre of the spectral cover $\pi:X\rightarrow C$, and let $\chi=\chi_{p_1,p_2}$ be the corresponding character of $G$ as in Proposition \ref{Prop-T-pi0}. Then for any $x\in\bar{P}$, we have the following isomorphisms of torsion-free sheaves on $\bar{P}$.
\begin{itemize}
\item[(i)] Hecke eigenproperty.
$$
T_{p_1,p_2}^{\ast}\mathbf{A}_x\cong\mathbf{A}_x.
$$
\item[(ii)] Fractional Hecke Eigenproperty.
$$
T^{\ast}_{p_1,p_2}\mathbf{A}_x(\xi)\cong\mathbf{A}_x(\xi\otimes\Res^G_H\chi)
$$
for any $\xi\in H^{\vee}$, where $H=\Stab_{\Gamma}(x)$.
\end{itemize}
\end{Thm}
\begin{proof}
By \cite[Lemma 6.5]{Ari2}, we have
$$
(T_{p_1,p_2}\times\Id)^{\ast}\overline{\mathcal{P}}\cong\overline{\mathcal{P}}\otimes\pr_2^{\ast}M
$$
where $M=\overline{\mathcal{P}}|_{\{\mathcal{O}_X(p_1-p_2)\}\times\bar{P}}$ is a line bundle. Restricting this isomorphism to $\bar{P}\times\{x\}$ gives (i). Now by Lemma \ref{Lem-Decomp} we have $$T_{p_1,p_2}^{\ast}\mathbf{A}_x\cong\bigoplus_{\xi\in H^{\vee}}T_{p_1,p_2}^{\ast}\mathbf{A}_x(\xi).$$ By Theorem \ref{Thm-B-proof}, the torsion-free sheaf $\mathbf{A}_x(\xi)$ is supported on the union of the irreducible components $\bar{P}_{\chi'}^{irr}$ with $\Res^G_H\chi'=\xi$, and so by Proposition \ref{Prop-T-pi0} the pullback $T_{p_1,p_2}^{\ast}\mathbf{A}_x(\xi)$ is supported on the union of the irreducible components $\bar{P}_{\chi'}$ with $\Res^G_H\chi'=\xi\otimes\Res^G_H\chi$. Since $\mathbf{A}_x(\xi)$ is an indecomposable coherent sheaf for each $\xi\in H^{\vee}$, we deduce that $T_{p_1,p_2}^{\ast}\mathbf{A}_x(\xi)\cong\mathbf{A}_x(\xi')$ for some $\xi'\in H^{\vee}$ by Krull-Schmidt theorem and (i). We see that if $\xi'=\xi\otimes\Res^G_H\chi$, then $\mathbf{A}_x(\xi')$ and $T_{p_1,p_2}^{\ast}\mathbf{A}_x(\xi)$ have the same support; otherwise, their supports have no common irreducible components. The only possibility is  $T_{p_1,p_2}^{\ast}\mathbf{A}_x(\xi)\cong\mathbf{A}_x(\xi\otimes\Res^G_H\chi)$. This completes the proof.
\end{proof}

\addtocontents{toc}{\protect\setcounter{tocdepth}{-1}}
\bibliographystyle{alpha}
\bibliography{BIB}

\begin{thebibliography}{FGOPN21}

\bibitem[AIK77]{AIK}
Allen~B. Altman, Anthony Iarrobino, and Steven~L. Kleiman.
\newblock Irreducibility of the compactified {J}acobian.
\newblock In {\em Real and complex singularities ({P}roc. {N}inth {N}ordic
  {S}ummer {S}chool/{NAVF} {S}ympos. {M}ath., {O}slo, 1976)}, pages 1--12.
  Sijthoff and Noordhoff, Alphen aan den Rijn, 1977.

\bibitem[AK79]{AK2}
Allen~B. Altman and Steven~L. Kleiman.
\newblock Compactifying the {P}icard scheme. {II}.
\newblock {\em Amer. J. Math.}, 101(1):10--41, 1979.

\bibitem[Ari11]{Ari1}
D.~Arinkin.
\newblock Cohomology of line bundles on compactified {J}acobians.
\newblock {\em Math. Res. Lett.}, 18(6):1215--1226, 2011.

\bibitem[Ari13]{Ari2}
Dima Arinkin.
\newblock Autoduality of compactified {J}acobians for curves with plane
  singularities.
\newblock {\em J. Algebraic Geom.}, 22(2):363--388, 2013.

\bibitem[BL04]{BL}
Christina Birkenhake and Herbert Lange.
\newblock {\em Complex abelian varieties}, volume 302 of {\em Grundlehren der
  mathematischen Wissenschaften [Fundamental Principles of Mathematical
  Sciences]}.
\newblock Springer-Verlag, Berlin, second edition, 2004.

\bibitem[BLR90]{BLR}
Siegfried Bosch, Werner L\"{u}tkebohmert, and Michel Raynaud.
\newblock {\em N\'{e}ron models}, volume~21 of {\em Ergebnisse der Mathematik
  und ihrer Grenzgebiete (3) [Results in Mathematics and Related Areas (3)]}.
\newblock Springer-Verlag, Berlin, 1990.

\bibitem[BNR89]{BNR}
Arnaud Beauville, M.~S. Narasimhan, and S.~Ramanan.
\newblock Spectral curves and the generalised theta divisor.
\newblock {\em J. Reine Angew. Math.}, 398:169--179, 1989.

\bibitem[Bro09]{B}
Sylvain Brochard.
\newblock Foncteur de {P}icard d'un champ alg\'{e}brique.
\newblock {\em Math. Ann.}, 343(3):541--602, 2009.

\bibitem[Car22]{Car}
Raffaele~Marco Carbone.
\newblock The direct image of generalized divisors and the {N}orm map between
  compactified {J}acobians.
\newblock {\em Geom. Dedicata}, 216(1):Paper No. 14, 40, 2022.

\bibitem[dC17]{dC}
Mark~Andrea de~Cataldo.
\newblock A support theorem for the {H}itchin fibration: the case of {${\rm
  SL}_n$}.
\newblock {\em Compos. Math.}, 153(6):1316--1347, 2017.

\bibitem[DP12]{DP}
R.~Donagi and T.~Pantev.
\newblock Langlands duality for {H}itchin systems.
\newblock {\em Invent. Math.}, 189(3):653--735, 2012.

\bibitem[EGK02]{EGK}
Eduardo Esteves, Mathieu Gagn\'{e}, and Steven Kleiman.
\newblock Autoduality of the compactified {J}acobian.
\newblock {\em J. London Math. Soc. (2)}, 65(3):591--610, 2002.

\bibitem[EK05]{EK}
Eduardo Esteves and Steven Kleiman.
\newblock The compactified {P}icard scheme of the compactified {J}acobian.
\newblock {\em Adv. Math.}, 198(2):484--503, 2005.

\bibitem[Est01]{E}
Eduardo Esteves.
\newblock Compactifying the relative {J}acobian over families of reduced
  curves.
\newblock {\em Trans. Amer. Math. Soc.}, 353(8):3045--3095, 2001.

\bibitem[EvdGM]{EGM}
Bas Edixhoven, Gerard van~der Geer, and Ben Moonen.
\newblock Abelian varieties.
\newblock Available at \url{http://van-der-geer.nl/~gerard/AV.pdf}.

\bibitem[FGOPN21]{FGOP}
Emilio Franco, Peter~B. Gothen, Andr\'{e} Oliveira, and Ana Pe\'{o}n-Nieto.
\newblock Unramified covers and branes on the {H}itchin system.
\newblock {\em Adv. Math.}, 377:Paper No. 107493, 61, 2021.

\bibitem[FHR22]{FHR}
Emilio Franco, Robert Hanson, and Jo\~ao Ruano.
\newblock Fourier-mukai transform for fine compactified prym varieties.
\newblock 2022.

\bibitem[FW08]{FW}
Edward Frenkel and Edward Witten.
\newblock Geometric endoscopy and mirror symmetry.
\newblock {\em Commun. Number Theory Phys.}, 2(1):113--283, 2008.

\bibitem[Gro61]{EGAII}
A.~Grothendieck.
\newblock \'{E}l\'{e}ments de g\'{e}om\'{e}trie alg\'{e}brique. {II}. \'{E}tude
  globale \'{e}l\'{e}mentaire de quelques classes de morphismes.
\newblock {\em Inst. Hautes \'{E}tudes Sci. Publ. Math.}, (8):222, 1961.

\bibitem[Gro62]{FGA}
Alexander Grothendieck.
\newblock {\em Fondements de la g\'{e}om\'{e}trie alg\'{e}brique. [{E}xtraits
  du {S}\'{e}minaire {B}ourbaki, 1957--1962.]}.
\newblock Secr\'{e}tariat math\'{e}matique, Paris, 1962.

\bibitem[GS22]{GS}
Michael Groechenig and Shiyu Shen.
\newblock Complex k-theory of moduli spaces of higgs bundles, 2022.

\bibitem[HH22]{HH}
Tam\'{a}s Hausel and Nigel Hitchin.
\newblock Very stable {H}iggs bundles, equivariant multiplicity and mirror
  symmetry.
\newblock {\em Invent. Math.}, 228(2):893--989, 2022.

\bibitem[Hor22]{Ho}
Johannes Horn.
\newblock Semi-abelian spectral data for singular fibres of the
  {$\text{SL}(2,\mathbb{C})$}-{H}itchin system.
\newblock {\em Int. Math. Res. Not. IMRN}, (5):3860--3917, 2022.

\bibitem[HP12]{HP}
Tam\'{a}s Hausel and Christian Pauly.
\newblock Prym varieties of spectral covers.
\newblock {\em Geom. Topol.}, 16(3):1609--1638, 2012.

\bibitem[HT03]{HT}
Tam\'{a}s Hausel and Michael Thaddeus.
\newblock Mirror symmetry, {L}anglands duality, and the {H}itchin system.
\newblock {\em Invent. Math.}, 153(1):197--229, 2003.

\bibitem[Ill71]{ILL}
L.~Illusie.
\newblock Conditions de finitude relatives.
\newblock In {\em Th{\'e}orie des Intersections et Th{\'e}or{\`e}me de
  Riemann-Roch}, pages 222--273, Berlin, Heidelberg, 1971. Springer Berlin
  Heidelberg.

\bibitem[Li22]{Li}
Mao Li.
\newblock The {P}oincar\'{e} line bundle and autoduality of {H}itchin fibers.
\newblock {\em Selecta Math. (N.S.)}, 28(3):Paper No. 49, 52, 2022.

\bibitem[MRV19]{MRV2}
Margarida Melo, Antonio Rapagnetta, and Filippo Viviani.
\newblock Fourier-{M}ukai and autoduality for compactified {J}acobians, {II}.
\newblock {\em Geom. Topol.}, 23(5):2335--2395, 2019.

\bibitem[MS22]{MS2}
Davesh Maulik and Junliang Shen.
\newblock On the intersection cohomology of the moduli of {${\rm
  SL}_n$}-{H}iggs bundles on a curve.
\newblock {\em J. Topol.}, 15(3):1034--1057, 2022.

\bibitem[Muk81]{Muk}
Shigeru Mukai.
\newblock Duality between {$D(X)$} and {$D(\hat X)$} with its application to
  {P}icard sheaves.
\newblock {\em Nagoya Math. J.}, 81:153--175, 1981.

\bibitem[Mum08]{M}
David Mumford.
\newblock {\em Abelian varieties}, volume~5 of {\em Tata Institute of
  Fundamental Research Studies in Mathematics}.
\newblock Published for the Tata Institute of Fundamental Research, Bombay; by
  Hindustan Book Agency, New Delhi, 2008.
\newblock With appendices by C. P. Ramanujam and Yuri Manin, Corrected reprint
  of the second (1974) edition.

\bibitem[Ngo10]{Ngo}
Bao~Chau Ngo.
\newblock Le lemme fondamental pour les alg\`ebres de {L}ie.
\newblock {\em Publ. Math. Inst. Hautes \'{E}tudes Sci.}, (111):1--169, 2010.

\bibitem[NR75]{NR}
M.~S. Narasimhan and S.~Ramanan.
\newblock Generalised {P}rym varieties as fixed points.
\newblock {\em J. Indian Math. Soc. (N.S.)}, 39:1--19 (1976), 1975.

\bibitem[{Sta}18]{S-P}
The {Stacks Project Authors}.
\newblock \textit{Stacks Project}.
\newblock \url{https://stacks.math.columbia.edu}, 2018.

\end{thebibliography}
\end{document}